\documentclass[smallextended,nospthms,envcountsect]{svjour3}

\smartqed 

\usepackage{graphicx}
\usepackage{mathptmx}
\usepackage{amssymb}
\usepackage{amsmath}
\usepackage[matrix,arrow,curve,cmtip]{xy}
\entrymodifiers={+!!<0pt,\fontdimen22\textfont2>} 
\usepackage{mathrsfs}
\usepackage{amsthm}
\usepackage{units}
\usepackage{enumitem}
\setenumerate[1]{leftmargin=2.3em}

\numberwithin{equation}{section}

\theoremstyle{plain}   
\newtheorem{bigthm}{Theorem}   

\newtheorem{theorem}[equation]{Theorem}  
\newtheorem{cor}[equation]{Corollary}     
\newtheorem{lemma}[equation]{Lemma}         
\newtheorem{prop}[equation]{Proposition} 
\newtheorem{addendum}[equation]{Addendum}

\theoremstyle{definition}
\newtheorem{definition}[equation]{Definition}

\theoremstyle{remark}
\newtheorem{remark}[equation]{Remark}
\newtheorem{example}[equation]{Example}

\newcommand{\Hom}{\operatorname{Hom}}
\newcommand{\Spec}{\operatorname{Spec}}

\newcommand{\TR}{\operatorname{TR}}

\newcommand{\N}{\mathbb{N}}
\newcommand{\Z}{\mathbb{Z}}
\newcommand{\R}{\mathbb{R}}
\newcommand{\F}{\mathbb{F}}
\newcommand{\Fp}{\mathbb{F}_p}

\newcommand{\id}{\operatorname{id}}
\newcommand{\pr}{\operatorname{pr}}

\begin{document}

\title{The big de~Rham-Witt complex
\thanks{Generous assistance from DNRF Niels Bohr Professorship, JSPS
  Grant-in-Aid 23340016, and CMI Senior Scholarship is gratefully acknowledged}
}

\titlerunning{The big de~Rham-Witt complex}

\author{Lars Hesselholt}

\institute{Lars Hesselholt \at
Nagoya University, Nagoya, Japan and University of Copenhagen, Denmark\\
\email{larsh@math.nagoya-u.ac.jp and larsh@math.ku.dk}
}

\date{}

\maketitle


\section*{Introduction}

The big de~Rham-Witt complex was introduced by the author and Madsen
in~\cite{hm2} with the purpose of giving an algebraic description of
the equivariant homotopy groups in low degrees of B\"{o}kstedt's
topological Hochschild spectrum of a commutative ring. This functorial
algebraic description, in turn, is essential for understading
algebraic $K$-theory by means of the cyclotomic trace map of
B\"{o}kstedt-Hsiang-Madsen~\cite{bokstedthsiangmadsen};
compare~\cite{hm4,h2,gh2}. The original construction, which relied on
the adjoint functor theorem, was very indirect and a direct
construction has been lacking. In this paper, we give a new and
explicit construction of the big de~Rham-Witt complex and we also
correct the $2$-torsion which was not quite correct in the original
construction. 

The new construction is based on a theory, which is
developed first, of modules and derivations over a $\lambda$-ring. The
main result of this first part of the paper is that the universal
derivation of a $\lambda$-ring is given by the universal derivation of
the underlying ring together with an additional structure that depends
directly on the $\lambda$-ring structure in question. In the case of
the universal $\lambda$-ring, which is given by the ring of big Witt
vectors, this additional structure consists in divided Frobenius
operators on the module of K\"{a}hler differentials. It is the
existence of these divided Frobenius operators that makes the new
direct construction of the big de~Rham-Witt complex possible. This is
carried out in the second part of the paper, where we also show that
the big de~Rham-Witt complex behaves well with respect to \'{e}tale
morphisms. Finally, we explicitly evaluate the big de~Rham-Witt
complex of the ring of integers.

In more detail, let $A$ be a ring, which we always assume to be
commutative and unital. The ring $\mathbb{W}(A)$ of big Witt vectors
in $A$ is equipped with a natural action through ring homomorphisms by
the multiplicative monoid $\N$ of positive integers, where the action
by $n \in \N$ is given by the $n$th Frobenius map
$$\xymatrix{
{ \mathbb{W}(A) } \ar[r]^-{F_n} &
{ \mathbb{W}(A). } \cr
}$$
The Frobenius maps give rise to a natural ring homomorphism
$$\xymatrix{
{ \mathbb{W}(A) } \ar[r]^-{\Delta} &
{ \mathbb{W}(\mathbb{W}(A)) } \cr
}$$
whose Witt components $\Delta_e \colon \mathbb{W}(A) \to \mathbb{W}(A)$
are characterized by the formula
$$F_n(a) = \sum_{e \mid n} e\Delta_e(a)^{n/e}.$$
The triple $(\mathbb{W}(-),\Delta,\epsilon)$ with $\epsilon \colon
\mathbb{W}(A) \to A$ the first Witt component is a comonad on the
category of rings and a $\lambda$-ring in the sense of
Grothendieck~\cite{grothendieck1} is precisely a coalgebra
$(A,\lambda_A)$ over this comonad. 

Recently, Borger~\cite{borger2} has proposed that a
$\lambda$-ring structure $\lambda_A \colon A \to \mathbb{W}(A)$ on a
ring $A$ be considered as descent data from $\Z$-algebras to algebras
over a deeper base $\F_1$. This begs the question as to the natural notions
of modules and derivations over $\lambda$-rings. We show here that the
general approach of Beck~\cite{beck} leads to the following
answer. First, if $(A,\lambda_A)$ is a $\lambda$-ring, then the ring $A$ is
equipped with an action by the multiplicative monoid $\N$ through ring
homomorphisms, where the action by $n \in \N$ is given by the $n$th
associated Adams operation
$$\xymatrix{
{ A } \ar[r]^-{\psi_{A,n}} &
{ A } \cr
}$$
defined by the formula
$$\psi_{A,n}(a) = \sum_{e \mid n} e\lambda_{A,e}(a)^{n/e}.$$
Here $\lambda_{A,e} \colon A \to A$ is the $e$th Witt component of
$\lambda_A \colon A \to \mathbb{W}(A)$. Now, the category of
$(A,\lambda_A)$-modules is identified with the category of left
modules over the twisted monoid algebra $A^{\psi}[\N]$ with the
product defined by the formula 
$$n \cdot a = \psi_{A,n}(a) \cdot n.$$
Hence, an $(A,\lambda_A)$-module is a pair $(M,\lambda_M)$ that
consists of an $A$-module $M$ and an $\N$-indexed family of maps
$\lambda_{M,n} \colon M \to M$ such that $\lambda_{M,n}$ is
$\psi_{A,n}$-linear, $\lambda_{M,1} = \id_M$, and
$\lambda_{M,m}\lambda_{M,n} = \lambda_{M,mn}$. Moreover, we identify
the derivations
$$\xymatrix{ 
{ (A,\lambda_A) } \ar[r]^-{D} &
{ (M,\lambda_M) } \cr
}$$
with the derivations $D \colon A \to M$ that satisfy the identities
$$\lambda_{M,n}(Da) = \sum_{e \mid n}
\lambda_{A,e}(a)^{(n/e)-1}D\lambda_{A,e}(a).$$
It is now easy to show that there is a universal derivation
$$\xymatrix{
{ (A,\lambda_A) } \ar[r]^-{d} &
{ (\Omega_{(A,\lambda_A)}^1,\lambda_{\Omega_{(A,\lambda_A)}^1}). } \cr
}$$
We prove the following result.

\begin{bigthm}\label{universalderivations}For every $\lambda$-ring
$(A,\lambda_A)$, the canonical map
$$\xymatrix{
{ \Omega_A^1 } \ar[r] &
{ {\Omega_{(A,\lambda_A)}^1} } \cr
}$$
is an isomorphism of $A$-modules.
\end{bigthm}

It follows that for a $\lambda$-ring $(A,\lambda_A)$, the $A$-module
of differentials $\Omega_A^1$ carries the richer structure of
an $(A,\lambda_A)$-module. In the case of $(\mathbb{W}(A),\Delta_A)$,
this implies that there are natural $F_n$-linear maps $F_n \colon
\Omega_{\mathbb{W}(A)}^1 \to \Omega_{\mathbb{W}(A)}^1$ defined by
$$F_n(da) = \sum_{e \mid n} \Delta_e(a)^{(n/e)-1}d\Delta_e(a)$$
such that $F_1 = \id$, $F_mF_n = F_{mn}$, $dF_n(a) = nF_n(da)$, and
$F_n(d[a]) = [a]^{n-1}d[a]$. The $p$-typical analog of $F_p$ was also
constructed by Borger and Wieland in~\cite[12.8]{borgerwieland}.

The construction of the de~Rham-Witt complex begins with the following
variant of the de~Rham complex. The ring $\mathbb{W}(\Z)$ contains
exactly the four units $\pm[\pm 1]$, all of which are square roots of
$[1]$, and the $2$-torsion element
$$d\log[-1] = [-1]^{-1}d[-1] = [-1]d[-1] \in \Omega_{\mathbb{W}(A)}^1$$
plays a special r\^{o}le. We define the graded $\mathbb{W}(A)$-algebra
$$\hat{\Omega}_{\mathbb{W}(A)}^{\boldsymbol{\cdot}} =
T_{\mathbb{W}(A)}^{\boldsymbol{\cdot}}\Omega_{\mathbb{W}(A)}^1 / J^{\boldsymbol{\cdot}}$$
to be the quotient of the tensor algebra of the $\mathbb{W}(A)$-module
$\smash{ \Omega_{\mathbb{W}(A)}^1 }$ by the graded ideal $J^{\boldsymbol{\cdot}}$
generated by all elements of the form
$$da \otimes da - d\log[-1] \otimes F_2da$$
with $a \in \mathbb{W}(A)$. It is an anticommutative graded ring which
carries a unique graded derivation $d$ that extends $d \colon
\mathbb{W}(A) \to \Omega_{\mathbb{W}(A)}^1$ and satisfies
$$dd\omega = d\log[-1] \cdot d\omega.$$
Moreover, the maps $F_n \colon \mathbb{W}(A) \to \mathbb{W}(A)$ and
$F_n \colon \Omega_{\mathbb{W}(A)}^1 \to \Omega_{\mathbb{W}(A)}^1$ extend
uniquely to a map of graded rings $F_n \colon
\Omega_{\mathbb{W}(A)}^{\boldsymbol{\cdot}} \to
\Omega_{\mathbb{W}(A)}^{\boldsymbol{\cdot}}$ which satisfies $dF_n =
nF_nd$. Next, we show that the maps $d$ and $F_n$ both descend to the
further quotient
$$\check{\Omega}_{\mathbb{W}(A)}^{\boldsymbol{\cdot}}
= \hat{\Omega}_{\mathbb{W}(A)}^{\boldsymbol{\cdot}} /
K^{\boldsymbol{\cdot}}$$
by the graded ideal generated by all elements of the form
$$F_pdV_p(a) - da - (p-1)d\log[-1] \cdot a$$
with $p$ a prime number and $a \in \mathbb{W}(A)$. We now recall the
Verschiebung maps
$$V_n \colon \mathbb{W}(A) \to \mathbb{W}(A)$$
which are additive and satisfy the projection formula
$$a V_n(b) = V_n(F_n(a)b).$$
These maps, however, do not extend to $\smash{
  \hat{\Omega}_{\mathbb{W}(A)}^{\boldsymbol{\cdot}} }$ or 
$\smash{ \check{\Omega}_{\mathbb{W}(A)}^{\boldsymbol{\cdot}} }$, and
the de~Rham-Witt complex, roughly speaking, is the largest quotient
$$\xymatrix{
{ \check{\Omega}_{\mathbb{W}(A)}^{\boldsymbol{\cdot}} } \ar[r]^{\eta} &
{ \mathbb{W}\,\Omega_A^{\boldsymbol{\cdot}} } \cr
}$$
such that the Verschiebung maps extend to
$\mathbb{W}\,\Omega_A^{\boldsymbol{\cdot}}$ and such that the extended
maps $F_n$ and $V_n$ satisfy the projection formula. The precise
definition given in Section~\ref{drwsection} below is by recursion
with respect to the quotients $\mathbb{W}_S(A)$ of $\mathbb{W}(A)$
where $S$ ranges over the finite subsets $S \subset \N$ that are
stable under division. We further prove the following result to the
effect that the de~Rham-Witt complex may be characterized as the
universal example of an algebraic structure called a Witt complex, the
precise definition of which is given in Definition~\ref{bigwittcomplex}.

\begin{bigthm}\label{bigdrwexists}
There exists an initial Witt complex $S \mapsto
\mathbb{W}_S\Omega_A^{\boldsymbol{\cdot}}$ over the ring $A$. In
addition, the canonical  maps
$$\xymatrix{
{ \check{\Omega}_{\mathbb{W}_S(A)}^q } \ar[r]^-{\eta_S} &
{ \mathbb{W}_S\Omega_A^q } \cr
}$$
are surjective, and the diagrams
$$\xy
(-31,0)*{\xymatrix{
{ \check{\Omega}_{\mathbb{W}_S(A)}^q } \ar[r]^-{\eta_S} \ar[d]^{R_T^S} &
{ \mathbb{W}_S\Omega_A^q } \ar[d]^{R_T^S} \cr
{ \check{\Omega}_{\mathbb{W}_T(A)}^q } \ar[r]^-{\eta_T} &
{ \mathbb{W}_T\Omega_A^q } \cr
}};
(-11,0)*{\xymatrix{
{ \check{\Omega}_{\mathbb{W}_S(A)}^q } \ar[r]^-{\eta_S} \ar[d]^{d} &
{ \mathbb{W}_S\Omega_A^q } \ar[d]^{d} \cr
{ \check{\Omega}_{\mathbb{W}_S(A)}^{q+1} } \ar[r]^-{\eta_S} &
{ \mathbb{W}_S\Omega_A^{q+1} } \cr
}};
(9,0)*{\xymatrix{
{ \check{\Omega}_{\mathbb{W}_S(A)}^{q\phantom{+1}} } \ar[r]^-{\eta_S} \ar[d]^{F_m} &
{ \mathbb{W}_S\Omega_A^q } \ar[d]^{F_m} \cr
{ \check{\Omega}_{\mathbb{W}_{S/m}(A)}^q } \ar[r]^-{\eta_{S/m}} &
{ \mathbb{W}_{S/m}\Omega_A^q } \cr
}};
(0,-17)*{};
\endxy$$
commute.
\end{bigthm}

If $A$ is an $\Fp$-algebra and $S = \{1, p, \dots, p^{n-1}\}$,
then $\mathbb{W}_S\Omega_A^{\boldsymbol{\cdot}}$ agrees with the
original $p$-typical de~Rham-Witt complex
$W_n\Omega_A^{\boldsymbol{\cdot}}$ of
Bloch-Deligne-Illusie~\cite{illusie}. More generally, if $A$ is a
$\Z_{(p)}$-algebra and $S = \{1, p, \dots, p^{n-1}\}$, then
$\mathbb{W}_S\Omega_A^{\boldsymbol{\cdot}}$ agrees with the
$p$-typical de~Rham-Witt complex $W_n\Omega_A^{\boldsymbol{\cdot}}$
constructed by the author and Madsen~\cite{hm3} for $p$ odd and by
Costeanu~\cite{costeanu} for $p = 2$. Finally, if $2$ is either
invertible or zero in $A$ and $S$ is arbitrary, then
$\mathbb{W}_S\Omega_A^{\boldsymbol{\cdot}}$ agrees with the big
de~Rham-Witt complex introduced by the author and
Madsen~\cite{hm2}. We also note that if $f \colon R \to A$ is a map of
$\smash{ \Z_{(p)} }$-algebras and $S = \{1, p, \dots, p^{n-1}\}$, then the
relative $p$-typical de~Rham-Witt complex
$\smash{ W_n\Omega_{A/R}^{\boldsymbol{\cdot}} }$ of
Langer-Zink~\cite{langerzink} agrees with the quotient of
$\mathbb{W}_S\Omega_A^{\boldsymbol{\cdot}}$ by the differential graded
ideal generated by the image of $\mathbb{W}_S\Omega_R^1 \to
\mathbb{W}_S\Omega_A^1$. 

We recall that van~der~Kallen~\cite[Theorem~2.4]{vanderkallen1} and
Borger~\cite[Theorem~B]{borger} have proved independently that for
every \'{e}tale morphism $f \colon A \to B$ and every finite subset $S
\subset \N$ stable under division, the induced morphism
$$\xymatrix{
{ \mathbb{W}_S(A) } \ar[r]^-{\mathbb{W}_S(f)} &
{ \mathbb{W}_S(B) } \cr
}$$
again is \'{e}tale. Based on this theorem, we prove the following
result.

\begin{bigthm}\label{etaledescent}Let $f \colon A \to B$ be an
  \'{e}tale map and let $S \subset \N$ be a finite subset stable under
  division. Then the induced map
$$\xymatrix{
{ \mathbb{W}_S(B) \otimes_{\mathbb{W}_S(A)} \mathbb{W}_S\Omega_A^q }
\ar[r] &
{ \mathbb{W}_S\Omega_B^q } \cr
}$$
is an isomorphism, for all $q$.
\end{bigthm}

To prove Theorem~\ref{etaledescent}, we verify that the left-hand
terms form a Witt complex over the ring $B$ and use
Theorem~\ref{bigdrwexists} to obtain the inverse of the map in the
statement. The verification of the Witt complex axioms, in turn, is
significantly simplified by the existence of the divided Frobenius on
$\Omega_{\mathbb{W}(A)}$ as follows from
Theorem~\ref{universalderivations}.

Finally, we evaluate the de~Rham-Witt complex of $\Z$. The result is
that $\mathbb{W}\,\Omega_{\Z}^q$ is non-zero for $q \leqslant 1$
only. Moreover, we may consider
$\mathbb{W}\,\Omega_{\Z}^{\boldsymbol{\cdot}}$ as the quotient
$$\xymatrix{
{ \Omega_{\mathbb{W}(\Z)}^{\boldsymbol{\cdot}} } \ar[r] &
{ \mathbb{W}\,\Omega_{\Z}^{\boldsymbol{\cdot}} } \cr
}$$
of the de~Rham complex of $\mathbb{W}(\Z)$ by a differential graded
ideal generated by elements of degree $1$. Hence, following
Borger~\cite{borger2}, we may interpret
$\mathbb{W}\,\Omega_{\Z}^{\boldsymbol{\cdot}}$ as the complex of
differentials along the leaves of a codimension $1$ foliation of
$\Spec(\Z)$ considered as an $\F_1$-space. We note that, by contrast, 
$\Omega_{\mathbb{W}(\Z)}^q$ is non-zero for all $q$.

As mentioned earlier, the big de~Rham-Witt complex was introduced
in~\cite{hm2} with the purpose of giving an algebraic
description of the equivariant homotopy groups
$$\TR_q^r(A) 
= [ S^q \wedge (\mathbb{T}/C_r)_+, T(A) ]_{\mathbb{T}}$$
of the topological Hochschild $\mathbb{T}$-spectrum $T(A)$ of the ring
$A$. Here $\mathbb{T} = \R/\Z$ is the circle group, $C_r \subset
\mathbb{T}$ is the subgroup of order $r$, and $[-,-]_{\mathbb{T}}$ is
the abelian group of maps in the homotopy category of orthogonal
$\mathbb{T}$-spectra. We proved in~\cite[Section~1]{h} that the
groups $\TR_q^r(A)$ give rise to a Witt complex over the ring $A$ in
the sense of Definition~\ref{bigwittcomplex} below. Therefore, by  Theorem~\ref{bigdrwexists}, there is a unique map
$$\xymatrix{
{ \mathbb{W}_{\langle r \rangle}\Omega_A^q } \ar[r] &
{ \TR_q^r(A) } \cr
}$$
of Witt complexes over $A$, where $\langle r \rangle$ denotes the set
of divisors of $r$. We will show elsewhere that this map is an
isomorphism for all $r$ and all $q \leqslant 1$. 

\section{Witt vectors}\label{wittvectorsection}

We begin with a review of Witt vectors and $\lambda$-rings. The
material in this section is due to Cartier~\cite{cartier},
Grothendieck~\cite{grothendieck1}, Teichm\"{u}ller~\cite{teichmuller},
and Witt~\cite{witt} and accordingly we make no claim of
originality. The reader is also referred to the very readable account by
Bergman~\cite[Appendix]{mumford1} and to the more modern and general
exposition by Borger~\cite{borger}. We further mention the books by
Hazewinkel~\cite{hazewinkel} and Knutson~\cite{knutson}, which focus
more on the r\^{o}le of symmetric functions.

In the approach to Witt vectors taken here, all necessary congruences
are isolated in following lemma, commonly attributed to
Dwork~\cite[E.2.4]{hazewinkel}. Let $\N$ be the set of positive 
integers. We say that a subset $S \subset \N$, possibly empty, is a
truncation set if whenever $n \in S$ and $e$ is a divisor in $n$, then
$e \in S$. The ring of big Witt vectors $\mathbb{W}_S(A)$ associated
with the ring $A$ and the truncation set $S$ is defined to be the set
$A^S$ equipped with a ring structure such that the ghost map
$$\xymatrix{
{ \mathbb{W}_S(A) } \ar[r]^-{w} &
{ {A^{S}} } \cr
}$$
%
%
%
%
that takes the vector $a = (a_n \mid n \in S\,)$ to the sequence $w(a)
= \langle w_n(a) \mid
n \in S \,\rangle$ with
$$w_n(a) = \sum_{e \mid n} ea_e^{n/e}$$
is a natural transformation of functors from the category of
rings to itself. Here the target $A^S$ is considered a
ring with componentwise addition and multiplication. The elements
$a_n$ and $w_n(a)$ are called the Witt components and the ghost
components of the vector $a$, respectively. To prove 
that there exists a unique ring structure on $\mathbb{W}_S(A)$
characterized in this way, we first recall the following result, a
different proof of which is given
in~\cite[Lemma~17.6.1]{hazewinkel}. We write $v_p(n)$ for the $p$-adic
valuation of $n$, normalized such that $v_p(p) = 1$.

\begin{lemma}\label{dwork}Suppose that for every prime number
$p \in S$, there exists a ring homomorphism $\phi_p \colon A \to A$ with
the property that $\phi_p(a) \equiv a^p$ modulo $pA$. Then for every
sequence $x = \langle x_n \mid n \in S \rangle$, the following {\rm
  (i)--(ii)} are equivalent.
\begin{enumerate}
\item[{\rm (i)}]The sequence $x$ is in the image of the ghost map
$w \colon \mathbb{W}_S(A) \to A^S.$
\item[{\rm (ii)}]For every prime number $p \in S$ and every $n \in S$
with $v_p(n) \geqslant 1$,
$$x_n \equiv \phi_p(x_{n/p}) \hskip8mm \text{modulo $p^{v_p(n)}A$.}$$
\end{enumerate}
\end{lemma}

\begin{proof}We first show that if $a \equiv b$ modulo $pA$, then
$a^{p^{v-1}} \equiv b^{p^{v-1}}$ modulo $p^vA$. If we write $a = b +
p\epsilon$, then
$$a^{p^{v-1}} = b^{p^{v-1}} + \sum_{1 \leqslant i \leqslant p^{v-1}}
\binom{p^{v-1}}{i} b^{p^{v-1}-i}p^i\epsilon^i.$$
In general, the $p$-adic valuation of the binomial coefficient
$\binom{m+n}{n}$ is equal to the number of carrys in the addition of
$m$ and $n$ in base $p$. So in particular,
$$v_p\left( \binom{p^{v-1}}{i} p^i \right) = v - 1 + i - v_p(i) \geqslant v$$
which proves the claim. Now, suppose that $x = w(a)$
satisfies~(i). Since $\phi_p \colon A \to A$ is a ring-homomorphism
and lifts the Frobenius of $A/pA$, we have
$$\phi_p(w_{n/p}(a)) = \sum_{e \mid (n/p)} e \phi_p(a_e^{n/pe})$$
which is congruent to $\smash{\sum_{e \mid (n/p)} ea_e^{n/e}}$ modulo
$p^{v_p(n)}A$. But if $e$ divides $n$ but not $n/p$, then $v_p(e) =
v_p(n)$ and hence, this sum is congruent to $\smash{\sum_{e \mid n}
  ea_e^{n/e} = w_n(a)}$ modulo $p^{v_p(n)}A$. This shows that $x$
satisfies~(ii). Conversely, suppose that $x$ satisfies~(ii). We find a
vector $a = (a_n \mid n \in S)$ with $w_n(a) = x_n$ as follows. We let
$a_1 = x_1$ and assume, inductively, that $a_e$ has been chosen, for
all $e \neq n$ that divide $n$, such that $w_e(a) = x_e$. The
calculation above shows that for every prime number $p$ dividing $n$,
$$x_n - \sum_{e \mid n, e \neq n} ea_e^{n/e}$$
is congruent to zero modulo $p^{v_p(n)}A$. Therefore, this difference is
divisible by $n$ and hence is equal to $na_n$ for some 
$a_n \in A$. This shows that $x$ satisfies~(i).
\end{proof}

\begin{prop}\label{ringstructure}There exists a unique ring structure
on the domain of the ghost map
$$\xymatrix{
{ \mathbb{W}_S(A) } \ar[r]^-{w} &
{ A^{S} } \cr
}$$
making it a natural transformation of functors from rings to rings.
\end{prop}

\begin{proof}Let $A$ be the free ring generated by $\{a_n,b_n \mid n
  \in S\}$. The unique ring homomorphism $\phi_p \colon A \to A$
that maps $a_n$ to $a_n^p$ and $b_n$ to $b_n^p$ satisfies 
$\phi_p(f) = f^p$ modulo $pA$. Hence, if $a$ and $b$ are the vectors
$(a_n \mid n \in S)$ and $(b_n \mid n \in S)$, respectively, then
Lemma~\ref{dwork} shows that the sequences $w(a) + w(b)$, $w(a) \cdot
w(b)$, and $-w(a)$ are in the image of the ghost map. It follows that
there are sequences of polynomials $s = (s_n \mid n \in S)$, $p = (p_n
\mid n \in S)$, and $i = (i_n \mid n \in S)$ such that $w(s) = w(a) +
w(b)$, $w(p) = w(a) \cdot w(b)$, and $w(i) = - w(a)$. Moreover, since
$A$ is torsion free, the ghost map is injective, and accordingly,
these polynomials are unique.

Let $A'$ be any ring. If $a' = (a_n' \mid n \in S)$ and $b' =
(b_n' \mid n \in S)$ are two vectors in $\mathbb{W}_S(A')$, then there
is a unique ring homomorphism $f \colon A \to A'$ with the property
that $\mathbb{W}_S(f)(a) = a'$ and $\mathbb{W}_S(f)(b) = b'$. We define
$a' + b' = \mathbb{W}_S(f)(s)$, $a \cdot b = \mathbb{W}_S(f)(p)$, and
$-a = \mathbb{W}_S(f)(i)$. To prove that the ring axioms are verified,
suppose first that $A'$ is torsion free. In this case, the ghost map is
injective, and hence, the ring axioms hold since they do so in
$A^S$. In general, we choose a surjective ring homomorphism 
$g \colon A'' \to A'$ from a torsion free ring $A''$. The induced
map $\mathbb{W}_S(g) \colon \mathbb{W}_S(A'') \to \mathbb{W}_S(A')$ is
again surjective, and since the ring axioms hold in the domain, they
do so, too, in the target.
\end{proof}

If $T \subset S$ are two truncation sets, then the forgetful map
$$\xymatrix{
{ \mathbb{W}_S(A) } \ar[r]^-{R_T^S} &
{ \mathbb{W}_T(A) } \cr
}$$
is a natural ring homomorphism called the restriction from $S$ to
$T$. If $S \subset \N$ is a truncation set and $n \in \N$, then the set 
$$S/n = \{ e \in \N \mid ne \in S\}$$
again is a truncation set, possibly empty. For every $n \in \N$, there
is a natural map
$$\xymatrix{
{ \mathbb{W}_{S/n}(A) } \ar[r]^-{V_n} &
{ \mathbb{W}_S(A) } \cr
}$$
that to the vector $a = (a_e \mid e \in S/n)$ assigns the vector $V_n(a) =
(b_m \mid m \in S)$, where $b_m$ is equal to $a_e$, if $m = ne$, and
$0$, otherwise. It is called the $n$th Verschiebung.

\begin{lemma}\label{verschiebung}For every $n \in \N$, the map
$V_n \colon \mathbb{W}_{S/n}(A) \to \mathbb{W}_S(A)$ is additive.
\end{lemma}

\begin{proof}The following diagram, where $V_n^w$ takes the sequence
$\langle x_e \mid e \in S/n \rangle$ to the sequence whose $m$th
component is $nx_e$, if $m = ne$, and $0$, otherwise, commutes.
$$\xymatrix{
{ \mathbb{W}_{S/n}(A) } \ar[r]^-{V_n} \ar[d]^-{w} &
{ \mathbb{W}_S(A) } \ar[d]^-{w} \cr
{ A^{S/n} } \ar[r]^-{V_n^w} &
{ A^S } \cr
}$$
Since $V_n^w$ is additive, so is $V_n$. Indeed, if $A$ is torsion
free, the horizontal maps are both injective, and hence $V_n$ is
additive in this case. In general, we choose a surjective ring
homomorphism $g \colon A' \to A$ and argue as in the proof of
Proposition~\ref{ringstructure}.
\end{proof}

\begin{lemma}\label{frobenius}For every $n \in \N$, there exists a
unique natural ring homomorphism 
$$\xymatrix{
{ \mathbb{W}_S(A) } \ar[r]^-{F_n} &
{ \mathbb{W}_{S/n}(A) } \cr
}$$
called the $n$th Frobenius that makes the following diagram, where the
map $F_n^w$ takes the sequence $x = \langle x_m \mid m \in S\rangle$
to the sequence $F_n^w(x) = \langle x_{ne} \mid e \in S/n \rangle$, commute.  
$$\xymatrix{
{ \mathbb{W}_S(A) } \ar[r]^-{F_n} \ar[d]^-{w} &
{ \mathbb{W}_{S/n}(A) } \ar[d]^-{w} \cr
{ A^S } \ar[r]^-{F_n^w} &
{ A^{S/n} } \cr
}$$
\end{lemma}

\begin{proof}The construction of the map $F_n$ is similar to the proof
of Proposition~\ref{ringstructure}. We let $A$ be the free ring
generated by $\{a_m \mid m \in S\}$, and let $a$ be the vector
$(a_m \mid m \in S)$. By Lemma~\ref{dwork}, the sequence $F_n^w(w(a))
\in A^{S/n}$ is the image by the ghost map of a vector
$$F_n(a) = (f_{n,e} \mid e \in S/n) \in \mathbb{W}_{S/n}(A),$$
this vector being unique since $A$ is torsion free. If $A'$ is any
ring and $a' = (a_m' \mid m \in S)$ a vector in $\mathbb{W}_S(A')$,
then we define $F_n(a') = \mathbb{W}_{S/n}(g)(F_n(a))$, where
$g \colon A \to A'$ is the unique ring homomorphism that maps $a$ to
$a'$. Finally, since the map $F_n^w$ is a ring homomorphism, an
argument similar to the proof of Lemma~\ref{verschiebung} shows that
also $F_n$ is a ring homomorphism.
\end{proof}

The Teichm\"{u}ller representative is the map
$$\xymatrix{
{ A } \ar[r]^-{[-]_S} &
{ \mathbb{W}_S(A) } \cr
}$$
whose $m$th component is $a$, if $m = 1$, and $0$, otherwise. It is a
multiplicative map. Indeed, the following diagram, where $[a]_S^w$
is the sequence with $m$th component $a^m$, commutes, and $[-]_S^w$ is
a multiplicative map.
$$\xymatrix{
{ A } \ar[r]^-{[-]_S} \ar@{=}[d] &
{ \mathbb{W}_S(A) } \ar[d]^-{w} \cr
{ A } \ar[r]^-{[-]_S^w} &
{ A^S } \cr
}$$
In particular, the Teichm\"{u}ller representative $[1]_S$ is the
multiplicative identity element in the ring $\mathbb{W}_S(A)$.

\begin{lemma}\label{variousrelations}Let $S \subset \N$ be a
truncation set and let $A$ be a ring.
\begin{enumerate}
\item[{\rm (i)}]For all $a \in \mathbb{W}_S(A)$, there is a convergent
sum
$$a = \sum_{n \in S} V_n([a_n]_{S/n}).$$
\item[{\rm (ii)}]For all $m,n \in \N$ with greatest common divisor 
$c = (m,n)$, 
$$F_mV_n = cV_{n/c}F_{m/c} \colon \mathbb{W}_{S/c}(A) \to \mathbb{W}_{S/c}(A).$$
\item[{\rm (iii)}]For all $n \in \N$, $a \in \mathbb{W}_S(A)$, and $a'
  \in \mathbb{W}_{S/n}(A)$,
$$a V_n(a') = V_n(F_n(a) a').$$
\item[{\rm (iv)}]For all $m,n \in \N$,
$$\begin{aligned}
F_m F_n = F_{mn} & \colon \mathbb{W}_S(A) \to \mathbb{W}_{S/mn}(A),
\cr
V_nV_m = V_{mn} & \colon \mathbb{W}_{S/mn}(A) \to \mathbb{W}_S(A). \cr
\end{aligned}$$
\item[{\rm (v)}]For all $n \in \N$ and $a \in A$,
$$F_n([a]_S^{\phantom{n}}) = [a]_{S/n}^n.$$
\end{enumerate}
\end{lemma}

\begin{proof}One readily verifies that the two sides of each equation have
the same image by the ghost map. This shows that the relations hold,
if $A$ is torsion free, and hence, in general. In statement~(i), the
convergence, for $S$ infinite, is with respect to the product topology
on $\mathbb{W}_S(A)$ induced by the discrete topology on $A$.
\end{proof}

\begin{prop}\label{W_S(Z)}The ring of Witt vectors in $\Z$ is equal to
the abelian group
$$\mathbb{W}_S(\Z) = \prod_{n \in S} \Z \cdot V_n([1]_{S/n})$$
with the multiplication given by
$$V_m([1]_{S/m}) \cdot V_n([1]_{S/n}) = c \cdot V_e([1]_{S/e})$$
where $c = (m,n)$ and $e = [m,n]$ are the greatest common divisor and
the least common multiple, respectively.
\end{prop}

\begin{proof}The formula for the product follows from
Lemma~\ref{variousrelations}~(ii)--(iv). For finite $S$, we
prove the statement by induction beginning from the
case $S = \emptyset$ which is trivial. So suppose that $S$ is
non-empty, let $m \in S$ be maximal, and let $T = S
\smallsetminus \{m\}$. The sequence of abelian groups
$$\xymatrix{
{ 0 } \ar[r] &
{ \mathbb{W}_{\{1\}}(\Z) } \ar[r]^-{V_m} &
{ \mathbb{W}_S(\Z) } \ar[r]^-{R_T^S} &
{ \mathbb{W}_T(\Z) } \ar[r] &
{ 0 } \cr
}$$
is exact, and we wish to show that it is equal to the following exact sequence.
$$\xymatrix{
{ 0 } \ar[r] &
{ \Z \cdot {[1]_{\{1\}}} } \ar[r]^-{V_m} &
{ \displaystyle{ \prod_{n \in S} \Z \cdot V_n([1]_{S/n}) } } \ar[r]^-{R_T^S} &
{ \displaystyle{ \prod_{n \in T} \Z \cdot V_n([1]_{T/n}) } } \ar[r] &
{ 0 } \cr
}$$ The latter sequence maps to the former, and by induction, the
right-hand terms of the two sequences are equal. Since also the
left-hand terms are equal, so are the middle terms. This completes the
proof for $S$ finite. Finally, every truncation set $S$ is the union
of its finite sub-truncation sets $S_{\alpha} \subset S$ and
$\mathbb{W}_S(\Z) = \lim_{\alpha} \mathbb{W}_{S_\alpha}(\Z)$.
\end{proof}

The values of the restriction, Frobenius, and Verschiebung maps
on the generators $V_n([1]_{S/n})$ are readily evaluated by using
Lemma~\ref{variousrelations}~(ii)--(iv). To give a formula for the
Teichm\"{u}ller representative, we recall the M\"{o}bius inversion
formula. Let $g \colon \N \to \Z$ be a function and define the
function $f \colon \N \to \Z$ by $f(n) = \sum_{e \mid n} g(e)$. Then
the original function is given by $g(n) = \sum_{e \mid n} \mu(e) f(n/e)$,
where $\mu \colon \N \to \{-1,0,1\}$ is the M\"{o}bius function
defined by $\mu(e) = (-1)^r$, if $e$ is a product of $r \geqslant 0$
distinct prime numbers, and $\mu(e) = 0$, otherwise.

\begin{addendum}\label{teichmuller}If $m$ is an integer and $S$ a
truncation set, then
$$[m]_S = \sum_{n \in S} \frac{1}{n} \big(\sum_{e \mid n}
\mu(e)m^{n/e} \big) V_n([1]_{S/n}).$$
In particular, the square root of unity $[-1]_S$ is equal to
$-[1]_S + V_2([1]_{S/2})$.
\end{addendum}

\begin{proof}It suffices to prove that the formula holds in
$\mathbb{W}_S(\Z)$. By Proposition~\ref{W_S(Z)}, there are
unique integers $r_e$, $e \in S$ such that
$$[m]_S = \sum_{e \in S} r_e V_e([1]_{S/e}).$$
Evaluating the $n$th ghost component of this equation, we find that
$$m^n = \sum_{e \mid n} er_e$$
from which the stated formula follows by M\"{o}bius
inversion. Finally, defining $g(n)$ to be $-1$, if $n = 1$; $2$, if
$n = 2$; and $0$, otherwise, we get $f(n) = \sum_{e \mid n}g(e) =
(-1)^n$, which proves the stated formula for $[-1]_S$.
\end{proof}

If $m = q$ is a prime power, then the coefficient of
$V_n([1]_{S/n})$ in $[m]_S$ is equal to the number of monic irreducible
polynomials of degree $n$ over the finite field $\F_q$.

\begin{lemma}\label{modpfrobenius}If $A$ is an $\Fp$-algebra and $S$ a
truncation set, then
$$F_p = R_{S/p}^S \circ \mathbb{W}_S(\varphi) \colon \mathbb{W}_S(A)
\to \mathbb{W}_{S/p}(A),$$
where $\varphi \colon A \to  A$ is the Frobenius endomorphism. 
\end{lemma}

\begin{proof}By definition $F_p(a) = (f_{p,e}(a) \mid e \in S/p)$ with
the elements $f_{p,e}$ of the free ring on $\{a_n \mid n \in S\}$
characterized by the system of equations 
$$\sum_{e \mid m} e f_{p,e}^{m/e} = \sum_{e \mid pm} ea_e^{pm/e}$$
indexed by $m \in S/p$. The lemma is equivalent to the statement that
for all $m \in S/p$, $f_{p,m} \equiv a_m^p$ modulo $p$, which we
proceed to prove by induction on $m \in S/p$. Since $f_{p,1} = a_1^p +
pa_p$, the statement holds for $m = 1$. So we let $m > 1$ and assume
that for all proper divisors $e$ of $m$, $f_{p,e} \equiv a_e^p$ modulo
$p$. This implies that $\smash{ ef_{p,e}^{m/e} \equiv ea_e^{pm/e} }$ modulo
$\smash{ p^{v_p(m)+1} }$ by the argument at the beginning of the proof
of Lemma~\ref{dwork}. We now write
$$\sum_{e \mid m} ef_{p,e}^{m/e} = \sum_{e \mid m} ea_e^{pm/e} + \sum_{e \mid
pm, e \nmid m} ea_e^{pm/e}$$
and note that if $e \mid pm$ and $e \nmid m$, then $v_p(e) =
v_p(m) + 1$. Therefore, we may conclude that 
$m f_{p,m} \equiv m a_m^p$
modulo $p^{v_p(m)+1}A$. But the free ring on $\{a_n \mid n \in S\}$ is
torsion free, so $f_{p,m} \equiv a_m^p$ modulo $p$ as desired. This
completes the proof.
\end{proof}

\begin{lemma}\label{invertible}Let $m$ be an integer, let $A$ be a
ring, and let $S$ be a truncation set. If $m$ is invertible in $A$,
then $m$ is invertible in $\mathbb{W}_S(A)$; and if $m$ is a
non-zero-divisor in $A$, then $m$ is a non-zero-divisor in
$\mathbb{W}_S(A)$.
\end{lemma}

\begin{proof}As in the proof of Proposition~\ref{W_S(Z)}, we may
assume that $S$ is finite. We proceed by induction on $S$ beginning
from the trivial case $S = \emptyset$. So let $S$ be non-empty and
assume the statement for all proper sub-truncation sets of $S$. We let
$n \in S$ be maximal, and let $T = S
\smallsetminus \{n\}$. In this situation, we have exact sequence
$$\xymatrix{
{ 0 } \ar[r] &
{ \mathbb{W}_{\{1\}}(A) } \ar[r]^-{V_n} &
{ \mathbb{W}_S(A) } \ar[r]^-{R_T^S} &
{ \mathbb{W}_T(A) } \ar[r] &
{ 0 } \cr
}$$
from which the induction step readily follows, since
$\mathbb{W}_{\{1\}}(A) = A$. 
\end{proof}

Let $p$ be prime number. We say that a sub-truncation set of the
truncation set
$$P = \{1, p, p^2, \dots\} \subset \N$$
is a $p$-typical truncation set. For instance, if $S$ is any truncation
set, then $S \cap P$ is a $p$-typical truncation set. The
$p$-typical truncation sets $T \subset P$ are $T = \emptyset$, $T
= P$, and $T = \{1, p, \dots, p^{n-1}\}$, where $n$ is a positive integer.
The ring
$\mathbb{W}_P(A)$ is called the ring of $p$-typical Witt vectors
and the ring $\mathbb{W}_{\{1,p,\dots,p^{n-1}\}}(A)$ is called the ring of
$p$-typical Witt vectors of length $n$ in $A$. 

\begin{prop}\label{wittptypicaldecomposition}Let $p$ be a prime
number, let $S$ be a truncation set, and let $I(S)$ be the set of $k
\in S$ not divisible by $p$. If $A$ is a ring in which every 
$k \in I(S)$ is invertible, then the ring homomorphism 
$$\xymatrix{
{ \mathbb{W}_S(A) } \ar[r]^-{g} &
{ \displaystyle{ \prod_{k \in I(S)} \mathbb{W}_{(S/k) \cap P}(A) } } \cr
}$$
whose $k$th component is $g_k = R_{(S/k) \cap P}^{S/k} \circ F_k$ is
an isomorphism.
\end{prop}

\begin{proof}We have a commutative diagram of ring homomorphisms
$$\begin{xy}
(-16,8)*+{ \mathbb{W}_S(A) }="a";
(16,8)*+{ \displaystyle{ \prod \mathbb{W}_{(S/k) \cap P}(A) } }="b";
(-16,-8)*+{ A^S }="c";
(16,-8)*+{ \displaystyle{ \prod A^{(S/k) \cap P} } }="d";
{ \ar^-{g} "b"; "a";};
{ \ar^-{w} "c"; "a";};
{ \ar^-{\prod w} "d";"b";};
{ \ar^-{g^w} "d";"c";};
\end{xy}$$
where the products on the right-hand side range over $k \in I(S)$ and
where $g^w$ is the map whose $k$th component $g_k^w$ is given by
$g_k^w(a)_{p^v} = a_{p^vk}$. The map $g^w$ is a bijection since the
sets $S \cap kP$ with $k \in I(S)$ partition $S$ and since the maps $S \cap
kP \to (S/k) \cap P$ that take $p^vk$ to $p^v$ are bijections. Let
$h^w$ be the inverse of $g^w$. We claim that there exists a natural
function
$h \colon \prod \mathbb{W}_{(S/k) \cap P}(A) \to \mathbb{W}_S(A)$ such
that $w \circ h = h^w \circ (\prod w)$. Granting this, the equalities
$g^w \circ h^w = \id$ and $h^w \circ g^w = \id$ imply that $g \circ h
= \id$ and $h \circ g = \id$, which proves the proposition.

To prove the claim, it suffices to show that, in the universal case,
where $A$ is the free $\Z[I(S)^{-1}]$-algebra generated by
$\{a_{k,p^v} \mid k \in I(S), p^v \in (S/k) \cap P\}$ and $a = (a_k)$
with $a_k = (a_{k,p^v}) \in \mathbb{W}_{(S/k) \cap P}(A)$, the element
$x = (h^w \circ (\prod w))(a)$ is in the image of $w \colon
\mathbb{W}_S(A) \to A^S$. The unique $\Z[I(S)^{-1}]$-algebra
homomorphism $\phi_p \colon A \to A$ that to $a_{k,p^v}$ associates
$a_{k,p^v}^p$ is a lift of the Frobenius of $A/pA$. Moreover, all
prime numbers $\ell \in S$ different from $p$ are invertible in
$A$. Therefore, we conclude from Lemma~\ref{dwork} that the sequence
$x = \langle x_n \mid n \in S\rangle$ is in the image of the ghost map
if and only if for all $n = p^vk \in S$ with $k \in I(S)$ and $v
\geqslant 1$, $x_{p^vk} \equiv \phi_p(x_{p^{v-1}k})$ modulo
$p^vA$. But $x_{p^vk} = w_{p^v}(a_k)$ and $\phi_p(x_{p^{v-1}k}) =
\phi_p(w_{p^{v-1}}(a_k))$ which are congruent modulo $p^vA$ by
Lemma~\ref{dwork}. Hence, there exists a vector $h(a) \in \mathbb{W}_S(A)$
such that $x = w(h(a))$ and this vector is unique, as $A$ is torsion
free. The vector $h(a)$, in turn, uniquely determines the desired
natural map $h$. This completes the proof. 
\end{proof}

\begin{example}\label{bigwittlengthn}If $S = \{1, 2, \dots, n\}$, then $(S/k) \cap P = \{1, p, \dots, p^{s-1} \}$ where $s =
  s(n,k)$ is the unique integer with $p^{s-1}k \leqslant n < p^s
  k$. Hence, if every integer $1 \leqslant k \leqslant n$
not divisible by $p$ is invertible in $A$, then
Proposition~\ref{wittptypicaldecomposition} gives an isomorphism
$$\xymatrix{
{ {}^{\phantom{()}}\mathbb{W}_{\{1,2,\dots,n\}}(A) } \ar[r]^-{\sim} &
{ \displaystyle{ \prod \; \mathbb{W}_{\{1,p,\dots,p^{s-1}\}}(A) } } \cr
}$$
where the product ranges over integers $1 \leqslant k \leqslant n$ not divisible
by $p$ and $s = s(n,k)$.
\end{example}

\begin{lemma}\label{VFp}If $A$ is an $\Fp$-algebra, then for every
truncation set $S$,
$$V_p \circ F_p = p \cdot \id \colon
\mathbb{W}_S(A) \to \mathbb{W}_S(A).$$ 
\end{lemma}

\begin{proof}We first reduce to the case where $S$ is a
$p$-typical truncation set. It follows from
Lemma~\ref{variousrelations} that the following diagram, where the
products range over $k \in I(S)$ and where the vertical maps are the
isomorphisms of Proposition~\ref{wittptypicaldecomposition},
commutes,
$$\begin{xy}
(-38,8)*+{ \mathbb{W}_S(A) }="a";
(0,8)*+{\mathbb{W}_{S/p}(A)}="b";
(38,8)*+{\mathbb{W}_S(A)}="c";
(-38,-8)*+{ \displaystyle{ \prod \mathbb{W}_{(S/k) \cap
      P}(A) } }="d";
(0,-8)*+{ \displaystyle{ \prod \mathbb{W}_{(S/pk) \cap
      P}(A) } }="e";
(38,-8)*+{ \displaystyle{ \prod \mathbb{W}_{(S/k) \cap
      P}(A). } }="f";
{ \ar^-{F_p} "b"; "a";};
{ \ar^-{V_p} "c"; "b";};
{ \ar^-{g} "d";"a";};
{ \ar^-{g} "e";"b";};
{ \ar^-{g} "f";"c";};
{ \ar^-{\prod F_p} "e";"d";};
{ \ar^-{\prod V_p} "f";"e";};
\end{xy}$$
Accordingly, it will suffice to prove the lemma for $p$-typical
truncation sets $S$, and we may further assume that $S$ is finite. 
It follows from Lemma~\ref{variousrelations}~(iii) that
$$V_p \circ F_p = V_p([1]_{S/p}) \cdot \id \colon \mathbb{W}_S(A) \to
\mathbb{W}_S(A)$$
and we proceed to prove that $V_p([1]_{S/p}) = p[1]_S$ by induction on
the cardinality $n$ of $S$. The case $n = 0$ holds trivially, so we
let $S = \{1, p, \dots, p^{n-1}\}$ be the $p$-typical truncation set
of cardinality $n> 0$ and assume that the identity in question has
been proved for all proper sub-truncation sets $T \subset S$. The
exact sequences
$$\xymatrix{
{ 0 } \ar[r] &
{ \mathbb{W}_{\{1\}}(A) } \ar[r]^-{V_{p^{n-1}}} &
{ \mathbb{W}_S(A) } \ar[r]^-{R_{S/p}^S} &
{ \mathbb{W}_{S/p}(A) } \ar[r] &
{ 0 } \cr
}$$
furnish an induction argument showing that $\mathbb{W}_S(A)$ is
annihilated by $p^n$. In particular, $V_p([1]_{S/p})$ is annihilated by
$p^{n-1}$. Moreover, it follows from Addendum~\ref{teichmuller} that
$$[p]_S = p [1]_S + \sum_{0 < s < n} \frac{1}{p^s}(p^{p^s} -
p^{p^{s-1}}) V_{p^s}([1]_{S/p^s})$$
and the left-hand side vanishes, since $A$ is an $\Fp$-algebra. The
inductive hypothesis shows that $V_{p^s}([1]_{S/p^s}) = 
p^{s-1}V_p([1]_{S/p})$, so the formula above becomes
$$0 = p[1]_S + (p^{p^{n-1}-1} - 1)V_p([1]_{S/p}).$$
But $p^{n-1}-1 \geqslant n-1$, so $V_p([1]_{S/p}) = p[1]_S$ which proves
the induction step. 
\end{proof}

Let $A$ be a $p$-torsion free ring equipped with a ring homomorphism
$\phi \colon A \to A$ such that $\phi(a) \equiv a^p$ modulo $pA$. By
Lemma~\ref{dwork}, there is a unique ring homomorphism
$$\lambda_{\phi} \colon A \to \mathbb{W}_P(A)$$
such that $w_{p^n} \circ \lambda_{\phi} = \phi^n$.  We define
$s_{\phi} \colon A \to \mathbb{W}_P(A/pA)$ to be the composition of
$\lambda_{\phi}$ and the map induced by the canonical projection of
$A$ onto $A/pA$. We recall that $A/pA$ is said to be perfect, if
the Frobenius $\varphi \colon A/pA \to A/pA$ is an automorphism.

\begin{prop}\label{padicring}Let $p$ be a prime number, let $n$ be a
non-negative integer, and let $S$ be the finite $p$-typical truncation
set of cardinality $n$. Let $A$ be a $p$-torsion free ring
equipped with a ring homomorphism $\phi \colon A \to A$ such that
$\phi(a) \equiv a^p$ modulo $pA$ and suppose that $A/pA$ is
perfect. In this situation, the map $s_{\phi}$ induces an isomorphism
$$\xymatrix{
{ A/p^nA } \ar[r]^-{\bar{s}_{\phi}} &
{ \mathbb{W}_S(A/pA). } \cr
}$$
\end{prop}

\begin{proof}We claim that the map $s_{\phi}$ induces a map
$\bar{s}_{\phi}$ as stated. Indeed, the restriction map 
$R_S^P \colon \mathbb{W}_P(A/pA) \to \mathbb{W}_S(A/pA)$ has kernel
$V_{p^n}\mathbb{W}_P(A/pA)$, and
$$V_{p^n}\mathbb{W}_P(A/pA) = V_{p^n}\mathbb{W}_P(\varphi^n(A/pA)) 
= V_{p^n}F_{p^n}\mathbb{W}(A/pA) = p^n\mathbb{W}_P(A/pA),$$
where the left-hand equality follows from $A/pA$ being perfect, the
middle equality from Lemma~\ref{modpfrobenius}, and the right-hand
equality from Lemma~\ref{VFp}. Now, the proof is completed by an induction
argument based on the commutative diagram
$$\begin{xy}
(-48,8)*+{ 0 }="a";
(-28,8)*+{ A/pA }="b";
(0,8)*+{ A/p^nA }="c";
(28,8)*+{ A/p^{n-1}A }="d";
(48,8)*+{ 0 }="e";
(-48,-8)*+{ 0 }="f";
(-28,-8)*+{ A/pA }="g";
(0,-8)*+{ \mathbb{W}_S(A/pA) }="h";
(28,-8)*+{ \mathbb{W}_{S/p}(A/pA) }="i";
(48.5,-8)*+{ 0, }="j";
{\ar "b";"a";};
{\ar^-{p^{n-1}} "c";"b";};
{\ar^-{\operatorname{pr}} "d";"c";};
{\ar^-{\varphi^{n-1}} "g";"b";};
{\ar^-{\bar{s}_{\phi}} "h";"c";};
{\ar^-{\bar{s}_{\phi}} "i";"d";};
{\ar "e";"d";};
{\ar "g";"f";};
{\ar^-{V_{p^{n-1}}} "h";"g";};
{\ar^-{R_{S/p}^S} "i";"h";};
{\ar "j";"i";};
\end{xy}$$
where the top horizontal sequence is exact since $A$ is $p$-torsion
free, and where the left-hand vertical map is an isomorphism since $A/pA$
is perfect.
\end{proof}

We return to the ring of big Witt vectors. We prove that the
underlying additive group of the ring $\mathbb{W}(A)$ is naturally
isomorphic to the multiplicative group
$$\Lambda(A) = (1 + tA[\![t]\!])^*$$
of power series with constant term $1$. We
also view the set $tA[\![t]\!]$ of power series with constant term $0$
as an abelian group under coefficientwise addition. We recall the
following result from~\cite[Section~1]{cartier}; see
also~\cite[Proposition~17.2.9]{hazewinkel}.

\begin{prop}\label{changeofcoordinates}The diagram of natural group
homomorphisms
$$\xymatrix{
{ \mathbb{W}(A) } \ar[r]^-{\gamma} \ar[d]^{w} &
{ \Lambda(A) } \ar[d]^{t\frac{d}{dt}\log} \cr
{ A_{\phantom{\N}}^{\N} } \ar[r]^-{\gamma^w} &
{ tA[\![t]\!], } \cr
}$$
where $\gamma(a_1,a_2,\dots ) = \prod_{n \geqslant 1} (1 -
a_nt^n)^{-1}$ and $\gamma^w\langle x_1,x_2,\dots \rangle = \sum_{n
  \geqslant 1} x_nt^n$,
commutes, and the horizontal maps are isomorphisms.
\end{prop}

\begin{proof}It is clear that the maps in the diagram are natural
transformations of functors from the category of rings to the category 
of sets. Moreover, the calculation
$$\begin{aligned}
{}  t \frac{d}{dt}\log ( \prod_{e \geqslant 1} (1 - a_et^e)^{-1})
{} & = - \sum_{e \geqslant 1} t \frac{d}{dt}\log(1 - a_et^e) 
{} = \sum_{e \geqslant 1} \frac{ea_et^e}{1 - a_et^e} \cr
{} & = \sum_{e \geqslant 1} \sum_{q \geqslant 1} ea_e^q t^{qe} 
{} = \sum_{n \geqslant 1} \big( \sum_{e \mid n} ea_e^{n/e} \big) t^n
\cr
\end{aligned}$$
shows that the diagram commutes. It is also clear that the two
vertical maps are group homomorphisms and that the map $\gamma^w$ is
an isomorphism of abelian groups. This implies that the map $\gamma$ 
is a group homomorphism. Indeed, if $A$ is torsion free, then the
vertical maps both are injective, and in general, we choose a
surjective ring homomorphism $A' \to A$ from a torsion free ring and
use that $\mathbb{W}(-)$ and $\Lambda(-)$ both take surjective ring
homomorphisms to surjective group homorphisms.

It remains to show that $\gamma$ is a bijection. To this end, we write
$$\prod_{n \geqslant 1} (1 - a_nt^n)^{-1} = (1 + b_1t + b_2t^2 + \dots)^{-1}$$
where the coefficient $b_n$ is given by the sum
$b_n = \sum (-1)^r a_{i_1} \dots a_{i_r}$
that ranges over all $1 \leqslant i_1 < \dots < i_r \leqslant n$ such
that $i_1 + 2i_2 + \dots +ri_r = n$. It follows that the Witt
coordinates $a_n$ are uniquely determined, recursively, by the
coefficients $b_n$, and hence, that $\gamma$ is a bijection as
stated.
\end{proof}

\begin{remark}\label{signs}We will always consider the set
$\Lambda(A) = 1 + tA[\![t]\!]$ as a ring with the unique ring
structure that makes the map $\gamma \colon \mathbb{W}(A) \to
\Lambda(A)$ a ring isomorphism. This ring structure 
is characterized by begin natural in $A$, by addition being given by
power series multiplication, and by the product satisfying
$$(1-at)^{-1} * (1-bt)^{-1} = (1-abt)^{-1}$$
for all $a,b \in A$; compare~\cite[Section~4]{grothendieck1}. We note
that $(1-t)^{-1}$ is the multiplicative unit element in
$\Lambda(A)$. The reader is warned, however, that there exists four
different ring structures on the set $1 + tA[\![t]\!]$ satisfying the
first two of these requirements but with the last requirement replaced
by the four possible choices of signs in the product formula $(1 \pm
at)^{\pm1} * (1 \pm bt)^{\pm1} = (1 \pm abt)^{\pm1}$. The choice $++$
is used in~\cite{grothendieck1,atiyahtall,SGA6}, while the choice $-+$
is used in~\cite[Section~17.2]{hazewinkel}. The four different rings
$\Lambda(A)_{\pm\pm}$ are all naturally isomorphic, the
natural isomorphism $u_{\pm\pm} \colon \Lambda(A) \to
\Lambda(A)_{\pm\pm}$ given by $u_{\pm\pm}(f(t)) = (1\pm t)^{\pm1} *
f(t)$, where the product is evaluated in $\Lambda(A)$. We also write
$\gamma_{\pm\pm} \colon \mathbb{W}(A) \to \Lambda(A)_{\pm\pm}$ for the
natural ring isomorphism $\gamma_{\pm\pm} = u_{\pm\pm} \circ \gamma$;
in particular, $\gamma = \gamma_{--}$. 
\end{remark}

\begin{addendum}\label{changeofcoordinatesS}The map $\gamma$ induces
an isomorphism of abelian groups
$$\xymatrix{
{ \mathbb{W}_S(A) } \ar[r]^-{\gamma_S} &
{ \Lambda_S(A) } \cr
}$$
where $\Lambda_S(A)$ is the quotient of the multiplicative group
$\Lambda(A) = (1 + tA[\![t]\!])^*$ by the subgroup $I_S(A)$ of all power
series of the form $\prod_{n \in \N \smallsetminus S} (1 - a_nt^n)^{-1}$.
\end{addendum}

\begin{proof}The kernel of the restriction map $R_S^{\N} \colon
\mathbb{W}(A) \to \mathbb{W}_S(A)$ is equal to the subset of all
vectors $a = (a_n \mid n \in \N)$ such that $a_n = 0$, if $n \in
S$. The image of this subset by the map $\gamma$ is the subset $I_S(A)
\subset \Lambda(A)$.
\end{proof}

\begin{example}\label{bigwittlengthm}If $S = \{1,2,\dots,n\}$, then
$I_S(A) = (1 + t^{n+1}A[\![t]\!])^*$. Hence, in this case,
Addendum~\ref{changeofcoordinatesS} gives an isomorphism of abelian
groups
$$\xymatrix{
{ \mathbb{W}_{\{1,2,\dots,n\}}(A) } \ar[r]^-{\gamma_S} &
{ (1 + tA[\![t]\!])^*/(1 + t^{n+1}A[\![t]\!])^*. } \cr
}$$
For $A$ a $\Z_{(p)}$-algebra, the structure of this group was examined
in Example~\ref{bigwittlengthn}.
\end{example}

\begin{lemma}\label{pfrobenius}Let $A$ be an arbitrary ring. For every prime
number $p$, the natural ring homomorphism
$F_p \colon \mathbb{W}(A) \to \mathbb{W}(A)$
satisfies that $F_p(a) \equiv a^p$ modulo $p\mathbb{W}(A)$.
\end{lemma}

\begin{proof}By naturality, it suffices to consider 
$A = \Z[a_1,a_2,\dots]$ and $a = (a_1,a_2,\dots)$ and show that there
exists $b \in \mathbb{W}(A)$ with $F_p(a) - a^p = pb$. We have
$$w_n(F_p(a) - a^p) = \sum_{e \mid pn} ea_e^{pn/e} - \big( \sum_{e
\mid n} ea_e^{n/e}\big)^p$$
which clearly is congruent to zero modulo $pA$. So we let
$x = \langle x_n \mid n \in \N \rangle$ with
$$x_n = \frac{1}{p}w_n(F_p(a) - a^p)$$
and employ Lemma~\ref{dwork} to show that $x = w(b)$ with $b \in
\mathbb{W}(A)$. To this end, we must show that for every prime number
$\ell$ and every $n \in \ell\N$,
$$x_n \equiv \phi_{\ell}(x_{n/\ell})$$
modulo $\ell^{v_{\ell}(n)}A$, where $\phi_{\ell} \colon A \to A$ is the
unique ring homomorphism that takes $a_n$ to $a_n^{\ell}$. The
congruence in question is equivalent to the statement that
$$w_n(F_p(a) - a^p) \equiv \phi_{\ell}(w_{n/\ell}(F_p(a) - a^p))$$
modulo $\ell^{v_{\ell}(n)}A$, if $\ell \neq p$ and $n \in \ell\N$, and 
modulo $p^{v_p(n)+1}A$, if $\ell = p$ and $n \in p\N$. If
$\ell \neq p$, the statement follows from Lemma~\ref{dwork}, and if
$\ell = p$ and $n \in p\N$, we find
$$\begin{aligned}
{} & w_n(F_p(a) - a^p) - \phi_p(w_{n/p}(F_p(a) - a^p)) \cr
{} & = \sum_{e \mid pn} ea_e^{pn/e}
- \big(\sum_{e \mid n} ea_e^{n/e}\big)^p
- \sum_{e \mid n} ea_e^{pn/e}
+ \big(\sum_{e \mid (n/p)} ea_e^{n/e}\big)^p. \cr
\end{aligned}$$
If $e \mid pn$ and $e \nmid n$, then $v_p(e) = v_p(n) + 1$, so 
$$\sum_{e \mid pn} ea_e^{pn/e} \equiv \sum_{e \mid n} ea_e^{pn/e}$$
modulo $p^{v_p(n)+1}A$. Similarly,
if $e \mid n$ and $e \nmid (n/p)$, then $v_p(e) = v_p(n)$, and hence,
$$\sum_{e \mid n} ea_e^{n/e} \equiv \sum_{e \mid (n/p)} ea_e^{n/e}$$
modulo $p^{v_p(n)}A$. But then
$$\big(\sum_{e \mid n} ea_e^{n/e}\big)^p \equiv
\big(\sum_{e \mid (n/p)} ea_e^{n/e}\big)^p$$
modulo $p^{v_p(n)+1}A$ as required; compare the proof of Lemma~\ref{dwork}.
\end{proof}

We next recall the following result of Cartier
from~\cite[Theorem~17.6.17]{hazewinkel}.

\begin{prop}\label{universallambdaoperation}There exists a unique
natural ring homomorphism
$$\Delta = \Delta_A \colon \mathbb{W}(A) \to \mathbb{W}(\mathbb{W}(A))$$
such that for every positive integer $n$,
$$w_n \circ \Delta = F_n \colon \mathbb{W}(A) \to \mathbb{W}(A).$$
In addition, the following diagrams, where
$\epsilon_A = w_1 \colon \mathbb{W}(A) \to A$, commute. 
$$\begin{xy}
(-48,7)*+{ \mathbb{W}(A) }="11";
(-25,7)*+{ \mathbb{W}(\mathbb{W}(A)) }="12";
(-2,7)*+{ \mathbb{W}(A) }="13";
(21,7)*+{ \mathbb{W}(\mathbb{W}(\mathbb{W}(A))) }="14";
(53,7)*+{ \mathbb{W}(\mathbb{W}(A)) }="15";
(-25,-7)*+{ \mathbb{W}(A) }="22";
(21,-7)*+{ \mathbb{W}(\mathbb{W}(A)) }="24";
(53,-7)*+{ \mathbb{W}(A) }="25";
{\ar_-{\epsilon_{\mathbb{W}(A)}} "11";"12";};
{\ar^-{\mathbb{W}(\epsilon_A)} "13";"12";};
{\ar_-{\Delta_A} "12";"22";};
{\ar@{=} "11";"22";};
{\ar@{=} "13";"22";};
{\ar_-{\Delta_{\mathbb{W}(A)}} "14";"15";};
{\ar_-{\mathbb{W}(\Delta_A)} "14";"24";};
{\ar_-{\Delta_A} "24";"25";};
{\ar_-{\Delta_A} "15";"25";};
\end{xy}$$
\end{prop}

\begin{proof}We first prove that a natural ring homomorphism as stated
exists. It suffices to prove that in the universal case $A =
\Z[a_1,a_2,\dots]$ and $a = (a_1,a_2,\dots)$, there exists an element
$\Delta_A(a) \in \mathbb{W}(\mathbb{W}(A)$ whose image by the ghost
map
$$w \colon \mathbb{W}(\mathbb{W}(A)) \to \mathbb{W}(A)^{\N}$$
is the sequence $\langle F_n(a) \mid n \in \N\rangle$. It follows from
Lemma~\ref{invertible} that, in this case, the ghost map is injective,
so the element $\Delta_A(a)$ necessarily is unique. Now
Lemmas~\ref{dwork} and~\ref{pfrobenius} show that the sequence
$\langle F_n(a) \mid n \in \N\rangle$ is in the image of the ghost map
if and only if for every prime number $p$ and $n \in p\N$, the
congruence
$$F_n(a) \equiv F_p(F_{n/p}(a)) \hskip5mm 
\text{modulo $p^{v_p(n)}\mathbb{W}(A)$}$$
holds. But in fact equality holds by
Lemma~\ref{variousrelations}~(iv), so we conclude that the 
desired element $\Delta_A(a)$ with $w_n(\Delta_A(a)) = F_n(a)$
exists. Hence, there exists a unique natural ring homomorphism
$\Delta$ such that $w_n \circ \Delta = F_n$ for every $n \in
\N$. Finally, one readily verifies the commutativity of the two
diagrams in the statement by evaluating the corresponding maps in
ghost coordinates.
\end{proof}

\begin{remark}\label{deltaremark}The map
$\Delta_n \colon \mathbb{W}(A) \to \mathbb{W}(A)$ given by the $n$th
Witt component of the map $\Delta$ is generally not a ring
homomorphism. For example, for a prime number $p$, the map $\Delta_p$
is the unique natural solution to the equation
$$F_p(a) = a^p + p\Delta_p(a).$$
We also note that the map $\Delta$ has the property that for all $a \in A$,
$$\Delta([a]) = [[a]].$$
Indeed, we may assume that $A = \Z[a]$, in which case the ghost map is
injective, and applying $w_n$ on both sides, we get $F_n([a]) = [a]^n$
which holds by Lemma~\ref{variousrelations}~(v).
\end{remark}

The natural transformation $\Delta$ in
Proposition~\ref{universallambdaoperation} is called the universal
$\lambda$-operation. Using it, we may restate Grothendieck's
definition of a $\lambda$-ring from~\cite{grothendieck1} as follows.

\begin{definition}\label{lambdaring}A $\lambda$-ring is a pair
$(A,\lambda)$ of a ring $A$ and a ring homomorphism $\lambda
\colon A \to \mathbb{W}(A)$ that makes the following diagrams
commute.
$$\begin{xy}
(0,7)*+{ A }="11";
(20,7)*+{ \mathbb{W}(A) }="12";
(45,7)*+{ \mathbb{W}(\mathbb{W}(A)) }="13";
(68,7)*+{ \mathbb{W}(A) }="14";
(20,-7)*+{ A }="22";
(45,-7)*+{ \mathbb{W}(A) }="23";
(68,-7)*+{ A }="24";
{\ar_-{\epsilon_A} "11";"12";};
{\ar@{=}@<+.25ex> "11";"22";};
{\ar_-{\lambda} "12";"22";};
{\ar_-{\Delta_A} "13";"14";};
{\ar_-{\mathbb{W}(\lambda)} "13";"23";};
{\ar_-{\lambda} "14";"24";};
{\ar_-{\lambda} "23";"24";};
\end{xy}$$
A morphism of $\lambda$-rings $f \colon (A,\lambda_A) \to
(B,\lambda_B)$ is a ring homomorphism $f \colon A \to B$ with the
property that $\lambda_B \circ f = \mathbb{W}(f) \circ \lambda_A$. 
\end{definition}

If $(A,\lambda)$ is a $\lambda$-ring $(A,\lambda)$, then we write
$\lambda_n \colon A \to A$ for the map that to $a$ assigns the
$n$th Witt component $\lambda_n(a)$ of the Witt vector
$\lambda(a)$. The map $\lambda_n$ is generally neither additive nor
multiplicative.

\begin{remark}\label{lambdaringremark}We recall the translation
between the above definition of a $\lambda$-ring and the original
definition by Grothendieck as stated in~\cite[Definition~V.2.4]{SGA6}
(or in~\cite{grothendieck1} and~\cite[Section~1]{atiyahtall}, where a
$\lambda$-ring is called a special $\lambda$-ring), 
emphasizing the choices of signs; see
also~\cite[E.2.1]{hazewinkel}. The commutativity of the 
diagrams in Proposition~\ref{universallambdaoperation} express that
the triple $(\mathbb{W}(-),\Delta,\epsilon)$ is a comonad on the
category of commutative rings, and the commutativity of the diagrams
in Definition~\ref{lambdaring} express that the pair $(A,\lambda)$ is a
coalgebra over this comonad. Similarly, in the original definition, a
$\lambda$-ring is defined to
be a coalgebra $(A,\lambda_t)$ over the comonad
$(\Lambda(-)_{++},\Delta_t,\epsilon_t)$, where $\Lambda(-)_{++}$ is the
functor from the category of commutative rings to itself defined in
Remark~\ref{signs}; $\epsilon_{t,A} \colon \Lambda(A)_{++} \to
A$ is the natural ring homomorphism defined by
$\epsilon_{t,A}(1+a_1t+\dots) = a_1$; and 
$\Delta_{t,A} \colon \Lambda(A)_{++} \to
\Lambda(\Lambda(A)_{++})_{++}$ is the unique natural ring homomorphism
that is a section of $\epsilon_{t,\Lambda(A)_{++}}$ and satisfies that
for all $a \in A$,
$$\Delta_{t,A}(1+at) = 1 + (1+at_2)t_1.$$
We claim that the natural ring isomorphism $\gamma_{++}$ is an
isomorphism of comonads from $(\mathbb{W}(-),\Delta,\epsilon)$ to
$(\Lambda(-)_{++},\Delta_t,\epsilon_t)$ in the sense that if
$(A,\lambda)$ is a coalgebra over the former comonad, then
$(A,\gamma_{++} \circ \lambda)$ is a coalgebra over the latter
comonad. Indeed, this follows immediately from the above
characterization of $\Delta_t$ and from the formula
$\Delta_A([a]) = [[a]]$ from Remark~\ref{deltaremark}. This shows that
the two definitions of a $\lambda$-ring agree. Finally, we remark that
if $(A,\lambda)$ is a $\lambda$-ring and if we expand $\lambda_t =
\gamma_{++} \circ \lambda$ as
$$\lambda_t(a) = 1 + \lambda^{\!1}(a)t + \lambda^{\!2}(a)t^2 + \dots +
\lambda^n(a)t^{\!n} + \cdots,$$
then $\lambda^{\!n} \colon A \to A$ is called the $n$th exterior operation
associated with $(A,\lambda)$; it should not be confused with
$\lambda_n \colon A \to A$. Similarly, if we expand
$\sigma_t = \gamma \circ \lambda$ as
$$\sigma_t(a) = 1 + \sigma^1(a)t + \sigma^2(a)t^2 + \dots +
\sigma^n(a)t^n + \cdots,$$
then $\sigma^n \colon A \to A$ is called the $n$th symmetric
operation associated with $(A,\lambda)$.
\end{remark}

\begin{definition}\label{adamsoperations}Let $(A,\lambda)$ be a
$\lambda$-ring. The associated $n$th Adams operation is the composite
ring homomorphism $\psi_n = w_n \circ \lambda \colon A \to A$.
\end{definition}

We note that, by Proposition~\ref{changeofcoordinates}, the series
$\psi_t(a) = \sum_{n \geqslant 1}\psi_n(a)t^n$ is given by either one
of the following formulas which are, perhaps, more familiar;
$$\psi_t(a) = t \frac{d}{dt}\log\sigma_t(a), \hskip10mm
\psi_{-t}(a) = -t \frac{d}{dt}\log\lambda_t(a).$$
We recall the following standard properties of the Adams
operations, and mention Wilkerson's
result~\cite[Proposition~1.2]{wilkerson} that, if $A$ is a ring flat
over $\Z$ equipped with a family of ring endomorphisms $\psi_n$
satisfying~(i)--(iii) below, then there is a unique $\lambda$-ring
structure on $A$ for which the $\psi_n$ are the associated Adams
operators.

\begin{lemma}\label{adamdsoperationslemma}Let $(A,\lambda)$ be a
$\lambda$-ring. The associated Adams operations satisfy that
\begin{enumerate}
\item[{\rm (i)}] the map $\psi_1$ is the identity map of $A$;
\item[{\rm (ii)}] for all positive integers $m$ and $n$, $\psi_m \circ
  \psi_n = \psi_{mn}$; and
\item[{\rm (iii)}] for every prime number $p$ and $a \in A$, $\psi_p(a)
  \equiv a^p$ modulo $pA$.
\end{enumerate}
\end{lemma}

\begin{proof}The properties~(i) and~(iii) follow immediately from the
definitions, and~(ii) follows from the identities
$$\begin{aligned}
\psi_m \circ \psi_n
{} & = w_m \circ \lambda \circ w_n \circ \lambda
{} = w_m \circ w_n \circ \mathbb{W}(\lambda) \circ \lambda \cr
{} & = w_m \circ w_n \circ \Delta \circ \lambda
{} = w_m \circ F_n \circ \lambda 
{} = w_{mn} \circ \lambda
{} = \psi_{mn}. \cr
\end{aligned}$$
Here, the second identity follows from the naturality of
$w_n$; the third identity from the definition of a $\lambda$-ring; the
forth identity from the definition of the map $\Delta$; and the fifth 
identity from the definition of the map $F_n$.
\end{proof}

Finally, we recall the following general theorem which was proved
independently by
Borger~\cite[Theorem~B]{borger},~\cite[Corollary~15.4]{borger2} and  van~der~Kallen~\cite[Theorem~2.4]{vanderkallen1}.

\begin{theorem}\label{borgervanderkallentheorem}Let $f \colon A \to B$ be an
\'{e}tale morphism, let $S$ be a finite truncation set, and let $n$ be
a positive integer. Then the induced morphism
$$\begin{xy}
(0,0)*+{ \mathbb{W}_S(A) }="1";
(26,0)*+{ \mathbb{W}_S(B) }="2";
{ \ar^-{\mathbb{W}_S(f)} "2";"1";};
\end{xy}$$
is \'{e}tale and the square diagram
$$\begin{xy}
(0,7)*+{ \mathbb{W}_S(A) }="11";
(26,7)*+{ \mathbb{W}_S(B) }="12";
(0,-7)*+{ \mathbb{W}_{S/n}(A) }="21";
(26,-7)*+{ \mathbb{W}_{S/n}(B) }="22";
{\ar^-{\mathbb{W}_S(f)} "12";"11";};
{\ar^-{F_n} "21";"11";};
{\ar^-{\mathbb{W}_{S/n}(f)} "22";"21";};
{\ar^-{F_n} "22";"12";};
\end{xy}$$
is a cocartesian square of rings. \qed
\end{theorem}

We remark that in loc.~cit., the
Theorem~\ref{borgervanderkallentheorem} is stated only for the
finite truncation sets $\langle n \rangle$ that consist of all
divisors of a given positive integer $n$. However, as explained
in~\cite[Section~9.5]{borger}, the case of a general finite truncation
set readily follows from the special case.

\section{Modules and derivations over $\lambda$-rings}\label{modulesandderivationssection}

In general, if $\mathscr{C}$ is a category in which finite limits
exist and if $X$ is an object of $\mathscr{C}$, then Beck, in his
thesis~\cite[Definition~5]{beck}, defines the category of $X$-modules
to be the category $(\mathscr{C}/X)_{\operatorname{ab}}$ of abelian group
objects in the category over $X$. He also defines the derivations from
$X$ to the $X$-module $(Y/X,+_Y,0_Y,-_Y)$ to be the set
$$\operatorname{Der}(X,(Y/X,+_Y,0_Y,-_Y)) = \Hom_{\mathscr{C}/X}(X/X,Y/X)$$
of morphisms in the category $\mathscr{C}/X$ equipped with the abelian
group structure induced by the abelian group object structure on
$Y/X$. In this section, we identify and study these notions in the
case of the category $\mathscr{A}_{\lambda}$ of $\lambda$-rings.

We recall that, in general, an adjunction from a category
$\mathscr{C}$ to a category $\mathscr{D}$ is a quadruple
$(F,G,\epsilon,\eta)$ of functors $F \colon \mathscr{C} \to
\mathscr{D}$ and $G \colon \mathscr{D} \to \mathscr{C}$ and 
natural transformations $\epsilon \colon F \circ G \Rightarrow
\id_{\mathscr{D}}$ and $\eta \colon \id_{\mathscr{C}} \Rightarrow
G \circ F$ such that the following composite natural transformations
are equal to the respective identity natural transformations,
$$\xymatrix{
{ F } \ar@{=>}[r]^-(.47){F \circ \eta} &
{ F \circ G \circ F } \ar@{=>}[r]^-(.47){\epsilon \circ F} &
{ F, } &
{ G } \ar@{=>}[r]^-(.47){\eta \circ G} &
{ G \circ F \circ G } \ar@{=>}[r]^-(.47){G \circ \epsilon} &
{ G; } \cr
}$$
compare~\cite[Theorem~IV.1.2]{maclane}. We refer to this requirement
by saying that the triangle identities hold.
The natural
transformations $\epsilon$ and $\eta$ are called the counit and the
unit of the adjunction, respectively, and the adjunction is said to be
an adjoint equivalence if they both are isomorphisms. A
functor $G \colon \mathscr{D} \to \mathscr{C}$ is said to admit a
left adjoint, if there exists an adjunction $(F,G,\epsilon,\eta)$ with
$G$ as its second component, and in this case, the functor $F$ is said
to be a left adjoint of the functor $G$. If $(F',G,\epsilon',\eta')$
is another such adjunction, then the composite
$$\xymatrix{
{ F } \ar@{=>}[r]^-(.47){F \circ \eta'} &
{ F \circ G \circ F' } \ar@{=>}[r]^-(.47){\epsilon \circ F'} &
{ F' } \cr
}$$
is the unique natural transformation $\sigma \colon F \Rightarrow F'$
with the property that the diagrams
$$\xymatrix{
{ F \circ G } \ar@{=>}[r]^-(.47){\epsilon} \ar@{=>}[d]^-{\,\sigma
  \circ G} &
{ \id_{\mathscr{D}} } \ar@{=}[d] &&
{ \id_{\mathscr{C}} } \ar@{=>}[r]^-(.47){\eta} \ar@{=}[d] &
{ G \circ F }  \ar@{=>}[d]^-{\,\sigma \circ G} \cr
{ F' \circ G } \ar@{=>}[r]^-(.47){\epsilon'} &
{ \id_{\mathscr{D}} } &&
{ \id_{\mathscr{C}} } \ar@{=>}[r]^-(.47){\eta'} &
{ G \circ F' } \cr
}$$
commute and is an isomorphism; see~\cite[Theorem~IV.7.2]{maclane}. In
this sense, a left adjoint of a functor $G$, if it exits, is unique, up
to unique isomorphism. Similar statements hold for right adjoint
functors.

Let $\mathscr{A}$ be the category of rings. We always assume rings to
be commutative and unital, unless otherwise stated. Given a ring $A$,
we define an adjunction $(F,G,\epsilon,\eta)$ from the category 
$(\mathscr{A}/A)_{\operatorname{ab}}$ of abelian group objects in the
over-category $\mathscr{A}/A$ to the category $\mathscr{M}(A)$ of
$A$-modules in the usual sense, following
Beck~\cite[Example~8]{beck}. So let $f \colon B \to A$ be an object of 
$\mathscr{A}/A$, and let
$$\xymatrix{
{ B \times_A^{\phantom{A}}B } \ar[r]^-{+_B} \ar[d] &
{ B } \ar[d]^{f} &&
{ A } \ar[r]^-{0_B} \ar@{=}[d] &
{ B } \ar[d]^{f} &&
{ B } \ar[r]^-{-_B} \ar[d]^{f} &
{ B } \ar[d]^{f} \cr
{ A } \ar@{=}[r] &
{ A } &&
{ A } \ar@{=}[r] &
{ A } &&
{ A } \ar@{=}[r] &
{ A } \cr
}
$$
be abelian group object structure maps. The functor $F$ associates to
the abelian group object $(f,+_B,0_B,-_B)$ the $A$-module $M$ given by the
kernel of $f$ with the $A$-module structure $a \cdot x  =
0_B(a)x$. Conversely, if $M$ is an $A$-module, then we let
$A \ltimes M$ be the ring given by the direct sum $A \oplus M$
equipped with the multiplication 
$$(a,x) \cdot (a',x') = (aa',ax'+a'x)$$
and define $G(M)$ to be the abelian group object $(f,+,0,-)$, where
$f \colon A \ltimes M \to A$ is the canonical projection, and 
the abelian group object structure maps are given 
by $(a,x)+(a,x') = (a,x+x')$, $0(a) = (a,0)$, and $-(a,x) =
(a,-x)$. Finally, we define $\epsilon \colon G \circ F \Rightarrow
\id$ by $\epsilon(a,x) = 0_B(a)+x$ and $\eta \colon \id \Rightarrow F
\circ G$ by $\eta(x) = (0,x)$. For later use, we include a proof of the
following result of Beck~\cite[Example~8]{beck}.

\begin{lemma}\label{adjunctionlemma1}If $A$ is a ring, then the
quadruple $(F,G,\eta,\epsilon)$ defined above is an adjoint
equivalence of categories from $(\mathscr{A}/A)_{\operatorname{ab}}$
to $\mathscr{M}(A)$.
\end{lemma}

\begin{proof}It is clear that $\eta$ is well-defined and a natural
isomorphism, and it is also clear that $\epsilon$ is a natural
isomorphism of underlying additive groups. We must show that
$\epsilon$ is a multiplicative map and a map of abelian group
objects; we first consider the latter statement. So we fix an object
$(f \colon B \to A, +_B, 0_B, -_B)$ of
$(\mathscr{A}/A)_{\operatorname{ab}}$ and let $M = \ker(f)$ with the
$A$-module structure defined above. By definition, we have
$\epsilon(a,0) = 0_B(a)$ which  shows that $\epsilon$ preserves zero
maps. To see that $\epsilon$ preserves addition maps, we first note
that, since $+_B$ is a ring homomorphism,
$$(u+v)+_B(u'+v') =  (u+_Bu')+(v+_Bv')$$
for all $(u,u'),(v,v') \in B \times_AB$. In particular, if $x,y \in
M$, then
$$x+_By = (x+0)+_B(0+y) = (x+_B0)+(0+_By) = x+y,$$
where we also use that $0= 0_B(0)$ is a common zero element for the
two compositions $+$ and $+_B$ on $M$. We therefore conclude that for
all $a \in A$ and $x,x' \in M$, 
$$\begin{aligned}
{} &\epsilon(a,x) +_B \epsilon(a,x') = (0(a) + x) +_B (0(a) + x') \cr
{} & = (0(a) +_B 0(a)) + (x +_B x') = 0(a) + (x + x') =
\epsilon(a,x+x'), \cr
\end{aligned}$$
as desired. We have showed that $\epsilon$ is compatible with the zero
and addition maps; but then it is also compatible with negation maps.

It remains to prove that the map $\epsilon$ is multiplicative, or
equivalently, that $M \subset B$ is a square-zero ideal. Since 
$+_B \colon B \times_AB \to B$ is a ring homomorphism, we have that
$uv +_B u'v' = (u+_Bu')(v+_Bv')$, for all
$(u,u'),(v,v') \in B \times_AB$. In particular,
$$xy + x'y' = (x+y)(x'+y')$$
for all $x,x',y,y' \in M$, since $+_B = +$ on $M$. Taking 
$y = x' = 0$, we find that $xy' = 0$ for all $x,y' \in M$ as
desired. This completes the proof.
\end{proof}

We will prove the analogous statement for $\lambda$-rings in
Proposition~\ref{adjunctionlemma2} below, but first we examine
the Witt vectors of $A \ltimes M$. The polynomials $s_n$,
$p_n$, and $i_n$ that define the 
sum, product, and opposite in the ring of Witt vectors all have
constant term zero. Therefore, the ring of Witt vectors is defined
also for non-unital rings. Moreover, modulo terms of higher degree,
these polynomials are congruent to $a_n+b_n$, $a_nb_n$, and $-a_n$,
respectively, as one readily proves by induction. Therefore, if $M$ is
an abelian group considered as a non-unital ring with zero
multiplication, then the non-unital ring $\mathbb{W}_S(M)$ has zero
multiplication and its underlying additive group is equal to $M^S$
with componentwise addition. In the same way, one shows that 
the polynomials $f_{n,e}$ and $d_{m,e}$ that define the $n$th
Frobenius and the universal $\lambda$-operation all have constant term
zero and that they are congruent to $na_{ne}$ and $a_{me}$,
respectively, modulo terms of higher degree. Therefore, for $M$ as above,
the map $F_n \colon \mathbb{W}_S(M) \to \mathbb{W}_{S/n}(M)$ takes $(x_m \mid
m \in S)$ to $(nx_{ne} \mid e \in S/n)$ and the map $\Delta_M \colon
\mathbb{W}(M) \to \mathbb{W}(\mathbb{W}(M))$ takes $(x_m \mid m \in
\N)$ to $((x_{me} \mid e \in \N) \mid m \in \N)$.

\begin{lemma}\label{directsumring}Let $S$ be truncation set, let $A$
be a ring, and let $M$ be an $A$-module. The canonical inclusions
induce a ring isomorphism
$$\operatorname{in}_{1*}+\operatorname{in}_{2*} \colon
\mathbb{W}_S(A) \ltimes \mathbb{W}_S(M) \to
\mathbb{W}_S(A \ltimes M),$$
provided that $\mathbb{W}_S(M)$ is given the $\mathbb{W}_S(A)$-module
structure, where for $a \in \mathbb{W}_S(A)$ and $x \in
\mathbb{W}_S(M)$, $ax \in \mathbb{W}_S(M)$ has Witt components $(ax)_n
= w_n(a)x_n$.
\end{lemma}

\begin{proof}We consider the following diagram of rings and ring
homomorphisms, whose underlying diagram of additive groups is
split-exact.
$$\xymatrix{
{ 0 } \ar[r] &
{ M } \ar[r]^(.42){\operatorname{in}_2} &
{ A \ltimes M } \ar@<-.7ex>[r]_(.62){\pr_1} &
{ A } \ar[r] \ar@<-.7ex>[l]_(.38){\operatorname{in}_1} &
{ 0 }
}$$
It induces the following diagram of rings and ring homomorphisms,
whose underlying diagram of additive groups again is split-exact.
$$\xymatrix{
{ 0 } \ar[r] &
{ \mathbb{W}_S(M) } \ar[r]^(.42){\operatorname{in}_{2*}} &
{ \mathbb{W}_S(A \ltimes M) } \ar@<-.7ex>[r]_(.62){\pr_{1*}} &
{ \mathbb{W}_S(A) } \ar[r] \ar@<-.7ex>[l]_(.38){\operatorname{in}_{1*}} &
{ 0 }
}$$
It follows that the map of the statement is a ring isomorphism, if
$\mathbb{W}_S(M)$ is given the $\mathbb{W}_S(A)$-module structure such
that
$\operatorname{in}_{2*}(ax) = \operatorname{in}_{1*}(a)
\operatorname{in}_{2*}(x)$,
for all $a \in \mathbb{W}_S(A)$ and $x \in \mathbb{W}_S(M)$.  It
remains to prove that $ax$ is equal to the Witt vector $y$ with $n$th
component $w_n(a)x_n$. Since every ring admits a surjective ring
homomorphism from a torsion free ring, we may assume that $A$ and $M$
are both torsion free. Moreover, since the ghost map is injective in 
this case, it will suffice to show that $w_n(ax) = w_n(y)$, or
equivalently, that $\operatorname{in}_2(w_n(ax) ) =
\operatorname{in}_2(w_n(y))$, for all $n \geqslant 1$. Now, since
$w_n$ is a natural ring homomorphism, we find that for all $n
\geqslant 1$,
$$\begin{aligned}
\operatorname{in}_2(w_n(ax))
{} & = w_n(\operatorname{in}_{2*}(ax))
{} = w_n(\operatorname{in}_{1*}(a)\operatorname{in}_{2*}(x))
{} = w_n(\operatorname{in}_{1*}(a))w_n(\operatorname{in}_{2*}(x)) \cr
{} & = \operatorname{in}_1(w_n(a))\operatorname{in}_2(w_n(x)) 
{} = \operatorname{in}_2(w_n(a)w_n(x)) 
{} = \operatorname{in}_2(nw_n(a)x_n) \cr
{} & = \operatorname{in}_2(ny_n) 
{} = \operatorname{in}_2(w_n(y)) \cr
\end{aligned}$$
as desired. Here the fifth equality follows from the definition of the
multiplication on the ring $A \ltimes M$.
\end{proof}

\begin{addendum}\label{explicitisomorphism}Let $S$ be a truncation
set, let $A$ be a ring, and let $M$ be an $A$-module. If $a \in
\mathbb{W}_S(A)$ and $x \in \mathbb{W}_S(M)$, then the components
$b_n = (a_n,y_n)$ of the Witt vector $b = \operatorname{in}_{1*}(a) +
\operatorname{in}_{2*}(x) \in \mathbb{W}_S(A \ltimes M)$ satisfy that
for all $n \in S$,
$$\sum_{e \mid n} a_e^{(n/e)-1}y_e = x_n.$$
\end{addendum}

\begin{proof}We may assume that $A$ and $M$ are torsion free and
proceed to calculate $w_n(b)$ in two different ways. First, since
$w_n$ is a natural ring homomorphism, we have
$$\begin{aligned}
w_n(b) 
{} & = w_n(\operatorname{in}_{1*}(a))+w_n(\operatorname{in}_{2*}(x))
= \operatorname{in}_1(w_n(a))+\operatorname{in}_2(w_n(x)) \cr
{} & = (w_n(a),w_n(x)) = (\sum_{e \mid n} ea_e^{n/e},nx_n). \cr
\end{aligned}$$
Second, by the definition of the multiplication in $A \ltimes M$, we
have
$$w_n(b) = 
\sum_{e \mid n} eb_e^{n/e} = \sum_{e \mid n} e(a_e,y_e)^{n/e} =
(\sum_{e \mid n} ea_e^{n/e}, \sum_{e \mid n}na_e^{(n/e)-1}y_e).$$
The stated formula follows as $M$ was assumed to be torsion free.
\end{proof}

\begin{example}\label{primeexample}Let $p$ be a prime number. Then
$y_p = x_p - a_1^{p-1}x_1$.
\end{example}

In general, if $f \colon A \to B$ is a ring homomorphism and if $M$
and $N$ are modules over $A$ and $B$, respectively, then we define an
$f$-linear map $h \colon M \to N$ to be an additive map such that
$h(ax) = f(a)h(x)$, for all $a \in A$ and $x \in M$. In the following,
given an $A$-module $M$ and a truncation set $S \subset \N$, we write
$\mathbb{W}_S(M)$ for the $\mathbb{W}_S(A)$-module given by the set
$M^S$ with componentwise addition and with the scalar multiplication
of $a \in \mathbb{W}_S(A)$ and $x \in \mathbb{W}_S(M)$ defined by to
be the element $ax \in \mathbb{W}_S(M)$ with
$$(ax)_n = \psi_{A,n}(a)x_n$$
for all $n \in S$; compare Definition~\ref{adamsoperations} and
Lemma~\ref{directsumring}. We remark that if $M$ is the ring $A$
considered a module over itself via multiplication, then the
$\mathbb{W}_S(A)$-module $\mathbb{W}_S(M)$ defined above usually is
not the same as the ring $\mathbb{W}_S(A)$ considered as a module over
itself via multiplication. To avoid confusion, we will use
$\mathbb{W}_S(A)$ to indicate the ring of Witt vectors only and will
not use it indicate either module over this ring.

\begin{definition}\label{lambdamodule}Let $(A,\lambda_A)$ be a
$\lambda$-ring. An $(A,\lambda_A)$-module is a pair
$(M,\lambda_M)$ of an $A$-module $M$ and a $\lambda_A$-linear map
$$\lambda_M \colon M \to \mathbb{W}(M)$$
with the property that the diagrams
$$\begin{xy}
(0,7)*+{ M }="11";
(20,7)*+{ \mathbb{W}(M) }="12";
(45,7)*+{ \mathbb{W}(\mathbb{W}(M)) }="13";
(68,7)*+{ \mathbb{W}(M) }="14";
(20,-7)*+{ M }="22";
(45,-7)*+{ \mathbb{W}(M) }="23";
(68,-7)*+{ M }="24";
{\ar_-{\epsilon_M} "11";"12";};
{\ar@{=}@<+.25ex> "11";"22";};
{\ar_-{\lambda_M} "12";"22";};
{\ar_-{\Delta_M} "13";"14";};
{\ar_-{\mathbb{W}(\lambda)} "13";"23";};
{\ar_-{\lambda_M} "14";"24";};
{\ar_-{\lambda_M} "23";"24";};
\end{xy}$$
commute. A morphism $h \colon (M,\lambda_M) \to (N,\lambda_N)$ of
$(A,\lambda)$-modules is an $A$-linear map $h \colon M \to N$ such
that $\lambda_N \circ h = \mathbb{W}(h) \circ \lambda_M$.
\end{definition}

\begin{remark}\label{lambdamoduleremark}Let $(A,\lambda_A)$ be a
$\lambda$-ring, $M$ an $A$-module, and $\lambda_M \colon M \to
\mathbb{W}(M)$ a map. Then $(M,\lambda_M)$ is an
$(A,\lambda_A)$-module if and only if the components $\lambda_{M,n}
\colon M \to M$ are $\psi_{A,n}$-linear and satisfy $\lambda_{M,1} =
\id_M$ and $\lambda_{M,m} \circ \lambda_{M,n} = \lambda_{M,mn}$, for
all $m,n \in \N$. Hence, we may identify the category
$\mathscr{M}(A,\lambda_A)$ with the category
$\mathscr{M}(A^{\psi}[\N])$ of left modules over the twisted monoid
algebra $A^{\psi}[\N]$ by associating to the $(A,\lambda_A)$-module
$(M,\lambda_M)$ the left $A^{\psi}[\N]$-module given by the $A$-module
$M$ and with $n \in \N$ acting through the map
$\lambda_{M,n} \colon M \to M$. In particular, the category
$\mathscr{M}(A,\lambda_A)$ is abelian.
\end{remark}

\begin{example}\label{freelambdamodule}Let $(A,\lambda_A)$ be a
$\lambda$-ring. The functor that to an $(A,\lambda_A)$-module
$(M,\lambda_M)$ assigns the underlying set of $M$ has a left adjoint
functor that to the set $S$ assigns the free $(A,\lambda_A)$-module
$(F(S),\lambda_{F(S)})$ defined as follows. The $A$-module $F(S)$ is
defined to be 
the free $A$-module generated by the symbols $\lambda_{F(S),n}(s)$,
where $s \in S$ and $n \in \N$, and $\lambda_{F(S)} \colon F(S) \to
\mathbb{W}(F(S))$ is defined to be the map with $m$th component 
$$\lambda_{F(S),m}( \sum a_{s,n}\lambda_{F(S),n}(s)) 
= \sum \psi_{A,m}(a_{s,n})\lambda_{F(S),mn}(s).$$
It follows from Remark~\ref{lambdamoduleremark} that the pair
$(F(S),\lambda_{F(S)})$ is an $(A,\lambda_A)$-module. The unit of the
adjunction maps $s \in S$ to $\lambda_{F(S),1}(s) \in F(S)$, and the
counit of the adjunction maps $\sum a_{x,n}\lambda_{F(M),n}(x) \in
F(M)$ to $\sum  a_{x,n}\lambda_{M,n}(x) \in M$. It is straightforward
to verify that the triangle identities hold.
\end{example}

\begin{example}\label{lambdaringwithadamsoperations}If $(A,\lambda_A)$
is a $\lambda$-ring, then there is an $(A,\lambda_A)$-module
$(M,\lambda_M)$ defined by setting $M = A$ and $\lambda_{M,n} =
\psi_{A,n}$. This $(A,\lambda_A)$-module is not a free
$(A,\lambda_A)$-module in the sense of Example~\ref{freelambdamodule},
except in trivial cases. We warn the reader that the pair
$(A,\lambda_A)$ is typically not an $(A,\lambda_A)$-module, let alone
a free $(A,\lambda_A)$-module; compare the discussion preceding
Definition~\ref{lambdamodule}. 
\end{example}

\begin{example}\label{algebraicktheory}Let $A$ be a ring, unital and
commutative, and let $K_*(A)$ be the graded ring given by the Quillen
$K$-groups. The ring $K_0(A)$ has a canonical $\lambda$-ring structure
defined by Grothendieck~\cite{grothendieck1}, and for all 
$q \geqslant 1$, the group $K_q(A)$ has a canonical structure of a
module over this $\lambda$-ring defined by Kratzer~\cite{kratzer} and
Quillen~\cite{hiller}. The $(K_0(A),\lambda_{K_0(A)})$-module
structure maps are given by
$$\lambda_{K_q(A),n}^{\phantom{n}} = (-1)^{n-1}\lambda_{K_q(A)}^n
\colon K_q(A) \to K_q(A)$$
with $\lambda_{K_q(A)}^n$ defined
in~\cite[Th\'{e}or\`{e}me~5.1]{kratzer}. 
\end{example}

Let $U \colon \mathscr{A}_{\lambda} \to \mathscr{A}$ be the forgetful
functor from the category of $\lambda$-rings to the category of rings
that to a $\lambda$-ring $(A,\lambda)$ assigns the underlying ring
$A$. It admits the right adjoint functor $R \colon \mathscr{A} \to
\mathscr{A}_{\lambda}$ defined by $R(A) = (\mathbb{W}(A),\Delta_A)$
with the counit and unit maps defined by $\lambda \colon (A,\lambda)
\to (\mathbb{W}(A),\Delta_A)$ and $\epsilon_A \colon \mathbb{W}(A) \to
A$, respectively. The forgetful functor $U$ also admits a left
adjoint, but this will not be relevant for the discussion below. Since
$\mathbb{W}(-)$ preserves limits, the forgetful functor $U$ creates
limits. Indeed, if $\{(A_i,\lambda_i)\}$ is a diagram of
$\lambda$-rings and if $\{ p_i \colon A \to A_i \}$ is a limit in
$\mathscr{A}$ of the diagram $\{A_i\}$, then $\{ \mathbb{W}(p_i)
\colon \mathbb{W}(A) \to \mathbb{W}(A_i) \}$ is a limit in
$\mathscr{A}$ of the diagram $\{\mathbb{W}(A_i)\}$.  Therefore, we
conclude that the pair $(A,\lambda)$, where $\lambda \colon A \to
\mathbb{W}(A)$ is defined to be the unique map with $i$th component
$\lambda_i \circ p_i \colon A \to \mathbb{W}(A_i)$, is a
$\lambda$-ring and that the family $\{ p_i \colon (A,\lambda) \to
(A_i,\lambda_i) \}$ is a limit in $\mathscr{A}_{\lambda}$ of the given
diagram. It follows that that for $(A,\lambda_A)$ a $\lambda$-ring, we
obtain an adjunction 
$$\begin{xy}
(-13,0)*+{ \mathscr{A}_{\lambda}/(A,\lambda_A) }="a";
(13,0)*+{ \mathscr{A}/A, }="b";
{ \ar@<.7ex>^-{U_{(A,\lambda_A)}} "b";"a";};
{ \ar@<.7ex>^-{R_{(A,\lambda_A)}} "a";"b";};
\end{xy}$$
where the left adjoint functor $U_{(A,\lambda_A)}$ takes $f \colon
(B,\lambda_B) \to (A,\lambda_A)$ to $f \colon B \to A$, the
right adjoint functor $R_{(A,\lambda_A)}$ takes $f \colon B \to A$ to
$p_2 \colon (C,\lambda_C) \to (A,\lambda_A)$, for
$$\xymatrix{
{ (C,\lambda_C) } \ar[r]^-{p_1} \ar[d]^{p_2} &
{ (\mathbb{W}(B),\Delta_B) } \ar[d]^-{\mathbb{W}(f)} \cr
{ (A,\lambda_A) } \ar[r]^-{\lambda_A} &
{ (\mathbb{W}(A),\Delta_A) } \cr
}$$
a choice of a pullback, the counit is $\epsilon_B \circ p_1$, and the
unit is the unique map with components $\lambda_B$ and $f$. Since
the functors $U_{(A,\lambda_A)}$ and $R_{(A,\lambda_A)}$ both preserve
limits, this adjunction, in turn, induces an adjunction
$$\begin{xy}
(-16,0)*+{ (\mathscr{A}_{\lambda}/(A,\lambda_A))_{\operatorname{ab}} }="a";
(16,0)*+{ (\mathscr{A}/A)_{\operatorname{ab}} }="b";
{ \ar@<.7ex>^-{U_{(A,\lambda_A)}} "b";"a";};
{ \ar@<.7ex>^-{R_{(A,\lambda_A)}} "a";"b";};
\end{xy}$$
between the associated categories of abelian group
objects. Corresponding to this, we have the following adjunction
$$\begin{xy}
(-12,0)*+{ \mathscr{M}(A,\lambda_A) }="a";
(12,0)*+{ \mathscr{M}(A) }="b";
{ \ar@<.7ex>^-{U'} "b";"a";};
{ \ar@<.7ex>^-{R'} "a";"b";};
\end{xy}$$
where the left adjoint functor $U'$ takes $(M,\lambda_M)$ to $M$,
the right adjoint functor $R'$ takes $N$ to
$(\lambda_{A*}(\mathbb{W}(N)),\Delta_N)$, and the counit and unit 
maps are defined to be the maps $\epsilon_N \colon
\lambda_{A*}(\mathbb{W}(N)) \to N$ and $\lambda_M \colon (M,\lambda_M)
\to (\lambda_{A*}(\mathbb{W}(M)),\Delta_M)$, respectively. Here we write
$\lambda_{A*}(\mathbb{W}(N))$ for the $\mathbb{W}(A)$-module
$\mathbb{W}(N)$ considered as an $A$-module via $\lambda_A$.

\begin{prop}\label{adjunctionlemma2}Let $(A,\lambda_A)$ be a
$\lambda$-ring. There exists, up to unique isomorphism, a unique
adjunction
$(F^{\lambda},G^{\lambda},\epsilon^{\lambda},\eta^{\lambda})$ from
$(\mathscr{A}_{\lambda}/(A,\lambda_A))_{\operatorname{ab}}$ to
$\mathscr{M}(A,\lambda_A)$ such that, in the following diagram, the
square of left adjoint functors commutes,
$$\begin{xy}
(-17,7)*+{ (\mathscr{A}/A)_{\operatorname{ab}} }="a";
 (17,7)*+{ \mathscr{M}(A) }="b";
(-17,-7)*+{ (\mathscr{A}_{\lambda}/(A,\lambda_A))_{\operatorname{ab}}
}="c";
(17,-7)*+{ \mathscr{M}(A,\lambda_A). }="d";
{ \ar@<.7ex>^-{F} "b";"a";};
{ \ar@<.7ex>^-{G} "a";"b";};
{ \ar@<.7ex>^-{U_{(A,\lambda_A)}} "a";"c";};
{ \ar@<.7ex>^-{R_{(A,\lambda_A)}} "c";"a";};
{ \ar@<.7ex>^-{F^{\lambda}} "d";"c";};
{ \ar@<.7ex>^-{G^{\lambda}} "c";"d";};
{ \ar@<.7ex>^-{U'} "b";"d";};
{ \ar@<.7ex>^-{R'} "d";"b";};
\end{xy}$$
Moreover, the adjunction
$(F^{\lambda},G^{\lambda},\epsilon^{\lambda},\eta^{\lambda})$ is an
adjoint equivalence of categories.
\end{prop}

We remark that, by the uniqueness statement for adjoints, which we
recalled at the beginning of the section, the commutativity of the
square of left adjoint functors in the diagram in
Proposition~\ref{adjunctionlemma2} implies that the corresponding
square of right adjoint functors commutes, up to unique natural
isomorphism.

\begin{proof}If $(f,+,0,-)$ is an object of
$(\mathscr{A}_{\lambda}/(A,\lambda_A))_{\operatorname{ab}}$ with
underlying map of $\lambda$-rings $f \colon (B,\lambda_B) \to
(A,\lambda_A)$, then we define $F^{\lambda}(f,+,-,0)$ to be the pair
$(M,\lambda_M)$ of the kernel $M = F(f)$ of $f$ and the induced map
$\lambda_M \colon M \to \mathbb{W}(M)$ of kernels of the vertical maps
in the following diagram. We note that $U' \circ F^{\lambda} = F \circ
U_{(A,\lambda_A)}$ as stated.
$$\begin{xy}
(-10,7)*+{ B }="a";
(10,7)*+{ \mathbb{W}(B) }="b";
(-10,-7)*+{ A }="c";
(10,-7)*+{ \mathbb{W}(A) }="d";
{ \ar^-{\lambda_B} "b";"a";};
{ \ar^-{f} "c";"a";};
{ \ar^-{\mathbb{W}(f)} "d";"b";};
{ \ar^-{\lambda_A} "d";"c";};
\end{xy}$$
Conversely, if $(M,\lambda_M)$ is an $(A,\lambda_A)$-module,
then we define $G^{\lambda}(M,\lambda_M)$ to be the abelian group
object $G(M)$ in $\mathscr{A}/A$ with the underlying ring
$B = A \ltimes M$ equipped with the $\lambda$-ring structure
$\lambda_B \colon B \to \mathbb{W}(B)$ given by
the composite map
$$\begin{xy}
(-31,0)*+{ \phantom{()} A \ltimes M }="a";
(0,0)*+{ \mathbb{W}(A) \ltimes \mathbb{W}(M) }="b";
(36,0)*+{ \mathbb{W}(A \ltimes M); \phantom{()} }="c";
{ \ar^-{\lambda_A \oplus \lambda_M} "b";"a";};
{ \ar^-{\operatorname{in}_{1*} +\operatorname{in}_{2*}} "c";"b";};
\end{xy}$$
compare Lemma~\ref{directsumring}. To prove that
$G^{\lambda}(M,\lambda_M)$ is well-defined, we must 
show (a)~that $(B,\lambda_B)$ is a $\lambda$-ring; (b)~that the
canonical projection $f \colon (B,\lambda_B) \to (A,\lambda_A)$ is a
$\lambda$-ring homomorphism; and (c)~that the abelian group object
structure maps $+_B$, $0_B$, and $-_B$ on $f \colon B \to A$ are
$\lambda$-ring homomorphisms. First, the map
$\lambda_A \oplus \lambda_M$ is a ring homomorphism, since $\lambda_M$
is a $\lambda_A$-linear map. Moreover, Lemma~\ref{directsumring} shows that also
$\operatorname{in}_{1*}+\operatorname{in}_{2*}$ is a ring
homomorphism, so $\lambda_B$ is a ring homomorphism. To prove~(a), it
remains to show that the diagrams in Definition~\ref{lambdaring} commute.
The left-hand diagram commutes, since $\epsilon_A
\circ \lambda_A = \id_A$ and $\epsilon_M \circ \lambda_M = \id_M$ and
since $\operatorname{in}_{1*} + \operatorname{in}_{2*}$ is the
identity map on the first Witt component. To see that the right-hand
square commutes, we consider the following larger diagram,
$$\begin{xy}
(-47,14)*+{ \mathbb{W}(\mathbb{W}(A \ltimes M)) }="11";
(42, 14)*+{ \mathbb{W}(A \ltimes M) }="13";
(-47,0)*+{ \mathbb{W}(\mathbb{W}(A) \ltimes \mathbb{W}(M)) }="21";
(0,0)*+{ \mathbb{W}(\mathbb{W}(A)) \ltimes \mathbb{W}(\mathbb{W}(M))
}="22";
(42,0)*+{ \mathbb{W}(A) \ltimes \mathbb{W}(M)}="23";
(-47,-14)*+{ \mathbb{W}(A \ltimes M) }="31";
(0,-14)*+{ \mathbb{W}(A) \ltimes \mathbb{W}(M) }="32";
(42,-14)*+{ A \ltimes M. }="33";
{ \ar_-{\Delta_{A \ltimes M}} "11";"13";};
{ \ar_-{\mathbb{W}(\operatorname{in}_{1*}+\operatorname{in}_{2*})} "11";"21";};
{ \ar_-{\operatorname{in}_{1*}+\operatorname{in}_{2*}} "13";"23";};
{ \ar_-{\operatorname{in}_{1*}+\operatorname{in}_{2*}} "21";"22";};
{ \ar_-{\Delta_A \oplus \Delta_M} "22";"23";};
{ \ar_-{\mathbb{W}(\lambda_A \oplus \lambda_M)} "21";"31";};
{ \ar_-{\mathbb{W}(\lambda_A) \oplus \mathbb{W}(\lambda_M)} "22";"32";};
{ \ar_-{\lambda_A \oplus \lambda_M} "23";"33";};
{ \ar_-{\operatorname{in}_{1*}+\operatorname{in}_{2*}} "31";"32";};
{ \ar_-{\lambda_A \oplus \lambda_M} "32";"33";};
\end{xy}$$
Here, the lower right-hand square commutes, since $(A,\lambda_A)$ is a
$\lambda$-ring and since $(M,\lambda_M)$ is an $(A,\lambda_A)$-module,
and the lower left-hand square commutes by the naturality of
$\operatorname{in}_{1*}+\operatorname{in}_{2*}$. To prove the upper
rectangular diagram commutes, it suffices to show that the two
compositions with the $n$th ghost map
$$\xymatrix{
{ \mathbb{W}(\mathbb{W}(A \ltimes M)) } \ar[r]^-{w_n} &
{ \mathbb{W}(A \ltimes M) } \cr
}$$
agree. 
This, in turn, follows from the calculation
$$\begin{aligned}
{} & w_n \circ \Delta_{A \ltimes M} \circ
(\operatorname{in}_{1*}+\operatorname{in}_{2*}) 
{} = F_n \circ (\operatorname{in}_{1*}+\operatorname{in}_{2*}) 
{} = (\operatorname{in}_{1*}+\operatorname{in}_{2*}) \circ (F_n
\oplus F_n) \cr
{} & = (\operatorname{in}_{1*}+\operatorname{in}_{2*}) \circ (w_n \oplus
w_n) \circ (\Delta_A \oplus \Delta_M) \cr
{} & = (\operatorname{in}_{1*}+\operatorname{in}_{2*}) \circ w_n \circ
(\operatorname{in}_{1*}+\operatorname{in}_{2*}) \circ (\Delta_A \oplus
\Delta_M) \cr
{} & {} = w_n \circ
\mathbb{W}(\operatorname{in}_{1*}+\operatorname{in}_{2*}) \circ
(\operatorname{in}_{1*}+\operatorname{in}_{2*}) \circ (\Delta_A \oplus
\Delta_M), \cr
\end{aligned}$$
where the first and third equalities hold by the definition of
$\Delta$, where the second and fourth equalities hold by the
naturality of $\operatorname{in}_{1*}+\operatorname{in}_{2*}$, and
where the fifth equality equality holds by the naturality of
$w_n$. This proves~(a). Next, if $h \colon (M,\lambda_M) \to
(N,\lambda_N)$ is a map of $(A,\lambda_A)$-modules, then the following
diagram commutes,
$$\begin{xy}
(-31,7)*+{ A \ltimes M }="11";
(0,7)*+{ \mathbb{W}(A) \ltimes \mathbb{W}(M) }="12";
(36,7)*+{ \mathbb{W}(A \ltimes M) }="13";
(-31,-7)*+{ A \ltimes N }="21";
(0,-7)*+{ \mathbb{W}(A) \ltimes \mathbb{W}(N) }="22";
(36,-7)*+{ \mathbb{W}(A \ltimes N). }="23";
{ \ar^-{\lambda_A \oplus \lambda_M} "12";"11";};
{ \ar^-{\operatorname{in}_{1*}+\operatorname{in}_{2*}} "13";"12";};
{ \ar^-{\id \oplus h} "21";"11";};
{ \ar^-{\id \oplus h_*} "22";"12";};
{ \ar^-{(\id \oplus h)_*} "23";"13";};
{ \ar^-{\lambda_A \oplus \lambda_N} "22";"21";};
{ \ar^-{\operatorname{in}_{1*}+\operatorname{in}_{2*}} "23";"22";};

\end{xy}$$
Taking $(N,\lambda_N)$ to be the trivial $(A,\lambda_A)$-module,~(b)
follows. We use other instances of this diagram to prove~(c). The maps
$0_B$ and $-_B$ are induced by the $(A,\lambda_A)$-module maps
$0_M \colon (0,\id) \to (M,\lambda_M)$ and
$-_M \colon (M,\lambda_M) \to (M,\lambda_M)$ that map $0$ to $0$ and
$x$ to $-x$, respectively. Hence, the diagram shows that both are
$\lambda$-ring homomorphisms. Finally, if we define
$\lambda_{M \oplus M}$ to be the composite map
$$\begin{xy}
(-33,0)*+{ \phantom{()} M \oplus M }="a";
(0,0)*+{ \mathbb{W}(M) \oplus \mathbb{W}(M) }="b";
(38,0)*+{ \mathbb{W}(M \oplus M). \phantom{()} }="c";
{ \ar^-{\lambda_M \oplus \lambda_M} "b";"a";};
{ \ar^-{\operatorname{in}_{1*} +\operatorname{in}_{2*}} "c";"b";};
\end{xy}$$
then $(M \oplus M, \lambda_{M \oplus M})$ is a direct sum of the
$(A,\lambda_A)$-module $(M,\lambda_M)$ with itself. Now, the addition
map $+_B$ is given by the map
$$\xymatrix{
{ (A \ltimes (M \oplus M),\lambda_{A \ltimes (M \oplus M)}) }
\ar[rr]^-{\id \oplus +_M} &&
{ (A \ltimes M, \lambda_{A \ltimes M}), }
}$$
where $+_M \colon (M \oplus M, \lambda_{M \oplus M}) \to
(M,\lambda_M)$ takes $(x,y)$ to $x+y$. To complete the proof of~(c), we
must verify that $+_M$ is a map of $(A,\lambda_A)$-modules, that is,
that the diagram
$$\begin{xy}
(-31,7)*+{ M \oplus M }="11";
(0,7)*+{ \mathbb{W}(M) \oplus \mathbb{W}(M) }="12";
(36,7)*+{ \mathbb{W}(M \oplus M) }="13";
(-31,-7)*+{ M }="21";
(0,-7)*+{ \mathbb{W}(M) }="22";
(36,-7)*+{ \mathbb{W}(M) }="23";
{ \ar^-{\lambda_M \oplus \lambda_M} "12";"11";};
{ \ar^-{\operatorname{in}_{1*}+\operatorname{in}_{2*}} "13";"12";};
{ \ar^-{+_M} "21";"11";};
{ \ar^-{+_{\mathbb{W}(M)}} "22";"12";};
{ \ar^-{\mathbb{W}(+_M)} "23";"13";};
{ \ar^-{\lambda_M} "22";"21";};
{ \ar@{=} "23";"22";};
\end{xy}$$
commutes. But the left-hand square commutes, as $\lambda_M$ is
an additive map, and the right-hand square commutes, since the addition in
$\mathbb{W}(M)$ is given by adding the Witt components of vectors, so
(c)~follows. This completes the proof that the functor $G^{\lambda}$
is well-defined. We also note that, by construction, we have 
$U_{(A,\lambda_A)} \circ G^{\lambda} = G \circ U'$.

Finally, we claim that there are unique natural isomorphisms
$$\xymatrix{
{ G^{\lambda} \circ F^{\lambda} } \ar@{=>}[r]^-{\epsilon^{\lambda}} &
{ \id } &
{ \id } \ar@{=>}[r]^-{\eta^{\lambda}} &
{ F^{\lambda} \circ G^{\lambda} } \cr
}$$
such that $U_{(A,\lambda_A)}(\epsilon^{\lambda}) = \epsilon \circ
U_{(A,\lambda_A)}$ and $U'(\eta^{\lambda}) = \eta \circ U'$. Indeed,
this amounts to the following diagrams being commutative, where, in
the bottom diagram, $i \colon M \to B$ is the (chosen) of kernel of $f
\colon B \to A$;
$$\begin{xy}
(-31,19)*+{ M }="11";
(0,19)*+{ \mathbb{W}(M) }="12";
(36,19)*+{ \mathbb{W}(M) }="13";
(-31,5)*+{ A \ltimes M }="21";
(0,5)*+{ \mathbb{W}(A) \ltimes \mathbb{W}(M) }="22";
(36,5)*+{ \mathbb{W}(A \ltimes M) }="23";
{ \ar^-{\lambda_M} "12";"11";};
{ \ar@{=} "13";"12";};
{ \ar^-{\operatorname{in}_2} "21";"11";};
{ \ar^-{\operatorname{in}_2} "22";"12";};
{ \ar^-{\operatorname{in}_{2*}} "23";"13";};
{ \ar^-{\lambda_A \oplus \lambda_M} "22";"21";};
{ \ar^-{\operatorname{in}_{1*}+\operatorname{in}_{2*}} "23";"22";};
(-31,-5)*+{ A \ltimes M }="11";
(0,-5)*+{ \mathbb{W}(A) \ltimes \mathbb{W}(M) }="12";
(36,-5)*+{ \mathbb{W}(A \ltimes M) }="13";
(-31,-19)*+{ B }="21";
(0,-19)*+{ \mathbb{W}(B) }="22";
(36,-19)*+{ \mathbb{W}(B). }="23";
{ \ar^-{\lambda_A \oplus \lambda_M} "12";"11";};
{ \ar^-{\operatorname{in}_{1*}+\operatorname{in}_{2*}} "13";"12";};
{ \ar^-{0_B + i} "21";"11";};
{ \ar^-{0_{B*}+i_*} "22";"12";};
{ \ar^-{(0_B+i)_*} "23";"13";};
{ \ar^-{\lambda_B} "22";"21";};
{ \ar@{=} "23";"22";};
\end{xy}$$
The left-hand squares in the two diagrams commute by naturality and
the right-hand squares commute by the universal property of the direct
sum. 
\end{proof}

\begin{remark}\label{changeofrings}A map of $\lambda$-rings $f \colon
(B,\lambda_B) \to (A,\lambda_A)$ gives rise to a functor
$$f_* \colon \mathscr{M}(A,\lambda_A) \to \mathscr{M}(B,\lambda_B)$$
defined by viewing an $(A,\lambda_A)$-module $(N,\lambda_N)$ as a
$(B,\lambda_B)$-module $f_*(N,\lambda_N)$ via the map $f$. The functor
$f_*$ has a left adjoint functor $f^*$ that to a $(B,\lambda_B)$-module
$(M,\lambda_M)$ associates the $(A,\lambda_A)$-module
$f^*(M,\lambda_M) = (A,\lambda_A)
\otimes_{(B,\lambda_B)}(M,\lambda_M)$ defined by
$$(A,\lambda_A) \otimes_{(B,\lambda_B)}(M,\lambda_M) = (A \otimes_B M,
\lambda_{A \otimes_BM})$$ 
where $\lambda_{A \otimes_BM}$ is given by the composition of
$\lambda_A \otimes_{\lambda_B} \lambda_M$ and the map
$$\mathbb{W}(A) \otimes_{\mathbb{W}(B)} \mathbb{W}(M) \to
\mathbb{W}(A \otimes_B M),$$
that to $a \otimes x$ associates the vector whose $n$th Witt component
is $w_n(a) \otimes x_n$. 
\end{remark}

\begin{definition}\label{lambdaderivation}Let $(A,\lambda_A)$ be a
$\lambda$-ring and let $(M,\lambda_M)$ be an
$(A,\lambda_A)$-module. A derivation from $(A,\lambda_M)$ to
$(M,\lambda_M)$ is a map of sets
$$D \colon (A,\lambda_A) \to (M,\lambda_M)$$
such that the following~(1)--(3) hold.
\begin{enumerate}
\item[(1)]For all $a,b \in A$, $D(a+b) = D(a) + D(b)$.
\item[(2)]For all $a,b \in A$, $D(ab) = bD(a) + aD(b)$.
\item[(3)]For all $a \in A$ and $n \in \N$, $\lambda_{M,n}(D(a)) =
  \sum_{e \mid n} \lambda_{A,e}(a)^{(n/e)-1}D(\lambda_{A,e}(a))$.
\end{enumerate}
The set of derivations from $(A,\lambda_A)$ to $(M,\lambda_M)$ is
denoted by $\operatorname{Der}((A,\lambda_A),(M,\lambda_M))$.
\end{definition} 

We next show that, under the equivalence of categories given in
Proposition~\ref{adjunctionlemma2}, Definition~\ref{lambdaderivation}
agrees with Beck's general definition of a derivation.

\begin{prop}\label{derivationslemma}Let $(A,\lambda_A)$ be a
$\lambda$-ring, let $(M,\lambda_M)$ be an $(A,\lambda_A)$-module, and
let $f \colon (A \ltimes M, \lambda_{A \ltimes M}) \to (A,\lambda_A)$
be the canonical projection. The map
$$\xymatrix{
{ \operatorname{Der}((A,\lambda_A),(M,\lambda_M)) } \ar[r] &
{ \Hom_{\mathscr{A}_{\lambda}/(A,\lambda_A)}(\id_{(A,\lambda_A)}, f) } \cr
}$$
that to $D$ assigns $(\id_A, D)$ is a bijection. 
\end{prop}

\begin{proof}The map from $\operatorname{Der}(A,M)$ to
$\Hom_{\mathscr{A}/A}(\id_A,f)$ that takes $D$ to $(\id_A,D)$ is a
bijection, as is well-known and readily verified. We must show that
$D$ satisfies~(3) if and only if $(\id_A,D) \colon A \to A \ltimes M$
is a $\lambda$-ring homomorphism, and the latter means that the
following diagram commutes.
$$\begin{xy}
(-31,7)*+{ A }="11";
(36,7)*+{ \mathbb{W}(A) }="13";
(-31,-7)*+{ A \ltimes M }="21";
(0,-7)*+{ \mathbb{W}(A) \ltimes \mathbb{W}(M) }="22";
(36,-7)*+{ \mathbb{W}(A \ltimes M) }="23";
{ \ar^-{\lambda_A} "13";"11";};
{ \ar^-{(\id_A,D)} "21";"11";};
{ \ar^-{(\id_A,D)_*} "23";"13";};
{ \ar^-{\lambda_A \oplus \lambda_M} "22";"21";};
{ \ar^-{\operatorname{in}_{1*}+\operatorname{in}_{2*}} "23";"22";};
\end{xy}$$
Now, on the one hand, the map $(\id_A,D)$ takes $a$ to $(a,Da)$
which by $\lambda_A \oplus \lambda_M$ is mapped to $(\lambda_A,
\lambda_M(Da))$ whose $n$th Witt component is
$(\lambda_{A,n}(a),\lambda_{M,n}(Da))$ and, on the other hand, the
$e$th Witt component of the image of $a$ by the composite map
$(\id_A,D)_* \circ \lambda_A$ is equal to
$(\lambda_{A,e}(a),D\lambda_{A,e}(a))$. Hence, 
Lemma~\ref{directsumring} and Addendum~\ref{explicitisomorphism} show
that the diagram commutes if and only if $D$ satisfies~(3).
\end{proof}

\begin{lemma}\label{universalderivation}Let $(A,\lambda_A)$ be a
$\lambda$-ring. There exists a derivation
$$\xymatrix{
{ (A,\lambda_A) } \ar[r]^-{d} &
{ (\Omega_{(A,\lambda_A)}, \lambda_{\Omega_{(A,\lambda_A)}}) } \cr
}$$
which corepresents the functor that to an $(A,\lambda_A)$-module
$(M,\lambda_M)$ assigns the set of derivations
$\operatorname{Der}((A,\lambda_A),(M,\lambda_M))$. 
\end{lemma}

\begin{proof}We define the target of the map $d$ to be the quotient of
the free $(A,\lambda_A)$-module $(F,\lambda_F)$ generated by 
$\{ d(a) \mid a \in A \}$ by the sub-$(A,\lambda_A)$-module
$(R,\lambda_R) \subset (F,\lambda_F)$ generated by 
$d(a+b) - d(a) - d(b)$ with $a,b \in A$; by $d(ab) - bd(a) -
ad(b)$ with $a,b \in A$; and by $\lambda_{F,n}(da) -
  \sum_{e \mid n} \lambda_{A,e}(a)^{(n/e)-1}d(\lambda_{A,e}(a))$ with
$a \in A$ and $n \in \N$. The map $d$ takes $a \in A$ to the class
of $d(a)$ in $\Omega_{(A,\lambda_A)}$. It is clear from the
construction that given a derivation $D \colon (A,\lambda_A) \to
(M,\lambda_M)$, there is a well-defined map of $(A,\lambda_A)$-modules
$f \colon (\Omega_{(A,\lambda_A)},\lambda_{\Omega_{(A,\lambda_A)}})
\to (M,\lambda_M)$
such that $D = f \circ d$ and that $f$ is unique with this
property. This proves the lemma.
\end{proof}

The map $d \colon A \to \Omega_{(A,\lambda_A)}$ in
Lemma~\ref{universalderivation}, in particular, is a derivation of the
ring $A$, and hence, it defines a map of $A$-modules 
$\Omega_A \to \Omega_{(A,\lambda_A)}$. We call this map the canonical
map and now prove Theorem~\ref{universalderivations}, which states that
it is an isomorphism.

\begin{proof}[Proof of Theorem~\ref{universalderivations}]We consider
the diagram of  adjunctions
$$\begin{xy}
(-34,7)*+{ \mathscr{A}/A }="11";
(0,7)*+{ (\mathscr{A}/A)_{\operatorname{ab}} }="12";
(34,7)*+{ \mathscr{M}(A) }="13";
(-34,-7)*+{ \mathscr{A}_{\lambda}/(A,\lambda_A) }="21";
(0,-7)*+{ (\mathscr{A}_{\lambda}/(A,\lambda_A))_{\operatorname{ab}}
}="22";
(34,-7)*+{ \mathscr{M}(A,\lambda_A), }="23";
{ \ar@<.7ex>^-{(-)_{\operatorname{ab}}} "12";"11";};
{ \ar@<.7ex>^-{i} "11";"12";};
{ \ar@<.7ex>^-{F} "13";"12";};
{ \ar@<.7ex>^-{G} "12";"13";};
{ \ar@<.7ex>^-{U_{(A,\lambda_A)}} "11";"21";};
{ \ar@<.7ex>^-{R_{(A,\lambda_A)}} "21";"11";};
{ \ar@<.7ex>^-{U_{(A,\lambda_A)}} "12";"22";};
{ \ar@<.7ex>^-{R_{(A,\lambda_A)}} "22";"12";};
{ \ar@<.7ex>^-{U'} "13";"23";};
{ \ar@<.7ex>^-{R'} "23";"13";};
{ \ar@<.7ex>^-{(-)_{\operatorname{ab}}^{\lambda}} "22";"21";};
{ \ar@<.7ex>^-{i^{\lambda}} "21";"22";};
{ \ar@<.7ex>^-{F^{\lambda}} "23";"22";};
{ \ar@<.7ex>^-{G^{\lambda}} "22";"23";};
\end{xy}$$
where the functors $i$ and $i^{\lambda}$ forget the abelian group
object structure maps, and where $(-)_{\operatorname{ab}}$ and
$(-)_{\operatorname{ab}}^{\lambda}$ are the respective left adjoint
functors which we now define. In the right-hand square, the top
and bottom adjunctions are adjoint equivalences of categories by
Lemma~\ref{adjunctionlemma1} and Proposition~\ref{adjunctionlemma2},
respectively. Hence, the composition of the top adjunctions in the
diagram determine the top adjunction in the left-hand square, up to
unique natural isomorphism, and similarly for the bottom
adjunctions. 

Now, we define an adjunction $(H,K,\epsilon,\eta)$ with
$K = i \circ G$ as follows. The functor $K$ takes the $A$-module $M$ to the 
canonical projection $f \colon A \ltimes M \to A$, and we let $H$
be the functor that to $f \colon B \to A$ assigns the $A$-module $A
\otimes_B\Omega_B$, and let $\epsilon$ and $\eta$ be the natural
transformations given by $\epsilon(1 \otimes d(a,x)) = x$ and $\eta(b)
= (f(b), 1 \otimes db)$, respectively. We must show that the two
composite natural transformations
$$\xymatrix{
{ H } \ar@{=>}[r]^-(.47){H \circ \eta} &
{ H \circ K \circ H } \ar@{=>}[r]^-(.47){\epsilon \circ H} &
{ H, } &
{ K } \ar@{=>}[r]^-(.47){\eta \circ K} &
{ K \circ H \circ K } \ar@{=>}[r]^-(.47){K \circ \epsilon} &
{ K } \cr
}$$
are equal to the respective identity natural transformations. But $H
\circ \eta$ maps $a \otimes db$ in $H(f \colon B \to A)$ to $a \otimes
d(f(b),1 \otimes db)$ in $(H \circ K \circ H)(f \colon B \to A)$ and
$\epsilon \circ H$, in turn, maps this element to $a \cdot (1 \otimes
db) = a \otimes db$ in $H(f \colon B \to A)$; and $\eta \circ K$ maps
$(a,x)$ in $K(M)$ to $(a,1 \otimes d(a,x))$ in $(K \circ H \circ
K)(M)$ and $K \circ \epsilon$, in turn, maps this element to $(a,x)$ in
$K(M)$. This shows that $(H,K,\epsilon,\eta)$ is an
adjunction. Similarly, we define an adjunction
$(H^{\lambda},K^{\lambda},\epsilon^{\lambda},\eta^{\lambda})$ with
$K^{\lambda} = i^{\lambda} \circ G^{\lambda}$ as follows. The functor
$K^{\lambda}$ takes the $(A,\lambda_A)$-module $(M,\lambda_A)$ to the
canonical projection $f \colon (A \ltimes M, \lambda_{A \ltimes M})
\to (A,\lambda_A)$, and we let $H^{\lambda}$ be the functor that to $f
\colon (B,\lambda_B) \to (A,\lambda_A)$ assigns the
$(A,\lambda_A)$-module $(A,\lambda_A) \otimes_{(B,\lambda_B)}
\Omega_{(B,\lambda_B)}$, and let $\epsilon^{\lambda}$ and
$\eta^{\lambda}$ be the natural transformations given by
$\smash{ \epsilon^{\lambda}(1 \otimes d(a,x)) = x }$ and
$\smash{ \eta^{\lambda}(b) = (f(b), 1 \otimes db) }$,
respectively. The change-of-rings functor that we use here was defined
in Remark~\ref{changeofrings}. The proof that the triangle
identities hold follows mutatis mutandis from the calculation in the
case of the adjunction $(H,K,\epsilon,\eta)$. This shows that
$(H^{\lambda},K^{\lambda},\epsilon^{\lambda},\eta^{\lambda})$ is an
adjunction. 

Having established the diagram of adjunctions at the beginning of the
proof, we note that the composite functors $\smash{ R_{(A,\lambda_A)}
  \circ K }$ and $\smash{ K^{\lambda} \circ R' }$ agree, up to unique
natural isomorphism. Indeed, the following diagram is cartesian,
$$\begin{xy}
(-34,7)*+{ A \ltimes \lambda_{A*}\mathbb{W}(M) }="11";
(0,7)*+{ \mathbb{W}(A) \ltimes \mathbb{W}(M) }="12";
(36,7)*+{ \mathbb{W}(A \ltimes M) }="13";
(-34,-7)*+{ A }="21";
(0,-7)*+{ \mathbb{W}(A) }="22";
(36,-7)*+{ \mathbb{W}(A). }="23";
{ \ar^-{\lambda_A \oplus \id} "12";"11";};
{ \ar^-{\operatorname{in}_{1*}+\operatorname{in}_{2*}} "13";"12";};
{ \ar^-{\pr_1} "21";"11";};
{ \ar^-{\pr_1} "22";"12";};
{ \ar^-{\pr_1} "23";"13";};
{ \ar^-{\lambda_A} "22";"21";};
{ \ar@{=} "23";"22";};
\end{xy}$$
By the uniqueness of left adjoint functors, up to unique natural
isomorphism, we conclude that also the composite functors 
$\smash{ H \circ U_{(A,\lambda_A)} }$ and $\smash{ U' \circ
  H^{\lambda} }$ agree, up to unique natural isomorphism. It follows
that the canonical natural transformation 
$$\xymatrix{
{ A \otimes_B \Omega_B } \ar[r] &
{ U'((A, \lambda_A) \otimes_{(B, \lambda_B)} \Omega_{(B,\lambda_B)}) }
\cr
}$$
is an isomorphism, and taking $(B,\lambda_B) = (A,\lambda_A)$, the theorem
follows.
\end{proof}

\begin{theorem}\label{dividedfrobenius}Let $A$ be a ring. There are
natural maps $F_n \colon \Omega_{\mathbb{W}(A)} \to
\Omega_{\mathbb{W}(A)}$ that are $F_n$-linear and satisfy that for all
$a \in \mathbb{W}(A)$,
$$F_n(da) = \sum_{e \mid n}
\Delta_{A,e}(a)^{(n/e)-1}d\Delta_{A,e}(a).$$
Moreover, the following~{\rm(1)--(3)} hold.
\begin{enumerate}
\item[{\rm(1)}] For all $m,n \in \N$, $F_mF_n = F_{mn}$, and $F_1 = \id$.
\item[{\rm(2)}] For all $n \in \N$ and $a \in \mathbb{W}(A)$, $dF_n(a) =
nF_n(da)$.
\item[\rm{(3)}] For all $n \in \N$ and $a \in A$, $F_n(d[a]) =
  [a]^{n-1}d[a]$.
\end{enumerate}
\end{theorem}

\begin{proof}Applying Theorem~\ref{universalderivations} to the
universal $\lambda$-ring $(\mathbb{W}(A),\Delta_A)$, we conclude that
the canonical map $\Omega_{\mathbb{W}(A)} \to
\Omega_{(\mathbb{W}(A),\Delta_A)}$ is an isomorphism. Since the target
of this map is a $(\mathbb{W}(A),\Delta_A)$-module, we have the
natural map
$$F_n = \lambda_{\Omega_{(\mathbb{W}(A),\Delta_A)},n} 
\colon \Omega_{(\mathbb{W}(A),\Delta_A)} 
\to \Omega_{(\mathbb{W}(A),\Delta_A)}$$
defined to be the $n$th Witt component of the
$(\mathbb{W}(A),\Delta_A)$-module structure map; compare
Remark~\ref{lambdamoduleremark}. It is an $F_n = w_n \circ
\Delta_A$-linear map and Definition~\ref{lambdaderivation}~(3) implies
that it is given by the stated formula. Properties~(1) and~(2) follow
immediately from the definition of a
$(\mathbb{W}(A),\Delta_A)$-module and from the calculation
$$dF_n(a) = d( \sum_{e \mid n} e\Delta_{A,e}(a)^{n/e})
= \sum_{e \mid n} n\Delta_{A,e}(a)^{(n/e)-1}d\Delta_{A,e}(a) =
nF_n(da),$$
where the first and last equality follow from the definition of
$\Delta_A$. Finally, to prove property~(3), it suffices to show that 
$\Delta_{A,e}([a])$ is equal to $[a]$, if $e = 1$, and is equal to
$0$, if $e > 1$, or equivalently, that $\Delta([a]) = [[a]]$, and this
was proved in Remark~\ref{deltaremark}.
\end{proof}

\section{The anticommutative graded algebras
  $\smash{ \hat{\Omega}_{\mathbb{W}(A)}^{\cdot} }$
and $\smash{ \check{\Omega}_{\mathbb{W}(A)}^{\cdot} }$}\label{derhamsection}

We next introduce the anticommutative graded $\mathbb{W}(A)$-algebra
$\smash{ \hat{\Omega}_{\mathbb{W}(A)}^{\cdot} }$. It agrees with the
alternating algebra $\smash{ \Omega_{\mathbb{W}(A)}^{\cdot} = 
  \bigwedge_{\mathbb{W}(A)}\Omega_{\mathbb{W}(A)}^1 }$, if the element
$$d\log[-1] = [-1]^{-1}d[-1] \in \Omega_{\mathbb{W}(A)}^1$$
is zero, but is different, in general. We note that $2d\log[-1] =
d\log[1] = 0$ and that, by Lemma~\ref{variousrelations}~(v) and
Theorem~\ref{dividedfrobenius}~(3), $F_n(d\log[-1]) = d\log[-1]$ for
all $n \in \N$. 

\begin{definition}\label{omegahat}Let $A$ be a ring. The graded
$\mathbb{W}(A)$-algebra
$$\hat{\Omega}_{\mathbb{W}(A)}^{\cdot} =
T_{\mathbb{W}(A)}^{\cdot}\Omega_{\mathbb{W}(A)}^1 /
J^{\cdot}$$
is defined to be the quotient of the tensor algebra of the
$\mathbb{W}(A)$-module $\Omega_{\mathbb{W}(A)}^1$ by the graded ideal
generated by all elements of the form
$$da \otimes da - d\log[-1] \otimes F_2(da)$$
with $a \in \mathbb{W}(A)$.
\end{definition}

We remark that the defining relation $da \cdot da = d\log[-1] \cdot
F_2(da)$ is analogous to the relation $\{a,a\} = \{-1,a\}$ in Milnor
$K$-theory. 

\begin{lemma}\label{antisymmetric}The graded $\mathbb{W}(A)$-algebra
$\hat{\Omega}_{\mathbb{W}(A)}^{\cdot}$ is anticommutative.
\end{lemma}

\begin{proof}It suffices to show that the sum
$da \cdot db + db \cdot da \in \hat{\Omega}_{\mathbb{W}(A)}^2$ is
equal to zero for all $a,b \in \mathbb{W}(A)$. Now, on the one hand, we have
$$d(a+b) \cdot d(a+b) = d\log[-1] \cdot F_2d(a+b) = d\log[-1] \cdot
F_2da + d\log[-1] \cdot F_2db,$$
since $F_2d$ is additive, and on the other hand, we have
$$\begin{aligned}
{} & d(a+b) \cdot d(a+b) = da \cdot da + da \cdot db + db \cdot da +
db \cdot db \cr
{} & = d\log[-1] \cdot F_2da + da \cdot db + db \cdot da + d\log[-1]
\cdot F_2db \cr
\end{aligned}$$
This shows that $da \cdot db + db \cdot da$ is zero as desired.
\end{proof}

\begin{prop}\label{uniquederivation}There exists a unique graded derivation
$$d \colon \hat{\Omega}_{\mathbb{W}(A)}^{\cdot}
\to \hat{\Omega}_{\mathbb{W}(A)}^{\cdot}$$
that extends the derivation $d \colon \mathbb{W}(A) \to
\Omega_{\mathbb{W}(A)}^1$ and satisfies the formula
$$dd\omega = d\log[-1] \cdot d\omega$$
for all $\omega \in
\hat{\Omega}_{\mathbb{W}(A)}^{\cdot}$. Moreover, the element
$d\log[-1]$ is a cycle.
\end{prop}

\begin{proof}The relation $dd\omega = d\log[-1] \cdot
  d\omega$
 implies that $d\log[-1]$ is a cycle for the desired derivation $d$. Indeed,
$$\begin{aligned}
d(d\log[-1]) & = d([-1]d[-1]) = d[-1] \cdot d[-1] + [-1] \cdot dd[-1] \cr
{} & = d\log[-1] \cdot F_2d[-1] + [-1]d\log[-1] \cdot d[-1] \cr
{} & = d\log[-1] \cdot [-1]d[-1] + d\log[-1] \cdot [-1]d[-1] \cr
\end{aligned}$$
which is zero by Lemma~\ref{antisymmetric}. This proves that the
desired derivation $d$ necessarily is unique in that for all
$a_0,a_1,\dots,a_q \in \mathbb{W}(A)$, the following formula must hold,
$$d(a_0da_1 \dots da_q) = da_0da_1 \cdots da_q + qd\log[-1] \cdot
a_0da_1\dots da_q.$$
Here $qd\log[-1]$ is equal to either $d\log[-1]$ or
$0$ as $q$ is odd or even. To complete the proof, it
remains to prove that the map $d$ given by this formula is
(a)~well-defined, (b)~a graded derivation, and (c)~satisfies
$dd\omega = d\log[-1] \cdot d\omega$. First, we have
$$\begin{aligned}
{} & d(a_0da_1 \dots da_p \cdot b_0db_1 \dots db_q)
= d(a_0b_0da_1 \dots da_p db_1 \dots db_q) \cr
{} & = d(a_0b_0)da_1 \dots da_pdb_1 \dots db_q
+ (p+q)d\log[-1] \cdot a_0b_0da_1 \dots da_pdb_1 \dots db_q \cr
{} & = da_0da_1 \dots da_p \cdot b_0db_1 \dots db_q + pd\log[-1] \cdot
a_0da_1 \dots da_p \cdot b_0db_1 \dots db_q \cr
{} & \hskip3mm +(-1)^p( a_0da_1\dots da_p \cdot db_0db_1 \dots db_q
+ a_0da_1\dots da_p \cdot qd\log[-1] \cdot b_0db_1 \dots db_q) \cr
{} & =  d(a_0da_1 \dots da_p) \cdot b_0db_1 \dots db_q + (-1)^p a_0da_1
\dots da_p \cdot d(b_0db_1 \dots db_q) \cr
\end{aligned}$$
which proves~(b). Next, using that $q^2+q$ is always even, we find that
$$\begin{aligned}
{} & dd(a_0da_1\dots da_q) = d(da_0da_1\dots da_q + q d\log[-1] \cdot
a_0da_1\dots da_q) \cr
{} & = (q+1)d\log[-1] \cdot da_0da_1 \dots da_q - qd\log[-1] \cdot
da_0da_1\dots da_q \cr
{} & \hskip7mm - qd\log[-1] \cdot qd\log[-1] \cdot a_0da_1 \dots da_q
\cr
{} & = d\log[-1] \cdot (da_0da_1 \dots da_q + qd\log[-1] \cdot a_0da_1
\dots da_q ) \cr
{} & = d\log[-1] \cdot d(a_0da_1 \dots da_q) \cr
\end{aligned}$$
which proves~(c). Finally, to prove~(a), we must show that for all
$a,b \in \mathbb{W}(A)$, the elements $d(d(ab) - bda -adb)$ and
$d(dada - d\log[-1] \cdot F_2da)$ of
$\hat{\Omega}_{\mathbb{W}(A)}^{\cdot}$ are zero.
First, using Lemma~\ref{antisymmetric} together with~(b) and~(c), we
find that
$$\begin{aligned}
{} & d(d(ab) - bda - adb) = dd(ab) - dbda -bdda - dadb - addb \cr
{} & = d\log[-1] \cdot d(ab) - d\log[-1] \cdot bda - d\log[-1] \cdot
adb \cr
{} & = d\log[-1] \cdot (d(ab) - bda - adb) \cr
\end{aligned}$$
which is zero, since $d \colon \mathbb{W}(A) \to
\hat{\Omega}_{\mathbb{W}(A)}^1$ is a derivation. This shows that the
first type of elements are zero. Next,~(b) and~(c) show that
$$d(dada) = 2d\log[-1] \cdot dada$$
which is zero as is
$$\begin{aligned}
d(d\log[-1] \cdot F_2da) & = d\log[-1] \cdot dF_2da 
= d\log[-1] \cdot d(ada + d\Delta_2(a)) \cr
{} & = d\log[-1] \cdot (dada + d\log[-1] \cdot F_2da) \cr
\end{aligned}$$
by the definition of $\hat{\Omega}_{\mathbb{W}(A)}^2$. Hence also
the second type of elements are zero. This completes the proof of~(a)
and hence of the proposition. 
\end{proof}

\begin{remark}\label{comparisonwithderhamcomplex}In general, there
is no $\mathbb{W}(A)$-algebra map $\smash{ f \colon
\hat{\Omega}_{\mathbb{W}(A)}^{\cdot} \to
\Omega_{\mathbb{W}(A)}^{\cdot} }$ that is compatible with
the derivations.
\end{remark}

\begin{prop}\label{omegafrobenius}Let $A$ be a ring and let $n$ be a
positive integer. There is a unique homomorphism of graded rings
$$F_n \colon \hat{\Omega}_{\mathbb{W}(A)}^{\cdot} \to
\hat{\Omega}_{\mathbb{W}(A)}^{\cdot}$$
that is given by the maps $F_n \colon \mathbb{W}(A) \to \mathbb{W}(A)$
and $F_n \colon \Omega_{\mathbb{W}(A)}^1 \to
\Omega_{\mathbb{W}(A)}^1$ in degrees $0$ and $1$,
respectively. In addition,  the following formula holds.
$$dF_n = nF_nd$$
\end{prop}

\begin{proof}The uniqueness statement is clear: The map $F_n$ is
necessarily given by
$$F_n(a_0da_1\dots da_q) = F_n(a_0)F_n(da_1) \dots F_n(da_q)$$
where $a_0, \dots, a_q \in \mathbb{W}(A)$. We show that this formula
gives a well-defined map. To prove this, we must show that for every $a \in
\mathbb{W}(A)$,
$$F_n(da)F_n(da) = F_n(d\log[-1])F_n(F_2da).$$
It will suffice to let $n = p$ be a prime number. In this case, we find that
$$\begin{aligned}
{} & F_p(da) F_p(da) 
= (a^{p-1}da + d\Delta_p(a)) \cdot (a^{p-1}da + d\Delta_p(a)) \cr
{} & = (a^{p-1})^2da \cdot da + d\Delta_p(a) \cdot d\Delta_p(a) 
= d\log[-1] \cdot ((a^{p-1})^2F_2da + F_2d\Delta_p(a)) \cr
{} & = d\log[-1] \cdot (F_2(a^{p-1})F_2da + F_2d\Delta_p(a)) 
= d\log[-1] \cdot F_2F_pda \cr
{} & = F_p(d\log[-1] \cdot F_2da) \cr
\end{aligned}$$
where we have used that $F_2(a)$ is congruent to $a^2$ modulo
$2\mathbb{W}(A)$. This shows that the map $F_n$ is well-defined. It is
a graded ring homomorphism by definition. 

We next prove the formula $dF_n = nF_nd$. Again, we may assume that $n
= p$ is a prime number. We already know from the definition of
$F_n \colon \Omega_{\mathbb{W}(A)}^1 \to \Omega_{\mathbb{W}(A)}^1$
that for all $a \in \mathbb{W}(A)$, $dF_p(a) = pF_p(a)$. Now, for all
$a \in \mathbb{W}(A)$, 
$$dF_p(da) = d(a^{p-1}da + d\Delta_p(a))
= (p-1)a^{p-2}dada + d\log[-1] \cdot F_pda$$
which is equal to zero for $p = 2$, and equal to $d\log[-1] \cdot
F_pda$ for $p$ odd. Hence, for every prime $p$ and every $a \in
\mathbb{W}(A)$, we have 
$$dF_p(da) = pd\log[-1] \cdot F_pda = pF_p(d\log[-1] \cdot da) =
pF_pd(da)$$
as desired. Now, let $a_0,\dots,a_q \in \mathbb{W}(A)$. We find
$$\begin{aligned}
{} & dF_p(a_0da_1\dots da_q)
= d(F_p(a_0)F_pda_1 \dots F_pda_q) \cr
{} & = dF_p(a_0)F_pda_1 \dots F_pda_q 
+ \sum_{1 \leqslant i \leqslant q} (-1)^{i-1}
F_p(a_0)F_pda_1 \dots dF_pda_i \dots F_pda_q \cr
{} & = pF_pd(a_0)F_pda_1\dots F_pda_q
+ \sum_{1 \leqslant i \leqslant q} (-1)^{i-1}
F_p(a_0)F_pda_1 \dots pF_pd(da_i) \dots F_pda_q \cr
{} & = pF_pd(a_0da_1 \dots da_q) \cr
\end{aligned}$$
as stated. This completes the proof.
\end{proof}

We next define the quotient graded algebra $\smash{
  \check{\Omega}_{\mathbb{W}(A)}^{\cdot} }$ of the graded
algebra $\smash{ \hat{\Omega}_{\mathbb{W}(A)}^{\cdot} }$
and show that the Frobenius $F_n$ and derivation $d$ descend to this
quotient.

\begin{definition}\label{omegacheck}Let $A$ be a ring. The graded
$\mathbb{W}(A)$-algebra
$$\check{\Omega}_{\mathbb{W}(A)}^{\cdot} 
=\hat{\Omega}_{\mathbb{W}(A)}^{\cdot} /
K^{\cdot}$$
is defined to be the quotient by the graded
ideal $K^{\cdot}$ generated by the elements
$$F_pdV_p(a) - da - (p-1)d\log[-1] \cdot a$$
where $p$ ranges over all prime numbers and $a$ over all elements of $\mathbb{W}(A)$.
\end{definition}

We remark that the element $F_pdV_p(a) - da - (p-1)d\log[-1] \cdot a$
is annihilated by $p$. In particular, it is zero, if $p$ is invertible
in $A$, and hence, in $\mathbb{W}(A)$. 

\begin{lemma}\label{frobeniusdescends}The Frobenius $F_n
  \colon \hat{\Omega}_{\mathbb{W}(A)}^{\cdot} \to
\hat{\Omega}_{\mathbb{W}(A)}^{\cdot}$ induces a map of graded algebras
$$F_n \colon \check{\Omega}_{\mathbb{W}(A)}^{\cdot} \to 
\check{\Omega}_{\mathbb{W}(A)}^{\cdot}.$$
\end{lemma}

\begin{proof}It will suffice to let $n = \ell$ be a prime
number and show that for all prime numbers $p$ and all $a \in
\mathbb{W}(A)$, the element 
$$F_{\ell}(F_pdV_p(a) - da - (p-1)d\log[-1] \cdot a) \in
\Omega_{\mathbb{W}(A)}^1$$
maps to zero in $\smash{ \check{\Omega}_{\mathbb{W}(A)}^1 }$. To this
end, we will repeatedly use that for every prime number $p$ and every
$b \in \mathbb{W}(A)$, Theorem~\ref{dividedfrobenius} and
Remark~\ref{deltaremark} shows that
$$F_pdb = b^{p-1}db + d\Delta_p(b) = b^{p-1}db +
d(\frac{F_p(b)-b^p}{p})$$
as elements of $\Omega_{\mathbb{W}(A)}^1$. We also use that, by
Lemma~\ref{variousrelations}~(ii)--(iii), we have
$$V_m(a)^n = m^{n-1}V_m(a^n),$$
for all $m,n \in \N$ and $a \in \mathbb{W}(A)$. Now, suppose first
that $\ell = p$. For $p$ odd, we find that 
$$\begin{aligned}
{} & F_p(F_pdV_p(a) - da)
{} = \smash{ F_p(V_p(a)^{p-1}dV_p(a) + d( \frac{F_pV_p(a) - V_p(a)^p}{p}) -
da) } \cr
& {} = F_p(p^{p-2}V_p(a^{p-1})dV_p(a) - p^{p-2}dV_p(a^p)) \cr
{} & = p^{p-1}a^{p-1}F_pdV_p(a) - p^{p-2}F_pdV_p(a^p), \cr
\end{aligned}$$
where we also use that $F_pV_p = p \cdot \id$. But this
element maps to zero $\smash{ \check{\Omega}_{\mathbb{W}(A)}^1 }$,
since, as maps from $\mathbb{W}(A)$ to $\smash{
\check{\Omega}_{\mathbb{W}(A)}^1 }$, we have $F_pdV_p = d$, and since
the common map is a derivation. Similarly, for $p = 2$, we find that
$$\begin{aligned}
{} &F_2(F_2dV_2(a) - da - d\log[-1] \cdot a)
{} = F_2(V_2(a)dV_2(a) -dV_2(a^2) -d\log[-1] \cdot a) \cr
{} & = 2aF_2dV_2(a) - F_2dV_2(a^2) - d\log[-1] \cdot F_2(a), \cr
\end{aligned}$$
where we further use that $F_m(d\log[-1]) = d\log[-1]$, for every $m
\in \N$. The image of this element in
$\check{\Omega}_{\mathbb{W}(A)}^1$ is equal to
$$2ada - d(a^2) - d\log[-1] \cdot (F_2(a) - a^2),$$
which, in turn, is zero, since $d$ is a derivation and since $F_2(a) -
a^2$ is divisible by $2$ by Lemma~\ref{pfrobenius}. We next suppose
that $p \neq \ell$. In this case, we further use that $\ell$ divides
$p^{\ell-1}-1$ and that, by Lemma~\ref{variousrelations}~(ii),
$F_{\ell}\,V_p = V_pF_{\ell}$. If $p$ and $\ell$ both are odd, then
$$\begin{aligned}
{} &  F_{\ell}(F_pdV_p(a) - da)
= F_p(F_{\ell}dV_p(a)) - F_{\ell}da \cr
{} & = F_p(V_p(a)^{\ell-1}dV_p(a) + d(\frac{F_{\ell}V_p(a) -
  V_p(a)^{\ell}}{\ell}))
- a^{\ell-1}da - d(\frac{F_{\ell}(a)-a^{\ell}}{\ell}) \cr
{} & = p^{\ell-1}a^{\ell-1}F_pdV_p(a) + F_pdV_p(\frac{F_{\ell}(a) -
  p^{\ell-1}a^{\ell}}{\ell})
- a^{\ell-1}da - d(\frac{F_{\ell}(a)-a^{\ell}}{\ell}), \cr
\end{aligned}$$
and the image of this element in 
$\check{\Omega}_{\mathbb{W}(A)}^1$ is equal to
$$\smash{ (p^{\ell-1}-1)a^{\ell-1}da -
\frac{p^{\ell-1}-1}{\ell}d(a^{\ell}), }$$
which is zero since $d$ is a derivation. Similarly, if $p = 2$ and
$\ell \neq p$, then
$$\begin{aligned}
{} & F_{\ell}(F_2dV_2(a) - da - d\log[-1] \cdot a) =
F_2(F_{\ell}dV_2(a)) - F_{\ell}da - d\log[-1] \cdot F_{\ell}(a) \cr
{} & = F_2(V_2(a)^{\ell-1}dV_2(a) + d(\frac{F_{\ell}V_2(a) -
  V_2(a)^{\ell}}{\ell})) \cr
{} & \hskip7mm - a^{\ell-1}da - d(\frac{F_{\ell}(a)-a^{\ell}}{\ell}) -
d\log[-1] \cdot F_{\ell}(a) \cr
{} & = 2^{\ell-1}a^{\ell-1}F_2dV_2(a) + F_2dV_2(\frac{F_{\ell}(a) -
  2^{\ell-1}a^{\ell}}{\ell}) \cr
{} & \hskip7mm
- a^{\ell-1}da - d(\frac{F_{\ell}(a) - a^{\ell}}{\ell}) - d\log[-1]
\cdot F_{\ell}(a), \cr
\end{aligned}$$
and the image of this element in $\check{\Omega}_{\mathbb{W}(A)}^1$ is
equal to
$$(2^{\ell-1}-1)a^{\ell-1}da - \frac{2^{\ell-1}-1}{\ell}d(a^{\ell})
+ d\log[-1] \cdot ( \frac{F_{\ell}(a) - 2^{\ell-1}a^{\ell}}{\ell} -
F_{\ell}(a)),$$
which is zero since $d$ is a derivation and since $\ell$ is congruent
to $1$ modulo $2$. Finally, if $\ell = 2$ and $p \neq \ell$, then we
find that
$$\begin{aligned}
{} & F_2(F_pdV_p(a) - da) = F_p(F_2dV_p(a)) - F_2da \cr
{} & = F_p(V_p(a)dV_p(a) + d(\frac{F_2V_p(a) - V_p(a)^2}{2}))
- ada - d(\frac{F_2(a) - a^2}{2}) \cr
{} & = paF_pdV_p(a) + F_pdV_p(\frac{F_2(a) - pa^2}{2})
- ada - d(\frac{F_2(a)-a^2}{2}) \cr
\end{aligned}$$
whose image in $\check{\Omega}_{\mathbb{W}(A)}^1$ is equal to
$$\smash{ (p-1)ada - \frac{p-1}{2}d(a^2), }$$
which again is zero since $d$ is a derivation. 
\end{proof}

\begin{lemma}\label{F_ndV_n(a)}For all positive integers $n$ and $a \in
\mathbb{W}(A)$, the relation
$$F_ndV_n(a) = da + (n-1)d\log[-1] \cdot a$$
holds in $\check{\Omega}_{\mathbb{W}(A)}^1$.
\end{lemma}

\begin{proof}We argue by induction on the number $r$ of prime factors
in $n$ that the stated relation holds for all $a \in
\mathbb{W}(A)$. The case $r = 1$ follows from
Definition~\ref{omegacheck}. So we let $n$ be a positive integer with
$r > 1$ prime factors and assume that the relation has been proved for
all positive integers with less than $r$ prime factors. We write $n
= pm$ with $p$ a prime number and use the inductive hypothesis to
conclude that
$$\begin{aligned}
F_ndV_n(a)
{} & = F_pF_mdV_mV_p(a)
{} = F_p(dV_p(a) + (m-1)d\log[-1] \cdot V_p(a)) \cr
{} & = F_pdV_p(a) + (m-1)d\log[-1] \cdot F_pV_p(a) \cr
{} & = da + (p-1)d\log[-1] \cdot a + p(m-1)d\log[-1] \cdot a \cr
{} & = da + (n-1)d\log[-1] \cdot a \cr
\end{aligned}$$
which proves the induction step. 
\end{proof}

\begin{lemma}\label{derivationdescends}The graded derivation $d \colon
\hat{\Omega}_{\mathbb{W}(A)}^{\cdot} \to
\hat{\Omega}_{\mathbb{W}(A)}^{\cdot}$ induces a graded derivation
$$d \colon
\check{\Omega}_{\mathbb{W}(A)}^{\cdot} \to
\check{\Omega}_{\mathbb{W}(A)}^{\cdot}.$$
\end{lemma}

\begin{proof}We must show that for all prime numbers $p$ and $a \in
\mathbb{W}(A)$, the element
$$d(F_pdV_p(a) - da - (p-1)d\log[-1] \cdot a) \in
\hat{\Omega}_{\mathbb{W}(A)}^2$$
maps to zero in $\check{\Omega}_{\mathbb{W}(A)}^2$. First, for $p =
2$, we have
$$\begin{aligned}
{} & d(F_2dV_2(a) - da - d\log[-1] \cdot a)
{} = dF_2dV_2(a) - dda + d\log[-1] \cdot da \cr
{} & = 2F_2ddV_2(a) = 2d\log[-1] \cdot F_2dV_2(a) \cr
\end{aligned}$$
which is even zero in $\hat{\Omega}_{\mathbb{W}(A)}^2$. For $p$ odd,
we recall from Theorem~\ref{dividedfrobenius} that
$$F_pdV_p(a) - da = V_p(a)^{p-1}dV_p(a) + d\Delta_pV_p(a) - da,$$
and using that $d$ is a derivation, we find that
$$\begin{aligned}
{} & d(F_pdV_p(a) - da) = d(V_p(a)^{p-1}dV_p(a) + d\Delta_pV_p(a)) -
dda \cr
{} & = (p-1)V_p(a)^{p-2}dV_p(a)dV_p(a) + V_p(a)^{p-1}ddV_p(a) +
dd\Delta_pV_p(a) - dda. \cr
\end{aligned}$$
Now, the first summand in the bottom line vanishes, since $p-1$ is
even and
$$dV_p(a)dV_p(a) = d\log[-1] \cdot F_2dV_p(a),$$
and by Proposition~\ref{uniquederivation}, the sum of the remaining
three summands is equal to
$$d\log[-1] \cdot (V_p(a)^{p-1}dV_p(a) + d\Delta_pV_p(a) - da) =
d\log[-1] \cdot (F_pdV_p(a) - da),$$
which maps to zero in $\check{\Omega}_{\mathbb{W}(A)}^2$,
since $p-1$ is even.
\end{proof}

\begin{definition}\label{truncatedomegahatomegacheck}Let $A$ be a
ring, let $S \subset \N$ be a truncation set, and let 
$I_S(A) \subset \mathbb{W}(A)$ be the kernel of the restriction map
$R_S^{\N} \colon \mathbb{W}(A) \to \mathbb{W}_S(A)$. The maps
$$\xymatrix{
{ \hat{\Omega}_{\mathbb{W}(A)}^{\cdot} } \ar[r]^-{R_S^{\N}} &
{ \hat{\Omega}_{\mathbb{W}_S(A)}^{\cdot}, } &
{ \check{\Omega}_{\mathbb{W}(A)}^{\cdot} } \ar[r]^-{R_S^{\N}} &
{ \check{\Omega}_{\mathbb{W}_S(A)}^{\cdot} } \cr
}$$
are defined to be the quotient maps that annihilate the respective
graded ideals generated by $I_S(A)$ and $dI_S(A)$. 
\end{definition}

\begin{remark}\label{K_S}The kernel $K_S^{\cdot}$ of the
canonical projection $\smash{
\hat{\Omega}_{\mathbb{W}_S(A)}^{\cdot} \to
\check{\Omega}_{\mathbb{W}_S(A)}^{\cdot} }$ is equal to
the graded ideal generated by the elements
$$p^{p-2}(V_p(R_{S/p}^S(a^{p-1}))dV_pR_{S/p}^S(a) - dV_pR_{S/p}^S(a^p))
- (p-1)d\log[-1]_S \cdot a$$
with $p$ a prime number and $a \in \mathbb{W}_S(A)$. Indeed, letting
$b = V_p(a)$, the formula for $F_pdb$ in the beginning of the proof of
Lemma~\ref{frobeniusdescends} shows that for all prime 
numbers $p$ and $a \in \mathbb{W}(A)$, the following identity holds in
$\Omega_{\mathbb{W}(A)}^1$,
$$F_pdV_p(a) - da = p^{p-2}(V_p(a^{p-1})dV_p(a) - dV_p(a^p)).$$
\end{remark}

\begin{remark}If $p$ is a prime number and $A$ a $\Z_{(p)}$-algebra,
then for every truncation set $S$, the ideal
$V_p\mathbb{W}_{S/p}(A) \subset \mathbb{W}_S(A)$ has a divided power
structure defined by
$$V_p(a)^{[n]} = \frac{p^{n-1}}{n!} V_p(a^n).$$
If $p$ is \emph{odd}, then $\smash{ d \colon
\mathbb{W}_S(A) \to \check{\Omega}_{\mathbb{W}_S(A)}^1 }$ is a divided
power derivation in the sense that
$$d(V_p(a)^{[n]}) = V_p(a)^{[n-1]}dV_p(a)$$
and it is universal with this property; see~\cite[Lemma~1.2]{langerzink}. 
\end{remark}

\begin{lemma}\label{variousmapsdescend}The derivation, restriction,
and Frobenius induce maps 
$$\begin{aligned}
d & \colon \hat{\Omega}_{\mathbb{W}_S(A)}^{\cdot}
\to \hat{\Omega}_{\mathbb{W}_S(A)}^{\cdot} 
\hskip10.65mm
 (\text{resp.\ } 
d \colon \check{\Omega}_{\mathbb{W}_S(A)}^{\cdot}
\to \check{\Omega}_{\mathbb{W}_S(A)}^{\cdot}) \cr
R_T^S & \colon \hat{\Omega}_{\mathbb{W}_S(A)}^{\cdot}
\to \hat{\Omega}_{\mathbb{W}_T(A)}^{\cdot}
\hskip10mm
 (\text{resp.\ } 
R_T^S \colon \check{\Omega}_{\mathbb{W}_S(A)}^{\cdot}
\to \check{\Omega}_{\mathbb{W}_T(A)}^{\cdot}) \cr
F_n & \colon \hat{\Omega}_{\mathbb{W}_S(A)}^{\cdot}
\to \hat{\Omega}_{\mathbb{W}_{S/n}(A)}^{\cdot} 
\hskip8.4mm
 (\text{resp.\ } 
F_n \colon \check{\Omega}_{\mathbb{W}_S(A)}^{\cdot}
\to \check{\Omega}_{\mathbb{W}_{S/n}(A)}^{\cdot}) \cr
\end{aligned}$$
Moreover, the maps $d$ are graded derivations; the maps $R_T^S$ and
$F_n$ are graded ring homomorphisms; the maps $R_T^S$ and $d$ commute;
and $dF_n = nF_nd$.
\end{lemma}

\begin{proof}To prove the statement for $d$, we note that as $d$ is a derivation, it suffices to show that
$\smash{ R_S^{\N}(ddI_S(A)) \subset \hat{\Omega}_{\mathbb{W}_S(A)}^2 }$
is zero. But if $x \in I_S(A)$, then
$$R_S^{\N}(ddx) = R_S^{\N}(d\log[-1] \cdot dx) = R_S^{\N}(d\log[-1])
\cdot R_S^{\N}(dx)$$
which is zero as desired. It follows that $R_T^{\N}(dI_S(A)) =
dR_T^{\N}(I_S(A)$. Hence, also the statement for $R_T^S$ follows as
$R_T^{\N}(I_S(A)$ is trivially zero. Finally, to prove the statement
for $F_n$, we show that both $R_{S/n}^{\N}(F_n(I_S(A)))$ and
$R_{S/n}^{\N}(F_n(dI_S(A)))$ are zero. For the former, this follows
immediately from Lemma~\ref{frobenius}, and for the latter, it will
suffice to show that for divisors $e$ of $n$, $\Delta_e(I_S(A))
\subset I_{S/e}(A)$. Moreover, to prove this, we may assume that $A$
is torsion free. So let $e$ be a divisor of $n$ and assume that
for all proper divisors $d$ of $e$, $\Delta_d(I_S(A)) \subset
I_{S/d}(A)$. Since $F_e(I_S(A)) \subset I_{S/e}(A)$, the formula
$$F_e(a) = \sum_{d \mid e} d\Delta_d(a)^{e/d}$$
shows that $e\Delta_e(I_S(A)) \subset I_{S/e}(A)$. This completes the
proof of the first part of the statement and the second part is clear.
\end{proof}

Finally, we record the following result concerning the case $S =
\{1\}$. 

\begin{lemma}\label{S=1}For every ring $A$, the differential graded
algebras $\smash{ \hat{\Omega}_A^{\cdot} }$ and
$\smash{ \Omega_A^{\cdot} }$ are equal and the canonical
projection $\smash{ \hat{\Omega}_A^{\cdot} \to
\check{\Omega}_A^{\cdot} }$ is an isomorphism. 
\end{lemma}

\begin{proof}Since $d\log[-1]_{\{1\}}$ is zero, $\smash{
\hat{\Omega}_A^{\cdot} = \Omega_A^{\cdot} }$
as stated. Moreover, Remark~\ref{K_S} shows that the kernel
$K_{\{1\}}$ of the canonical projection $\smash{
\hat{\Omega}_A^{\cdot} \to
\check{\Omega}_A^{\cdot} }$ is zero. 
\end{proof}

\section{The big de~Rham-Witt complex}\label{drwsection}

In this section, we construct the big de~Rham-Witt complex. We let $J$
be the category with objects the truncation sets $S \subset \N$ and
with a single morphism from $T$ to $S$ if $T \subset S$. If $A$ is a
ring, then there is a contravariant functor from $J$ to the category of
rings that to $S$ assigns $\mathbb{W}_S(A)$ and that to $T \subset S$
assigns $R_T^S \colon \mathbb{W}_S(A) \to \mathbb{W}_T(A)$; it takes
colimits in $J$ to limits in the category of rings. For every $n \in \N$,
there is an endofunctor on $J$ that takes $S$ to $S/n$, and
the ring homomorphism $F_n \colon \mathbb{W}_S(A) \to
\mathbb{W}_{S/n}(A)$ and the abelian group homomorphism $V_n \colon
\mathbb{W}_{S/n}(A) \to \mathbb{W}_S(A)$ are natural transformations
with respect to $S$.

We proceed to define the notion of a Witt complex over $A$. The original
definition given in~\cite[Definition~1.1.1]{hm2} is not quite correct
unless the prime $2$ is either invertible or zero in $A$. The correct
definition of a $2$-typical Witt complex was given first by
Costeanu~\cite[Definition~1.1]{costeanu}. The definition given below 
was also inspired by~\cite{stienstra}.

\begin{definition}\label{bigwittcomplex}A \emph{Witt complex} over $A$
is a contravariant functor that to every truncation set $S \subset \N$
assigns an anticommutative graded ring $E_S^{\cdot}$ and that takes
colimits to limits together with a natural ring homomorphism
$$\xymatrix{ 
{ \mathbb{W}_S(A) } \ar[r]^-{\eta_S} &
{ E_S^0 } \cr
}$$
and natural maps of graded abelian groups
$$\xymatrix{
{ E_S^q } \ar[r]^-{d} &
{ E_S^{q+1} } &
{ E_S^q } \ar[r]^-{F_n} &
{ E_{S/n}^q } &
{ E_{S/n}^q } \ar[r]^-{V_n} &
{ E_S^q } &
{ (n \in \N) } \cr
}$$
such that the following (i)--(v) hold.
\begin{enumerate}
\item[(i)] For all $x \in E_S^q$ and $x' \in E_S^{q'}$,
$$\begin{aligned}
{} & d(x \cdot x') = d(x) \cdot x' + (-1)^q x \cdot d(x'), \cr
{} & d(d(x)) = d\log\eta_S([-1]_S) \cdot d(x), \cr
\end{aligned}$$
where $d\log\eta_S([-1]_S) = \eta_S([-1])^{-1}d\eta_S([-1]_S)$.
\item[(ii)] For all positive integers $m$ and $n$,
$$\begin{aligned}
{} & F_1 = V_1 = \id, \hskip4mm F_mF_n = F_{mn}, \hskip4mm V_nV_m =
V_{mn}, \cr
{} & F_nV_n = n \cdot \id, \hskip4mm F_mV_n = V_nF_m \hskip3mm
\text{if $(m,n) = 1$,} \cr
{} & F_n \eta_S = \eta_{S/n} F_n, \hskip6mm \eta_S V_n = V_n \eta_{S/n}.
\cr
\end{aligned}$$
\item[(iii)] For all positive integers $n$, the map $F_n$ is a ring
homomorphism and the maps $F_n$ and $V_n$ satisfy the projection
formula that for all $x \in E_S^q$ and $y \in E_{S/n}^{q''}$,
$$x \cdot V_n(y) = V_n(F_n(x)y).$$
\item[(iv)] For all positive integers $n$ and all $y \in E_{S/n}^q$,
$$F_ndV_n(y) = d(y) + (n-1)d\log\eta_{S/n}([-1]_{S/n}) \cdot y.$$
\item[(v)] For all positive integers $n$ and $a \in A$,
$$F_nd\eta_S([a]_S) = \eta_{S/n}([a]_{S/n}^{n-1})d\eta_{S/n}([a]_{S/n}).$$
\end{enumerate}
A map of Witt complexes is a natural map of graded rings
$f \colon E_S^{\cdot} \to E_S'{}^{\cdot}$ such that $f\eta = \eta'$,
$fd = d'f$, $fF_n = F_n'f$, and $fV_n = V_n'f$.
\end{definition}

\begin{remark}\label{dlog[-1]}(a) For $T \subset S$ a pair of truncation
sets, we write $R_T^S \colon E_S^{\cdot} \to
E_T^{\cdot}$ for the map of graded rings that is part of
the structure of a Witt complex and call it the restriction from $S$
to $T$. 

\noindent~(b) Every Witt complex is determined, up to canonical
isomorphism, by its value on finite truncation sets. Indeed, for every
truncation set $S$ and non-negative integer $q$, the maps in~(a)
defined a bijection from $E_S^q$ to the limit with respect to the
restriction maps of the $E_T^q$, where $T \subset S$ ranges over the
finite sub-truncation sets. In particular, if we write $a \in
\mathbb{W}(A)$ as a convergent sum $a = \sum_{n \in
  S}V_n([a_n]_{S/n})$ as in Lemma~\ref{variousrelations}~(i), then the
element $\smash{ d\eta_S(a) \in E_S^1 }$, too, admits the convergent
sum expression
$$d\eta_S(a) = \sum_{n \in S} dV_n([a_n]_{S/n}).$$

\noindent~(c) The element $d\log\eta_S([-1]_S)$ is annihilated by
$2$. Indeed, since $d$ is a derivation,
$$2d\log\eta_S([-1]_S) = d\log\eta_S([1]_S) = 0.$$
Therefore, $d\log\eta_S([-1]_S)$ is zero if $2$ is invertible in $A$
and hence in $\mathbb{W}_S(A)$. It is also zero if $2 = 0$ in $A$
since, in this case, $[-1]_S = [1]_S$. Finally, it follows from the
general formula $[-1]_S = -[1]_S + V_2([1]_{S/2})$ proved in
Addendum~\ref{teichmuller} that $d\log\eta_S([-1]_S)$ is zero if every
$n \in S$ is odd.

\noindent~(d) Let $A$ be a ring and let $E_S^{\boldsymbol{\cdot}}$ be
a Witt complex over $A$. For every non-negative integer $q$, the pair
$(E_{\N}^q,\lambda_{E^q})$ consisting of $E_{\N}^q$ considered as a
$\mathbb{W}(A)$-module via the ring homomorphism $\eta_{\N} \colon
\mathbb{W}(A) \to E_{\N}^0$ and of the maps
$\lambda_{E^q,n} = F_n \colon E_{\N}^q \to E_{\N}^q$ is a module over
the $\lambda$-ring $(\mathbb{W}(A),\Delta_A)$ in the sense of
Definition~\ref{lambdamodule}. Moreover, we may substitute the
axiom~(v) in Definition~\ref{bigwittcomplex} by the statement~(v')
that the map
$$\xymatrix{
{ (\mathbb{W}(A),\Delta_A) } \ar[r]^-{d} &
{ (E_{\N}^1,\lambda_{E^1}) } \cr
}$$
is a derivation in the sense of
Definition~\ref{lambdaderivation}. Indeed, it follows from
Theorem~\ref{dividedfrobenius} that~(i)--(iv) and~(v') imply~(v),
and we will show in Proposition~\ref{frobeniuscompatible} below
that~(i)--(v) imply~(v').
\end{remark}

\begin{lemma}\label{wittcomplexrelations}Let $m$ and $n$ be positive
integers, let $c = (m,n)$ be the greatest common divisor, and let $i$
and $j$ be any pair of integers such that $mi+nj = c$. The following
relations hold in every Witt complex.
$$\begin{aligned}
{} & dF_n = nF_nd, \hskip 4mm V_nd = ndV_n, \cr
{} & F_mdV_n = idF_{m/c}V_{n/c} + jF_{m/c}V_{n/c}d +
(c-1)d\log\eta_{S/m}([-1]_{S/m}) \cdot F_{m/c}V_{n/c}, \cr
{} & d\log\eta_S([-1]_S) =  \textstyle{\sum_{r \geqslant 1}
  2^{r-1}dV_{2^r}\eta_{S/2^r}([1]_{S/2^r})},  \cr
{} & d\log\eta_S([-1]_S) \cdot d\log\eta_S([-1]_S) = 0, \hskip4mm 
dd\log\eta_S([-1]_S) = 0, \cr
{} & \smash{ F_n(d\log\eta_S([-1]_S)) = d\log\eta_{S/n}([-1]_{S/n}), } \cr
%
\end{aligned}$$
\end{lemma}
 
\begin{proof}The following calculation verifies the first two relations.
$$\begin{aligned}
dF_n(x) & = F_ndV_nF_n(x) - (n-1)d\log\eta[-1] \cdot F_n(x) \cr
{} & = F_nd(V_n\eta([1]) \cdot x) - (n-1)d\log\eta([-1]) \cdot F_n(x) \cr
{} & = F_n(dV_n\eta([1]) \cdot x + V_n\eta([1]) \cdot dx) - (n-1)d\log\eta([-1])
\cdot F_n(x) \cr
{} & = F_ndV_n\eta([1]) \cdot F_n(x) + F_nV_n\eta([1]) \cdot F_nd(x)
- (n-1)d\log\eta([-1]) \cdot F_n(x) \cr
{} & = (n-1)d\log\eta([-1]) \cdot F_n(x) + nF_nd(x) - (n-1)d\log\eta([-1]) \cdot
F_n(x) \cr
{} & = nF_nd(x) \cr
\end{aligned}$$
$$\begin{aligned}
V_nd(x) & = V_n(F_ndV_n(x) - (n-1)d\log\eta([-1]) \cdot x) \cr
{} & = V_n\eta([1]) \cdot dV_n(x) - (n-1)V_n(d\log\eta([-1]) \cdot x) \cr
{} & = d(V_n\eta([1]) \cdot V_n(x)) - dV_n\eta([1]) \cdot V_n(x) -
(n-1)V_n(d\log\eta([-1]) \cdot x) \cr
{} & = dV_n(F_nV_n\eta([1]) \cdot x) - V_n(F_ndV_n\eta([1]) \cdot x) -
(n-1)V_n(d\log\eta([-1]) \cdot x) \cr
{} & = ndV_n(x) - 2(n-1)V_n(d\log\eta([-1]) \cdot x) \cr
{} & = ndV_n(x) \cr
\end{aligned}$$
Next, the last formula follows from $F_m([-1]) = [-1]^m$ and the
calculation 
$$\begin{aligned}
F_m(d\log\eta([-1]_S)) & = F_m(\eta([-1]^{-1})d\eta([-1])) =
F_m\eta([-1]^{-1})F_md\eta([-1]) \cr
{} & = \eta([-1]^{-m})\eta([-1]^{m-1})d\eta([-1]) \cr
{} & = \eta([-1]^{-1})d\eta([-1]) = d\log\eta([-1]). \cr
\end{aligned}$$
Using the three relations proved thus far together with the projection
formula, we find
$$\begin{aligned}
F_mdV_n(x) & = F_{m/c}F_cdV_cV_{n/c}(x) \cr
{} & = F_{m/c}dV_{n/c}(x) + (c-1)d\log\eta([-1]) \cdot F_{m/c}V_{n/c}(x) \cr
{} & = ((m/c)i + (n/c)j)F_{m/c}dV_{n/c}(x) + (c-1)d\log\eta([-1]) \cdot
F_{m/c}V_{n/c}(x) \cr 
{} & = idF_{m/c}V_{n/c}(x) + jF_{m/c}V_{n/c}d(x) + (c-1)d\log\eta([-1]) \cdot
F_{m/c}V_{n/c}(x). \cr 
\end{aligned}$$
Next, to prove the stated formula for $d\log\eta([-1]_S)$, we use that,
by Addendum~\ref{teichmuller}, we have
$[-1]_S = -[1]_S + V_2([1]_{S/2})$ to see that
$$\begin{aligned}
d\log\eta([-1]_S) & = \eta([-1]_S)d\eta([-1]_S) \cr
{} & = \eta(-[1]_S + V_2([1]_{S/2})d\eta(-[1]_S + V_2([1]_{S/2}) \cr
{} & = -dV_2\eta([1]_{S/2}) + V_2(F_2dV_2\eta([1]_{S/2})) \cr
{} & = -dV_2\eta([1]_{S/2}) + V_2(d\log\eta([-1]_{S/2})) \cr
{} & = dV_2\eta([1]_{S/2}) + V_2(d\log\eta([-1]_{S/2})), \cr
\end{aligned}$$
from which the stated formula follows by easy induction. Here, in the
last equality, we have used that $2dV_2\eta([1]_{S/2}) = V_2d\eta([1]_{S/2}) =
0$. Using this, we find that
$$\begin{aligned}
dV_2(d\log\eta([-1]_{S/2})) & = \sum_{r \geqslant 1}
2^rddV_{2^{r+1}}\eta([1]_{S/2^{r+1}}) \cr
{} & = \sum_{r \geqslant 1}
2^r d\log\eta([-1]_S)  \cdot dV_{2^{r+1}}\eta([1]_{S/2^{r+1}}) \cr
\end{aligned}$$
which is zero, since $2d\log\eta([-1]_S) = 0$. Now, using
Addendum~\ref{teichmuller}, we find that
$$\begin{aligned}
(d\log\eta([-1]_S))^2 & = (d\eta([-1]_S))^2 = (dV_2\eta([1]_{S/2}))^2 \cr
{} & = d(V_2\eta([1]_{S/2}) \cdot dV_2\eta([1]_{S/2})) -
V_2\eta([1]_{S/2}) \cdot ddV_2\eta([1]_{S/2}) \cr
{} & = dV_2(d\log\eta([-1]_{S/2})) - V_2\eta([1]_{S/2}) \cdot dV_2d\log\eta([-1]_{S/2}) \cr
{} & = dV_2(d\log\eta([-1]_{S/2})) \cdot \eta([1]_S - V_2([1]_{S/2})), \cr
\end{aligned}$$
which is zero, since the first factor in the bottom line is zero, by
what was just proved. This, in turn, shows that $(d\eta([-1]_S))^2 =
0$, from which we find that 
$$\begin{aligned}
dd\log\eta([-1]_S) & = d(\eta([-1]_S) \cdot d\eta([-1]_S)) \cr
{} & = d\eta([-1]_S) \cdot d\eta([-1]_S) + \eta([-1]_S) \cdot dd\eta([-1]_S) \cr
{} & = \eta([-1]_S)dd\eta([-1]_S) \cr
{} & =\eta([-1]_S)d\log\eta([-1]_S)d\eta([-1]_S) \cr
{} & = d\eta([-1]_S)d\eta([-1]_S) = 0. \cr
\end{aligned}$$
This completes the proof.
\end{proof}

\begin{prop}\label{frobeniuscompatible}For every Witt complex
$E_S^{\boldsymbol{\cdot}}$ over $A$ and every $m \in \N$, the diagram
$$\xymatrix{
{ \Omega_{\mathbb{W}_S(A)}^1 } \ar[r]^(.55){\eta_S} \ar[d]^{F_m} &
{ E_S^1 } \ar[d]^{F_m} \cr
{ \Omega_{\mathbb{W}_{S/m}(A)}^1 } \ar[r]^(.55){\eta_{S/m}} &
{ E_{S/m}^1 } \cr
}$$
commutes. Here, the horizontal maps take $a_0da_1$ to
$\eta(a_0)d\eta(a_1)$.
\end{prop}

\begin{proof}Since the restriction map $\smash{ R_S^{\N} \colon
\Omega_{\mathbb{W}(A)}^1 \to \Omega_{\mathbb{W}_S(A)}^1 }$ is
surjective and satisfies that both
$\smash{ F_mR_S^{\N} = R_{S/m}^{\N}F_m }$ and
$\smash{ \eta_S R_S^{\N} = R_S^{\N} \eta_{\N} }$, we may assume that
$S = \N$. In addition, by Remark~\ref{dlog[-1]}~(b), it will suffice to
show that for every $n \in \N$ and $a \in A$,
$$F_pdV_n\eta_{\N}([a]_{\N}) = \eta_{\N} F_pdV_n([a]_{\N})$$
as elements of $E_{\N}^1$. To ease notation, we suppress the subscript $\N$. We
first suppose that $p$ does not divide $n$ and set $k = (1-n^{p-1})/p$
and $l = n^{p-2}$ such that, in particular, $kp+ln = 1$. By
Lemma~\ref{wittcomplexrelations}, we have
$$\begin{aligned}
F_pdV_n\eta([a]) 
{} & = k \cdot dV_nF_p\eta([a]) + l \cdot V_nF_pd\eta([a]) \cr
{} & = k \cdot dV_n\eta([a]^p) + l \cdot
V_n(\eta([a])^{p-1}d\eta([a])) \cr
{} & = k \cdot dV_n\eta([a]^p) + l \cdot V_n\eta([a]^{p-1}d[a]) \cr
\end{aligned}$$
Moreover, arguing as in the proof of Lemma~\ref{frobeniusdescends}
above, we find that
$$\begin{aligned}
\eta F_pdV_n([a])
{} & = \eta(V_n([a])^{p-1} \cdot dV_n([a]) + d\Delta_pV_n([a])) \cr
{} & = \eta(l \cdot V_n([a]^{p-1}) \cdot dV_n([a]) + k \cdot dV_n([a]^p)) \cr
{} & = l \cdot V_n\eta([a]^{p-1}) \cdot dV_n\eta([a]) + k \cdot dV_n\eta([a]^p) \cr
{} & = l \cdot V_n(\eta([a]^{p-1}) \cdot F_ndV_n\eta([a])) + k \cdot dV_n\eta([a]^p) \cr
{} & = l \cdot V_n\eta([a]^{p-1}d[a]) + k \cdot dV_n\eta([a]^p) \cr
\end{aligned}$$
where the last equality uses that $n^{p-2}(n-1) d\log\eta([-1])$ is
zero. This proves that the desired equality holds if $p$ does not
divide $n$. Suppose next that $p$ divides $n$ and write $n = pr$. We
consider the cases $p = 2$ and $p$ odd separately. First, for $p$ odd, 
$$\begin{aligned}
F_pdV_n\eta([a])
{} & = dV_r\eta([a]) \cr
\eta F_pdV_n([a])
{} & = \eta(V_n([a])^{p-1} \cdot dV_n([a]) + d( \frac{F_pV_n([a]) -
  V_n([a])^p}{p} )) \cr
{} & = \eta(V_n([a])^{p-1} \cdot dV_n([a]) + dV_r([a]) -
p^{p-2}r^{p-1}dV_n([a]^p)) \cr
{} & = V_n(\eta([a]))^{p-1} \cdot dV_n\eta([a]) + dV_r\eta([a]) -
p^{p-2}r^{p-1}dV_n(\eta([a])^p), \cr
\end{aligned}$$
and the following calculation shows that the first and third terms cancel,
$$\begin{aligned}
{} p^{p-2}r^{p-1}dV_n(\eta([a])^p)  
{} & = p^{p-3}r^{p-2}V_nd(\eta([a])^p) \cr
{} & = p^{p-2}r^{p-2}V_n(\eta([a])^{p-1}d\eta([a])) \cr
{} & = n^{p-2}V_n(\eta([a])^{p-1}) \cdot dV_n\eta([a]) \cr
{} & = V_n(\eta([a]))^{p-1} \cdot dV_n\eta([a]). \cr
\end{aligned}$$
Here, the third equality follows from
Definition~\ref{bigwittcomplex}~(iii)--(iv) and from the fact that
$n^{p-2}(n-1)d\log[-1]$ vanishes. Finally, if $p = 2$, then
$$\begin{aligned}
F_2dV_n\eta([a]) & = dV_r\eta([a]) + d\log\eta([-1]) \cdot V_r\eta([a]) \cr
\eta F_2dV_n([a]) 
{} & = \eta(V_n([a]) \cdot dV_n([a]) + d( \frac{F_2V_n([a]) - V_n([a])^2}{2}
)) \cr
{} & = \eta(V_n([a]) \cdot dV_n([a]) + dV_r([a]) - rdV_n([a]^2)) \cr
{} & = V_n\eta([a]) \cdot dV_n\eta([a]) + dV_r\eta([a]) - rdV_n(\eta([a])^2), \cr
\end{aligned}$$
and hence, we must show that
$$d\log\eta([-1]) \cdot V_r\eta([a]) = V_n\eta([a]) \cdot dV_n\eta([a]) -
rdV_n(\eta([a])^2).$$
Suppose first that $r = 1$. By Addendum~\ref{teichmuller}, $[-1] =
-[1] + V_2([1])$, so that
$$d\log\eta([-1]) = V_2(\eta([1])) \cdot dV_2\eta([1]) - dV_2\eta([1]),$$
and hence, using the Witt complex axioms, we find
$$\begin{aligned}
{} & d\log\eta([-1]) \cdot \eta([a])
= V_2(\eta([a])^2) \cdot dV_2\eta([1]) - \eta([a]) \cdot dV_2\eta([1]) \cr
{} & = V_2(\eta([a])^2 \cdot F_2dV_2\eta([1])) - d(\eta([a]) \cdot V_2\eta([1]))
+ d(\eta([a])) \cdot V_2\eta([1]) \cr
{} & = V_2(\eta([a])^2 \cdot d\log\eta([-1])) - dV_2(\eta([a])^2) +
V_2(\eta([a])d\eta([a])) \cr
{} & = V_2(\eta([a])) \cdot dV_2\eta([a]) - dV_2(\eta([a])^2) \cr
\end{aligned}$$
as desired. In general, we apply $V_r$ to the formula that we just proved. This gives
$$\begin{aligned}
{} & d\log\eta([-1]) \cdot V_r(\eta([a]))
= V_r(V_2(\eta([a])) \cdot dV_2\eta([a])) - rdV_n(\eta([a])^2) \cr
{} & = V_n(\eta([a]) \cdot F_2dV_2\eta([a])) - rdV_n(\eta([a])^2) \cr
{} & = V_n(\eta([a]) \cdot F_ndV_n\eta([a])) - rdV_n(\eta([a])^2) \cr
{} & = V_n\eta([a]) \cdot dV_n\eta([a]) - rdV_n(\eta([a])^2), \cr
\end{aligned}$$
where we have $F_2dV_2 = F_ndV_n$, since $n$ is even. This completes
the proof.
\end{proof}

\begin{cor}\label{extendedeta}Let $E_S^{\cdot}$ be a
Witt complex over the ring $A$. There is a unique natural
homomorphism of graded rings 
$$\xymatrix{
{ \check{\Omega}_{\mathbb{W}_S(A)}^{\boldsymbol{\cdot}} } \ar[r]^-{\eta_S} &
{ E_S^{\boldsymbol{\cdot}} } \cr
}$$
that extends the natural ring homomorphism $\eta_S \colon \mathbb{W}_S(A) \to
E_S^0$ and commutes with the derivations.  In addition,
for every positive integer $m$, the diagram
$$\xymatrix{
{ \check{\Omega}_{\mathbb{W}_S(A)}^{\cdot} }
\ar[r]^(.55){\eta_S} \ar[d]^{F_m} &
{ E_S^{\cdot} } \ar[d]^{F_m} \cr
{ \check{\Omega}_{\mathbb{W}_{S/m}(A)}^{\cdot} }
\ar[r]^(.55){\eta_{S/m}} &
{ E_{S/m}^{\cdot} } \cr
}$$
commutes.
\end{cor}

\begin{proof}The map $\eta_S$ necessarily is given by
$$\eta_S(a_0da_1 \dots da_q) = \eta_S(a_0)d\eta_S(a_1) \dots
d\eta_S(a_q).$$
We show that this formula gives a well-defined map. First, from
Proposition~\ref{frobeniuscompatible}, we find that for all 
$a \in \mathbb{W}(A)$,
$$F_2d\eta_{\N}(a) = \eta_{\N}F_2d(a)
= \eta_{\N}(ada + d\Delta_2(a)) = \eta_{\N}(a)d\eta_{\N}(a) +
d\eta_{\N}\Delta_2(a).$$
Applying $d$ to this equation, the left-hand side becomes
$$dF_2d\eta_{\N}(a) = 2F_2dd\eta_{\N}(a) = 0$$
while the right-hand side becomes
$$\begin{aligned}
{} & d\eta_{\N}(a)d\eta_{\N}(a) + d\log\eta_{\N}([-1])_{\N} \cdot
(\eta_{\N}(a)d\eta_{\N}(a) + d\eta_{\N}\Delta_2(a)) \cr
{} & = d\eta_{\N}(a)d\eta_{\N}(a) + d\log\eta_{\N}([-1])_{\N} \cdot
F_2d\eta_{\N}(a). \cr
\end{aligned}$$
Hence, there is a well-defined map of graded rings $\eta_S
\colon \hat{\Omega}_{\mathbb{W}_S(A)}^{\boldsymbol{\cdot}} \to
E_S^{\boldsymbol{\cdot}}$ given by the formula stated at the
beginning of the proof, and by axiom~(iv) in
Definition~\ref{bigwittcomplex}, this map factors through the
canonical projection from $\smash{
  \hat{\Omega}_{\mathbb{W}_S(A)}^{\boldsymbol{\cdot}} }$ onto $\smash{
  \check{\Omega}_{\mathbb{W}_S(A)}^{\boldsymbol{\cdot}} }$. Finally,
Proposition~\ref{frobeniuscompatible} shows that the diagram in the
statement commutes.
\end{proof}

\begin{proof}[Proof of Theorem~\ref{bigdrwexists}]We recall that, in
the diagrams in the statement, the left-hand vertical maps were
defined in Lemma~\ref{variousmapsdescend}. We define maps of
graded rings
$$\xymatrix{ 
{ \check{\Omega}_{\mathbb{W}_S(A)}^{\boldsymbol{\cdot}} }
\ar[r]^-{\eta_S} &
{ \mathbb{W}_S\Omega_A^{\cdot} } \cr
}$$
as quotients by graded ideals $N_S^{\boldsymbol{\cdot}}$ and verify
that, in the diagrams in the statement, the right-hand vertical maps
$R_T^S$, $F_n$, and $d$ making the respective diagrams commute 
exist. We further define maps of graded abelian groups
$$\xymatrix{
{ \mathbb{W}_{S/n}\Omega_A^{\cdot} } \ar[r]^-{V_n} &
{ \mathbb{W}_S\Omega_A^{\cdot} } \cr
}$$
that make the following diagrams commute,
$$\xymatrix{
{ \mathbb{W}_{S/n}(A) } \ar[r]^{\eta_{S/n}} \ar[d]^{V_n} &
{ \mathbb{W}_{S/n}\Omega_A^0 } \ar[d]^{V_n} &
{ \mathbb{W}_{S/n}\Omega_A^{\boldsymbol{\cdot}} } \ar[r]^{V_n}
\ar[d]^{R_{T/n}^{S/n}} &
{ \mathbb{W}_S\Omega_A^{\boldsymbol{\cdot}} } \ar[d]^{R_T^S} \cr
{ \mathbb{W}_S(A) } \ar[r]^{\eta_S} &
{ \mathbb{W}_S\Omega_A^0 } &
{ \mathbb{W}_{T/n}\Omega_A^{\boldsymbol{\cdot}} } \ar[r]^{V_n} &
{ \mathbb{W}_T\Omega_A^{\boldsymbol{\cdot}} } \cr
}$$
$$\xymatrix{
{ \mathbb{W}_{S/n}\Omega_A^{\boldsymbol{\cdot}} \otimes
  \mathbb{W}_{S/n}\Omega_A^{\boldsymbol{\cdot}} } \ar[d]^{\mu} &
{ \mathbb{W}_{S/n}\Omega_A^{\boldsymbol{\cdot}} \otimes
  \mathbb{W}_S\Omega_A^{\boldsymbol{\cdot}} } \ar[l]_(.47){\id \otimes F_n}
\ar[r]^(.51){V_n \otimes \id} &
{ \mathbb{W}_S\Omega_A^{\boldsymbol{\cdot}} \otimes
  \mathbb{W}_S\Omega_A^{\boldsymbol{\cdot}} } \ar[d]^{\mu} \cr
{ \mathbb{W}_{S/n}\Omega_A^{\boldsymbol{\cdot}} } \ar[rr]^{V_n} &&
{ \mathbb{W}_S\Omega_A^{\boldsymbol{\cdot}}. } \cr
}$$
The definition of these maps, as $S$ ranges over all finite truncation
sets, $T \subset S$ over all sub-truncation sets, and $n$ over all
positive integers, will be by induction on the cardinality of $S$ and
will take up most of the proof. Once this is completed, we will show
that the combined structure is a Witt complex over $A$ and that 
it is initial among Witt complexes over $A$.

We define $\mathbb{W}_{\emptyset}\Omega_A^{\boldsymbol{\cdot}}$
to be the terminal graded ring, which is zero in all degrees, and
define $\eta_{\emptyset}$ to be the unique map of graded rings. So let
$S$ be a finite non-empty truncation set and 
assume, inductively, that the maps $\eta_T$, $R_U^T$, $F_n$, $d$, and
$V_n$ have been defined, for all proper truncation sets $T \subset S$,
all truncation sets $U \subset T$, and all positive integers $n$, with
the properties listed at the beginning of the proof. In this
situation, we define $\eta_S \colon
\check{\Omega}_{\mathbb{W}_S(A)}^{\boldsymbol{\cdot}} \to
\mathbb{W}_S\Omega_A^{\boldsymbol{\cdot}}$ to be the quotient map that
annihilates the graded ideal $N_S^{\,\cdot}$ generated by all sums
$$\sum_{\alpha} V_n(x_{\alpha})dy_{1,\alpha} \dots dy_{q,\alpha},
\hskip8mm
d( \sum_{\alpha} V_n(x_{\alpha})dy_{1,\alpha} \dots
dy_{q,\alpha}),$$
where the Witt vectors $x_{\alpha} \in \mathbb{W}_{S/n}(A)$ and
$y_{1,\alpha},\dots,y_{q,\alpha} \in \mathbb{W}_S(A)$ and the integers 
$n \geqslant 2$ and $q \geqslant 1$ are such that the sum
$$\eta_{S/n}(\sum_{\alpha} x_{\alpha}F_ndy_{1,\alpha} \dots F_ndy_{q,\alpha})$$
in $\mathbb{W}_{S/n}\Omega_A^q$ is zero; and for every
positive integer $n$, we define the map of graded abelian groups
$V_n \colon \mathbb{W}_{S/n}\Omega_A^{\boldsymbol{\cdot}} \to
\mathbb{W}_S\Omega_A^{\boldsymbol{\cdot}}$
by
$$V_n\eta_{S/n}(x F_ndy_1\dots F_ndy_q) = \eta_S(V_n(x)dy_1\dots
dy_q).$$
Here we use that every element of $\mathbb{W}_{S/n}\Omega_A^q$ can be
written as a sum of elements of the form $\eta_{S/n}(xF_ndy_1\dots F_ndy_q)$ 
with $x \in \mathbb{W}_{S/n}(A)$ and $y_1,\dots,y_q \in
\mathbb{W}_S(A)$. Indeed,
$$dx = F_ndV_n(x) - (n-1)d\log[-1]_{S/n} \cdot x
= F_ndV_n(x) - (n-1)xF_nd([-1]_S).$$
To prove the existence of the necessarily unique right-hand vertical
maps $R_T^S$, $d$, and $F_n$ making the diagrams in the statement of
the theorem commute, we must show that the left-hand vertical maps in
these diagrams satisfy 
$\eta_T(R_T^S(N_S^q)) = 0$, $\eta_S(d(N_S^q)) = 0$, and
$\eta_{S/m}(F_m(N_S^q)) = 0$, respectively, and to this end, we use 
the properties of the latter maps established in
Lemma~\ref{variousmapsdescend}. So we fix a positive integer $n$ and
an element
$$\omega = \sum_{\alpha} V_n(x_{\alpha})dy_{1,\alpha} \dots
dy_{q,\alpha} \in \check{\Omega}_{\mathbb{W}_S(A)}^q$$
with
$$\eta_{S/n}(\sum_{\alpha} x_{\alpha}F_ndy_{1,\alpha} \dots
F_ndy_{q,\alpha}) \in \mathbb{W}_{S/n}\Omega_A^q$$
equal to zero and show that $\eta_TR_T^S(\omega)$, $\eta_S(dd\omega)$,
$\eta_{S/m}F_m(\omega)$, and $\eta_{S/m}F_m(d\omega)$ all are
zero. First, in order to show that
$$\eta_TR_T^S(\omega)
= \eta_T(\sum_{\alpha} V_nR_{T/n}^{S/n}(x_{\alpha})dR_T^S(y_{1,\alpha})
\dots dR_T^S(y_{q,\alpha}))$$
is zero, it suffices by the definition of the ideal
$N_T^{\boldsymbol{\cdot}}$ to show that
$$\eta_{T/n}(\sum_{\alpha}R_{T/n}^{S/n}(x_{\alpha})F_ndR_T^S(y_{1,\alpha}) \dots
F_ndR_T^S(y_{q,\alpha}))$$
is zero. But this element is equal to
$$\eta_{T/n}R_{T/n}^{S/n}(\sum_{\alpha} x_{\alpha}F_ndy_{1,\alpha}
\dots F_ndy_{q,\alpha})$$
which, by the inductive hypothesis, is equal to
$$R_{T/n}^{S/n}\eta_{S/n}(\sum_{\alpha} x_{\alpha}F_ndy_{1,\alpha}
\dots F_ndy_{q,\alpha})$$
which we assumed to be zero. Similarly, we have
$$\begin{aligned}
\eta_S(dd\omega)
{} & = \eta_S(d\log([-1]_S) \cdot d\omega)
= \eta_S(d(d\log([-1]_S) \cdot \omega)) \cr
{} & = \eta_S(d(\sum_{\alpha}
V_n(x_{\alpha}[-1]_{S/n}^{-n})d([-1]_S)dy_{1,\alpha} \dots dy_{q,\alpha})), \cr
\end{aligned}$$
and by the definition of $N_S^{\boldsymbol{\cdot}}$, this element is zero, since
$$\begin{aligned}
{} & \eta_{S/n}(\sum_{\alpha}
x_{\alpha}[-1]_{S/n}^{-n}F_nd([-1]_S)F_ndy_{1,\alpha} \dots
F_ndy_{q,\alpha}) \cr
{} &= d\log\eta_{S/n}([-1]_{S/n}) \cdot \eta_{S/n}(\sum_{\alpha} x_{\alpha}F_ndy_{1,\alpha}
\dots F_ndy_{q,\alpha}) \cr
\end{aligned}$$
is zero. Next, to prove that $\eta_{S/m}F_m(\omega)$ and
$\eta_{S/m}F_m(d\omega)$ are zero, we may assume that $m = p$ is a
prime number. Indeed, if $m = kp$, then, by the inductive hypothesis,
we have $\eta_{S/m}F_m = \eta_{S/m}F_kF_p = F_k\eta_{S/p}F_p$. Suppose 
first that $n = lp$ is divisible by $p$. We remark that we have
\begin{equation}\label{term1}
\begin{aligned}
{} & \eta_{S/p}(\sum_{\alpha} V_l(x_{\alpha})F_pdy_{1,\alpha}\dots
F_pdy_{q,\alpha}) \cr
{} & = \sum_{\alpha} V_l(\eta_{S/n}(x_{\alpha}))\eta_{S/p}(F_pdy_{1,\alpha}\dots
F_pdy_{q,\alpha}) \cr
{} & = \sum_{\alpha} V_l(\eta_{S/n}(x_{\alpha})F_l(\eta_{S/p}(F_pdy_{1,\alpha}\dots
F_pdy_{q,\alpha}))) \cr
{} & = V_l\eta_{S/n}(\sum_{\alpha} x_{\alpha}F_ndy_{1,\alpha}\dots
F_ndy_{q,\alpha}) \cr
\end{aligned}
\end{equation}
which is zero. Indeed, the three equalities hold since, by the
inductive hypothesis, the maps $V_l \colon
\mathbb{W}_{S/n}\Omega_A^q \to \mathbb{W}_{S/p}\Omega_A^q$ and $F_l
\colon \mathbb{W}_{S/p}\Omega_A^q \to \mathbb{W}_{S/n}\Omega_A^q$
exist and have the properties listed at the beginning of the
proof. Now,
$$\begin{aligned}
\eta_{S/p}F_p(\omega)
{} & = \eta_{S/p}(\sum_{\alpha}F_pV_n(x_{\alpha})F_pdy_{1,\alpha} \dots
F_pdy_{q,\alpha}) \cr
{} & = p\eta_{S/p}(\sum_{\alpha} V_l(x_{\alpha})F_pdy_{1,\alpha}\dots
F_pdy_{q,\alpha}) \cr
\end{aligned}$$
which is zero by~(\ref{term1}). Similarly, using
Proposition~\ref{uniquederivation}, we have
$$\eta_{S/p}F_p(d\omega)
= \eta_{S/p}F_p(\sum_{\alpha}dV_n(x_{\alpha})dy_{1,\alpha} \dots
dy_{q,\alpha}) 
+ \epsilon \cdot \eta_{S/p}F_p(\omega)$$
with $\epsilon = qd\log\eta_{S/p}([-1]_{S/p})$, and we have just
proved that $\eta_{S/p}F_p(\omega)$ is zero. By
Lemma~\ref{F_ndV_n(a)}, we may rewrite the first summand as the sum
$$\eta_{S/p}(\sum_{\alpha} dV_l(x_{\alpha})F_pdy_{1,\alpha} \dots
F_pdy_{q,\alpha}) + \epsilon \cdot
\eta_{S/p}(\sum_{\alpha} V_l(x_{\alpha})F_pdy_{1,\alpha} \dots
F_pdy_{q,\alpha})$$
with $\epsilon =  (p-1) d\log\eta_{S/p}([-1]_{S/p})$. Here, the second
term is zero by~(\ref{term1}), and we rewrite the first summand as
$$\begin{aligned}
{} & \eta_{S/p}(d(\sum_{\alpha} V_l(x_{\alpha}) F_pdy_{1,\alpha} \dots
F_pdy_{q,\alpha}) 
- \sum_{\alpha}
V_l(x_{\alpha})d(F_pdy_{1,\alpha} \dots F_pdy_{q,\alpha})) \cr
{} & = \eta_{S/p}d(\sum_{\alpha} V_l(x_{\alpha}) F_pdy_{1,\alpha} \dots
F_pdy_{q,\alpha}) - \epsilon \cdot
\eta_{S/p}(\sum_{\alpha}
V_l(x_{\alpha})F_pdy_{1,\alpha} \dots F_pdy_{q,\alpha}) \cr
{} & = d\eta_{S/p}(\sum_{\alpha} V_l(x_{\alpha}) F_pdy_{1,\alpha} \dots
F_pdy_{q,\alpha}) - \epsilon \cdot
\eta_{S/p}(\sum_{\alpha}
V_l(x_{\alpha})F_pdy_{1,\alpha} \dots F_pdy_{q,\alpha}) \cr
\end{aligned}$$
with $\epsilon = pqd\log\eta_{S/p}([-1]_{S/p})$. Here, the last
equality uses that $\eta_{S/p}$, by definition, commutes with
$d$. It follows from~(\ref{term1}) that both summands in the last line
vanish, so $\eta_{S/p}F_p(d\omega) = 0$ as desired. Next, suppose that
$p$ does not divide $n$. We have
$$\begin{aligned}
\eta_{S/p}F_p(\omega)
{} & = \sum_{\alpha} \eta_{S/p}V_nF_p(x_{\alpha}) \cdot
\eta_{S/p}F_p(dy_{1,\alpha} \dots dy_{q,\alpha}) \cr
{} & = \sum_{\alpha} V_nF_p\eta_{S/n}(x_{\alpha}) \cdot
\eta_{S/p}F_p(dy_{1,\alpha} \dots dy_{q,\alpha}) \cr
{} & = \sum_{\alpha} V_n(F_p\eta_{S/n}(x_{\alpha}) \cdot
F_n\eta_{S/p}F_p(dy_{1,\alpha} \dots dy_{q,\alpha})) \cr 
{} & = \sum_{\alpha} V_n(F_p\eta_{S/n}(x_{\alpha}) \cdot
\eta_{S/np}F_{np}(dy_{1,\alpha} \dots dy_{q,\alpha})) \cr 
{} & = \sum_{\alpha} V_n(F_p\eta_{S/n}(x_{\alpha}) \cdot
F_p\eta_{S/n}F_n(dy_{1,\alpha} \dots dy_{q,\alpha})) \cr
{} & = V_nF_p\eta_{S/n}(\sum_{\alpha} x_{\alpha}F_ndy_{1,\alpha} \dots
F_ndy_{q,\alpha}) \cr
\end{aligned}$$
which is zero. Here the second, third, and forth equalities use that,
by the inductive hypothesis, the maps $F_n \colon
\mathbb{W}_{S/p}\Omega_A^q \to \mathbb{W}_{S/np}\Omega_A^q$ and $V_n
\colon \mathbb{W}_{S/np}\Omega_A^q \to \mathbb{W}_{S/p}\Omega_A^q$
exist and have the properties listed at the beginning of the proof,
and similarly, the fifth equality uses that the map $F_p \colon
\mathbb{W}_{S/n}\Omega_A^q \to \mathbb{W}_{S/np}\Omega_A^q$ with
$F_p\eta_{S/n} = \eta_{S/np}F_p$ exists. We proceed to show that also
$\eta_{S/p}F_p(d\omega)$ vanishes, and to this end, it suffices to
show that both $p\eta_{S/p}F_P(d\omega)$ and $n\eta_{S/p}F_p(d\omega)$
vanish. First,
$$p\eta_{S/p}F_pd(\omega) = \eta_{S/p}dF_p(\omega) =
d\eta_{S/p}F_p(\omega),$$
which is zero by what was just proved. Here the two equalities hold
by Proposition~\ref{omegafrobenius} and by the definition of
$\eta_{S/p}$, respectively. Next,
$$\begin{aligned}
n\eta_{S/p}F_pd(\omega) 
{} & = \sum_{\alpha}
n\eta_{S/p}F_pd(V_n(x_{\alpha})dy_{1,\alpha} \dots dy_{q,\alpha}) \cr
{} & = \sum_{\alpha} n\eta_{S/p}F_p(dV_n(x_{\alpha})dy_{1,\alpha}
\dots dy_{q,\alpha}) + \epsilon \cdot \eta_{S/p}F_p(\omega) \cr
\end{aligned}$$
with $\epsilon = nqd\log\eta_{S/p}([-1])_{S/p}$, and we have already
proved that $\eta_{S/p}F_p(\omega)$ is zero. Moreover, we may rewrite
the first term in the lower line as
$$\begin{aligned}
{} & \sum_{\alpha}\eta_{S/p}F_p(dV_n([1]_{S/n})V_n(x_{\alpha})dy_{1,\alpha}
\dots dy_{q,\alpha}) \cr
{} &  + \sum_{\alpha}\eta_{S/p}F_p(V_n([1]_{S/n})dV_n(x_{\alpha})dy_{1,\alpha}
\dots dy_{q,\alpha}), \cr
\end{aligned}$$
since, by Lemma~\ref{variousrelations} and by $d$ being a derivation,
$$\begin{aligned}
ndV_n(x) 
{} & = dV_nF_nV_n(x) = d(V_n([1]_{S/n}) \cdot V_n(x)) \cr
{} & = dV_n([1]_{S/n}) \cdot V_n(x) + V_n([1]_{S/n}) \cdot dV_n(x). \cr
\end{aligned}$$
Now, since both $\eta_{S/p}$ and $F_p$ are graded ring homomorphisms, we
have
$$\sum_{\alpha}
\eta_{S/p}F_p(dV_n([1]_{S/n})V_n(x_{\alpha})dy_{1,\alpha} \dots
dy_{q,\alpha})
= \eta_{S/p}F_pdV_n([1]_{S/n}) \cdot \eta_{S/p}F_p(\omega)$$
which is zero, since $\eta_{S/p}F_p(\omega)$ is zero, and
$$\begin{aligned}
{} & \sum_{\alpha}
\eta_{S/p}F_p(V_n([1]_{S/n})dV_n(x_{\alpha})dy_{1,\alpha} \dots
dy_{q,\alpha}) \cr
{} & = \sum_{\alpha} \eta_{S/p}F_pV_n([1]_{S/n}) \cdot
\eta_{S/p}F_p(dV_n(x_{\alpha})dy_{1,\alpha} \dots dy_{q,\alpha}) \cr
{} & = \sum_{\alpha} \eta_{S/p}V_nF_p([1]_{S/n}) \cdot
\eta_{S/p}F_p(dV_n(x_{\alpha})dy_{1,\alpha} \dots dy_{q,\alpha}) \cr
{} & = \sum_{\alpha} V_n\eta_{S/np}F_p([1]_{S/n}) \cdot
\eta_{S/p}F_p(dV_n(x_{\alpha})dy_{1,\alpha} \dots dy_{q,\alpha}) \cr
{} & = \sum_{\alpha} V_n(\eta_{S/np}F_p([1]_{S/n}) \cdot
F_n\eta_{S/p}F_p(dV_n(x_{\alpha})dy_{1,\alpha} \dots dy_{q,\alpha})) \cr
{} & = \sum_{\alpha} V_n(\eta_{S/np}F_p([1]_{S/n}) \cdot
\eta_{S/np}F_{np}(dV_n(x_{\alpha})dy_{1,\alpha} \dots dy_{q,\alpha})) \cr
{} & = \sum_{\alpha} V_n(\eta_{S/np}F_p([1]_{S/n}) \cdot
F_p\eta_{S/n}F_n(dV_n(x_{\alpha})dy_{1,\alpha} \dots dy_{q,\alpha})), \cr
\end{aligned}$$
where the third and forth equalities hold, since, by the inductive
hypothesis, both the maps $F_n \colon \mathbb{W}_{S/p}\Omega_A^q \to
\mathbb{W}_{S/np}\Omega_A^q$ and $V_n \colon
\mathbb{W}_{S/np}\Omega_A^q \to \mathbb{W}_{S/p}\Omega_A^q$ exist and
have the properties listed at the beginning of the proof, and where
the fifth and sixth equalities hold, since the maps $F_p
\colon \mathbb{W}_{S/n}\Omega_A^q \to \mathbb{W}_{S/np}\Omega_A^q$ and
$V_p \colon \mathbb{W}_{S/np}\Omega_A^q \to
\mathbb{W}_{S/n}\Omega_A^q$ exist and have the properties listed at
the beginning of the proof. Since $\eta_{S/np}F_p([1]_{S/n})$ is the
identity and $F_ndV_n(x_{\alpha}) = dx_{\alpha} +
(n-1)d\log([-1]_S) \cdot x_{\alpha}$, this becomes
$$V_nF_p\eta_{S/n}(\sum_{\alpha} d(x_{\alpha})F_ndy_{1,\alpha} \dots
F_ndy_{q,\alpha}) + \epsilon \cdot V_nF_p\eta_{S/n}(\sum_{\alpha}
x_{\alpha}F_ndy_{1,\alpha} \dots F_ndy_{q,\alpha})$$
with $\epsilon = (n-1)d\log\eta_{S/p}([-1]_{S/p})$, where the second
term is zero. Finally, since $d$ is a derivation, we may rewrite the
first term as
$$V_nF_p\eta_{S/n}d(\sum_{\alpha} x_{\alpha}F_ndy_{1,\alpha} \dots
F_ndy_{q,\alpha}) - \epsilon \cdot V_nF_p\eta_{S/n}(\sum_{\alpha}
x_{\alpha}F_ndy_{1,\alpha} \dots F_ndy_{q,\alpha})$$
with $\epsilon = nqd\log\eta_{S/p}([-1]_{S/p})$, and these terms both are
zero. Hence, $nF_pd(\omega)$ is zero, and therefore, we conclude
that $F_pd(\omega)$ is zero as desired. 

In order to complete the recursive definition of the maps $\eta_S$,
$R_T^S$, $F_n$, $d$, and $V_n$, we must show that the three diagrams
at the beginning of the proof commute. The top left-hand diagram
commutes by the definition of $V_n$, and the calculation
$$\begin{aligned}
{} & R_T^SV_n\eta_{S/n}(xF_ndy_1 \dots F_ndy_q)
= R_T^S\eta_S(V_n(x)dy_1 \dots dy_q) \cr
{} & = \eta_TR_T^S(V_n(x)dy_1 \dots dy_q)
= \eta_T(V_nR_{T/n}^{S/n}(x)dR_T^S(y_1) \dots dR_T^S(y_q)) \cr
{} & = V_n\eta_{T/n}(R_{T/n}^{S/n}(x)F_ndR_T^S(y_1) \dots
F_ndR_T^S(y_q))
= V_nR_{T/n}^{S/n}\eta_{S/n}(xF_ndy_1 \dots F_ndy_q) \cr
\end{aligned}$$
shows that the top right-hand diagram commutes. Finally, the following
calculation shows that the bottom diagram commutes,
$$\begin{aligned}
{} & V_n\eta_{S/n}(xF_ndy_1 \dots F_ndy_q) \cdot \eta_S(zdw_1 \dots
dw_r) \cr
{} & = \eta_S(V_n(x)dy_1 \dots dy_q) \cdot \eta_S(zdw_1 \dots dw_r) \cr
{} & = \eta_S(V_n(x)dy_1 \dots dy_q \cdot zdw_1 \dots dw_r) \cr
{} & = \eta_S(V_n(x F_n(z))dy_1 \dots dy_qdw_1 \dots dw_r) \cr
{} & = V_n\eta_{S/n}(xF_n(z)F_ndy_1 \dots F_ndy_qF_ndw_1
\dots F_ndw_r) \cr
{} & = V_n(\eta_{S/n}(xF_ndy_1 \dots F_ndy_q) \cdot \eta_{S/n}F_n(zdw_1
\dots dw_r)) \cr
{} & = V_n(\eta_{S/n}(xF_ndy_1 \dots F_ndy_q) \cdot F_n\eta_S(zdw_1
\dots dw_r)). \cr
\end{aligned}$$
Here the first and fourth equalities hold by the definition of the
map $V_n$; the second and fifth equalties hold by the multiplicativity
of the maps $\eta_S$, $\eta_{S/n}$, and $F_n$; the third equality
holds by Lemma~\ref{variousrelations}; and the sixth equality holds
by the existence of the map $F_n$ with $F_n\eta_S =
\eta_{S/n}F_n$. This completes the recursive definition of the graded rings
$\mathbb{W}_S\Omega_A^{\boldsymbol{\cdot}}$ and the maps $\eta_S$,
$R_T^S$, $F_n$, $d$, and $V_n$ for finite truncation sets $S$. We
extend to infinite truncation sets as discussed in
Remark~\ref{dlog[-1]}~(b). 

To show that the structure defined above forms a Witt complex over
$A$, it remains to prove that $V_1 = \id$; $V_nV_m = V_{mn}$; $F_nV_n
= n \cdot \id$; and $F_mV_n = V_nF_m$, if $(m,n) = 1$. The first
identity holds by definition, and the second identity holds, since
$$\begin{aligned}
{} & V_{mn}\eta_{S/mn}(xF_{mn}dy_1 \dots F_{mn}dy_q)
= \eta_S(V_{mn}(x)dy_1 \dots dy_q) \cr
{} & = \eta_S(V_m(V_n(x))dy_1 \dots dy_q) 
= V_m\eta_{S/m}(V_n(x)F_mdy_1 \dots F_mdy_q) \cr
{} & = V_m(V_n(\eta_{S/mn}(x))F_md\eta_S(y_1) \dots
F_md\eta_S(y_q)) \cr
{} & = V_m(V_n(\eta_{S/mn}(x)F_{mn}d\eta_S(y_1) \dots
F_{mn}d\eta_S(y_q))) \cr
{} & = V_m(V_n\eta_{S/mn}(xF_{mn}dy_1 \dots F_{mn}dy_q)). \cr
\end{aligned}$$
Here the first and third equalities hold by the definition of
$V_{mn}$ and $V_m$, respectively; the fourth equality holds by the
existence of the map $F_m$ with $\eta_{S/m}F_m = F_m \eta_S$; the
fifth equality holds by the inductive hypothesis; and the last
equality holds by the existence of the map $F_{mn}$ with
$\eta_{S/mn}F_{mn} = F_{mn}\eta_S$. Similarly, we have
$$\begin{aligned}
{} & F_nV_n\eta_{S/n}(xF_ndy_1 \dots F_ndy_q)
= F_n\eta_S(V_n(x)dy_1 \dots dy_q) \cr
{} & = \eta_{S/n}F_n(V_n(x) dy_1 \dots dy_q)
= n\eta_{S/n}(xF_ndy_1 \dots F_ndy_q), \cr
\end{aligned}$$
which shows that $F_nV_n = n \cdot \id$; and finally, if $(m,n) = 1$,
we have
$$\begin{aligned}
{} & F_mV_n\eta_{S/n}(x \cdot F_ndy_1 \dots F_ndy_q)
= F_m\eta_S(V_n(x) \cdot dy_1 \dots dy_q) \cr
{} & = \eta_{S/m}F_m(V_n(x) \cdot dy_1 \dots dy_q) 
= \eta_{S/m}F_m(V_n(x)) \cdot \eta_{S/m}F_m(dy_1 \dots dy_q) \cr
{} & = \eta_{S/m}V_n(F_m(x)) \cdot \eta_{S/m}F_m(dy_1 \dots dy_q) 
= V_n(\eta_{S/mn}F_m(x)) \cdot \eta_{S/m}F_m(dy_1 \dots dy_q) \cr
{} & = V_n(\eta_{S/mn}F_m(x) \cdot F_n\eta_{S/m}F_m(dy_1 \dots dy_q))
= V_n\eta_{S/mn}F_m(x \cdot F_n(dy_1 \dots dy_q)) \cr
{} & = V_nF_m\eta_{S/n}(x \cdot F_n(dy_1 \dots dy_q))
= V_nF_m\eta_{S/n}(x \cdot F_ndy_1 \dots F_ndy_q), \cr
\end{aligned}$$
which shows that $F_mV_n = V_nF_m$ as desired.

Finally, let $E_S^{\boldsymbol{\cdot}}$ be a Witt complex over $A$ and
let $\eta_S^E \colon
\check{\Omega}_{\mathbb{W}_S(A)}^{\boldsymbol{\cdot}} \to
E_S^{\boldsymbol{\cdot}}$ be the map in
Corollary~\ref{extendedeta}. We claim that this map factors as
$$\xymatrix{
{ \check{\Omega}_{\mathbb{W}_S(A)}^{\boldsymbol{\cdot}} }
\ar[r]^-{\eta_S} &
{ \mathbb{W}_S\Omega_A^{\boldsymbol{\cdot}} } \ar[r]^-{f_S} &
{ E_S^{\boldsymbol{\cdot}}. } \cr
}$$
Since the left-hand map $\eta_S$ is surjective, the right-hand map
$f_S$ necessarily is unique. We may further assume that the truncation
set $S$ is finite. To prove the claim, we proceed by induction on the
cardinality of $S$, the case $S = \emptyset$ being trivial as
$\eta_{\emptyset}$ is a bijection. So we let $S$ be a finite non-empty
truncation set and assume that for every proper sub-truncation set $T
\subset S$, the factorization $\eta_T^E = f_T\eta_T$ exists. To prove
that also the factorization $\eta_S^E = f_S\eta_S$ exists, we must show
that whenever $n$ is a positive integer and $x_{\alpha} \in
\mathbb{W}_{S/n}(A)$ and $y_{1,\alpha}, \dots, y_{q,\alpha} \in
\mathbb{W}_S(A)$ are Witt vectors such that
$$\eta_{S/n}(\sum_{\alpha}x_{\alpha}F_ndy_{1,\alpha} \dots
F_ndy_{q,\alpha}) \in \mathbb{W}_{S/n}\Omega_A^q$$
vanishes, then so does
$$\eta_S^E(\sum_{\alpha}V_n(x_{\alpha})dy_{1,\alpha} \dots
dy_{q,\alpha}) \in E_S^q.$$
Now, using that $E_S^{\boldsymbol{\cdot}}$ is a Witt complex over $A$,
we find
$$\begin{aligned}
{} & \eta_S^E(\sum_{\alpha}V_n(x_{\alpha}) \cdot dy_{1,\alpha} \dots
dy_{q,\alpha}) 
= \sum_{\alpha}\eta_S^E(V_n(x_{\alpha})) \cdot \eta_S^E(dy_{1,\alpha}
\dots dy_{q,\alpha}) \cr
{} & =  \sum_{\alpha} V_n(\eta_{S/n}^E(x_{\alpha})) \cdot \eta_S^E(dy_{1,\alpha}
\dots dy_{q,\alpha}) 
= \sum_{\alpha} V_n(\eta_{S/n}^E(x_{\alpha}) \cdot
F_n\eta_S^E(dy_{1,\alpha} \dots dy_{q,\alpha})) \cr
{} & = \sum_{\alpha} V_n(\eta_{S/n}^E(x_{\alpha}) \cdot
\eta_{S/n}^EF_n(dy_{1,\alpha} \dots dy_{q,\alpha})) 
= V_n\eta_{S/n}^E(\sum_{\alpha} x_{\alpha}F_ndy_{1,\alpha} \dots
F_ndy_{q,\alpha}) \cr
{} & = V_nf_{S/n}\eta_{S/n}(\sum_{\alpha} x_{\alpha}F_ndy_{1,\alpha}
\dots F_ndy_{q,\alpha}), \cr
\end{aligned}$$
which vanishes as required. Here the last equality holds by the
inductive hypothesis. This shows that, for every truncation set $S$,
the map $\eta_S^E$ factors as $f_S\eta_S$. Finally, to show that
$\mathbb{W}_S\Omega_A^{\boldsymbol{\cdot}}$ is initial among Witt
complexes over $A$, it remains to verify that the maps $f_S$
constitute a map of Witt complexes. In view of
Corollary~\ref{extendedeta}, the only statement that needs proof is
that for every truncation set $S$ and for every positive integer $n$,
we have $f_SV_n = V_nf_{S/n}$, and this follows from the calculation
$$\begin{aligned}
{} & f_SV_n\eta_{S/n}(x \cdot F_ndy_1 \dots F_ndy_q)
= f_S\eta_S(V_n(x) \cdot dy_1 \dots dy_q) \cr
{} & = \eta_S^E(V_n(x) \cdot dy_1 \dots dy_q)
= \eta_S^E(V_n(x)) \cdot \eta_S^E(dy_1 \dots dy_q) \cr
{} & = V_n(\eta_{S/n}^E(x)) \cdot \eta_S^E(dy_1 \dots dy_q) 
= V_n(\eta_{S/n}^E(x) \cdot F_n\eta_S^E(dy_1 \dots dy_q)) \cr
{} & = V_n(\eta_{S/n}^E(x) \cdot \eta_{S/n}^EF_n(dy_1 \dots dy_q))
= V_n\eta_{S/n}^E(x \cdot F_ndy_1 \dots F_ndy_q) \cr
{} & = V_nf_{S/n}\eta_{S/n}(x \cdot F_ndy_1 \dots F_ndy_q), \cr
\end{aligned}$$
since every element in $\mathbb{W}_S\Omega_A^q$ can be written as a
convergent sum of elements of the form $\eta_{S/n}(x \cdot F_ndy_1
\dots F_ndy_q)$ with $n \in \N$, and $x \in \mathbb{W}_S(A)$, and
$y_1, \dots, y_q \in \mathbb{W}_{S/n}(A)$. This completes the proof of
Theorem~\ref{bigdrwexists}. 
\end{proof}

\begin{definition}\label{derhamwittcomplex}The initial Witt complex
$\smash{ \mathbb{W}_S\Omega_A^{\boldsymbol{\cdot}} }$ over the ring
$A$ is called the big de~Rham-Witt complex of $A$.
\end{definition}

\begin{addendum}\label{whydrw}
\begin{enumerate}
\item[{\rm(i)}] For all $q$, the map $\eta_{\{1\}} \colon \Omega_A^q
  \to \mathbb{W}_{\{1\}}\Omega_A^q$ is an isomorphism.
\item[{\rm(ii)}] For all $S$, the map $\eta_S \colon \mathbb{W}_S(A) \to
\mathbb{W}_S\Omega_A^0$ is an isomorphism. 
\end{enumerate}
\end{addendum}

\begin{proof}This follows immediately from the proof of
Theorem~\ref{bigdrwexists} and Lemma~\ref{S=1}.
\end{proof}

The statement~(i) in Addendum~\ref{whydrw} is a special case of the
question raised at the top of page~133 in~\cite{hm2}. The explicit
construction of the big de~Rham-Witt complex given in the proof of
Theorem~\ref{bigdrwexists} answers this question in the affirmative.

\begin{remark}\label{relativedrw}Suppose that $(k,\lambda)$ is a
$\lambda$-ring and that $f \colon k \to A$ is a $k$-algebra. We let
$f_S \colon k \to \mathbb{W}_S(A)$ be the composite ring homomorphism
$R_S^{\N} \circ \mathbb{W}(f) \circ \lambda$ and define the big
de~Rham-Witt complex of $A$ relative to $(k,\lambda)$ to be the quotient
$$\mathbb{W}_S\Omega_{A/(k,\lambda)}^{\boldsymbol{\cdot}} =
\mathbb{W}_S\Omega_A^{\boldsymbol{\cdot}}/R_S^{\boldsymbol{\cdot}}$$
of the big de~Rham-Witt complex $\mathbb{W}_S\Omega_A^{\boldsymbol{\cdot}}$ by the
graded ideal $R_S^{\boldsymbol{\cdot}}$ generated by the images of
$\eta_S \circ f_{S*} \colon \Omega_k^1 \to \mathbb{W}_S\Omega_A^1$ and
$d \circ \eta_S \circ f_{S*} \colon \Omega_k^1 \to
\mathbb{W}_S\Omega_A^2$. It is initial among Witt complexes over $A$
in which the map $d$ is $k$-linear when its domain and target are
viewed as $k$-modules via the map $\eta_Sf_S \colon k \to E_S^0$. In
the particular case, where $(k,\lambda)$ is $(\mathbb{W}(R),\Delta_R)$
and where $A$ is an $R$-algebra viewed as a $k$-algebra via
$\epsilon_R \colon k \to R$, we obtain (a big version of) the
Langer-Zink relative de~Rham-Witt complex~\cite{langerzink}.
\end{remark}

\section{\'{E}tale morphisms}\label{etalesection}

The functor that to the ring $A$ associates the
$\mathbb{W}_S(A)$-module $\mathbb{W}_S\Omega_A^q$ defines a presheaf
of $\mathbb{W}_S(\mathscr{O})$-modules on the category of affine
schemes. In this section, we use the theorem of Borger~\cite{borger} and van der
Kallen~\cite{vanderkallen1} which we recalled as
Theorem~\ref{borgervanderkallentheorem} to show that for $S$ finite,
this presheaf is a quasi-coherent sheaf of
$\mathbb{W}_S(\mathscr{O})$-modules for the \'{e}tale topology. This
is the statement of Theorem~\ref{etaledescent} which we now prove.

\begin{proof}[Proof of Theorem~\ref{etaledescent}]We fix an \'{e}tale
morphism $f \colon A \to B$ and consider the map
$$\xymatrix{
{ \mathbb{W}_S(B)
  \otimes_{\mathbb{W}_S(A)}\mathbb{W}_S\Omega_A^{\boldsymbol{\cdot}} }
  \ar[r]^-{\alpha} &
{ \mathbb{W}_S\Omega_B^{\boldsymbol{\cdot}} } \cr
}$$
that to $b \otimes \omega$ assigns $b \cdot f_*(\omega)$. To show that
this map is an isomorphism, we define a structure of Witt complex over
$B$ on the domain $E_S^{\boldsymbol{\cdot}}$ of $\alpha$. By
Theorem~\ref{borgervanderkallentheorem}, the map
$$\mathbb{W}_S(f) \colon \mathbb{W}_S(A) \to \mathbb{W}_S(B)$$
is \'{e}tale. Hence, the graded derivation $\smash{ d \colon
\mathbb{W}_s\Omega_A^q \to \mathbb{W}_S\Omega_A^{q+1} }$ extends
uniquely to a graded derivation $d^E \colon E_S^q \to E_S^{q+1}$
defined by
$$d^E(b \otimes x) = (d'b)x + b \otimes dx$$
with $d'b$ the image of $b$ by the composition
$$\xymatrix{
{ \mathbb{W}_S(B) } \ar[r]^-{d} &
{ \Omega_{\mathbb{W}_S(B)}^1 } &
{ \mathbb{W}_S^{\phantom{U}}(B) \otimes_{\mathbb{W}_S(A)}
\Omega_{\mathbb{W}_S(A)}^1 } \ar[l]_-{\sim} \ar[r]^-(.48){\id \otimes
\eta_S} &
{ \mathbb{W}_S(B) \otimes_{\mathbb{W}_S(A)} \mathbb{W}_S\Omega_A^1, } \cr
}$$
where the middle map is the canonical isomorphism. We further define
the maps $\smash{ R_T^{E,S} \colon E_S^q \to E_T^q }$ and $\smash{
  F_n^E \colon E_S^q \to E_{S/n}^q }$ to be $\smash{ R_T^S \otimes
  R_T^S }$ and $\smash{ F_n^E = F_n \otimes F_n }$,
respectively. Next, to define the map $\smash{ V_n^E 
  \colon E_{S/n}^q   \to E_S^q }$, we use that, since the square in
the statement of  Theorem~\ref{borgervanderkallentheorem} is
cocartesian, the map
$$\xymatrix{
{ \mathbb{W}_S(B) \otimes_{\mathbb{W}_S(A)} \mathbb{W}_{S/n}\Omega_A^q
} \ar[r]^-{F_n \otimes \id} &
{ \mathbb{W}_{S/n}(B) \otimes_{\mathbb{W}_{S/n}(A)}
\mathbb{W}_{S/n}\Omega_A^q } \cr
}$$
is an isomorphism, and we then define $V_n^E$ to be the composition of
the inverse of this isomorphism and the map
$$\xymatrix{
{ \mathbb{W}_S(B) \otimes_{\mathbb{W}_S(A)} \mathbb{W}_{S/n}\Omega_A^q
} \ar[r]^-{\id \otimes V_n} &
{ \mathbb{W}_S(B) \otimes_{\mathbb{W}_{S}(A)}
  \mathbb{W}_{S}\Omega_A^q. } \cr
}$$
Finally, we define the map $\eta_S^E \colon \mathbb{W}_S(B) \to E_B^0$
to be the composition 
$$\xymatrix{
{ \mathbb{W}_S(B) } \ar[r] &
{ \mathbb{W}_S(B) \otimes_{\mathbb{W}_S(A)} \mathbb{W}_S(A) } \ar[r] &
{ \mathbb{W}_S(B) \otimes_{\mathbb{W}_S(A)} \mathbb{W}_S\Omega_A^0 }
\cr
}$$
of the canonical isomorphism and the map $\id \otimes \eta_S$. We
proceed to show that the maps defined above make
$E_S^{\boldsymbol{\cdot}}$ a Witt complex over $B$.  The
axioms~(i)--(iii) of Definition~\ref{bigwittcomplex} are readily
verified. For example, we have $d^Ed^E(-) =
d\log\eta_S^E([-1]_S)d^E(-)$, since both sides are derivations which
agree on $\mathbb{W}_S\Omega_A^q$; and the calculation
$$\begin{aligned}
{} & V_n^E(F_n(b) \otimes \omega) \cdot b' \otimes \omega' = b \otimes
V_n(\omega) \cdot b' \otimes \omega'  = bb' \otimes V_n(\omega)\omega' \cr
{} & = bb' \otimes V_n(\omega F_n(\omega')) = V_n^E(F_n(bb') \otimes
\omega F_n(\omega')) = V_n^E(F_n(b) \otimes \omega \cdot F_n^E(b' \otimes
\omega')) \cr
\end{aligned}$$
verifies axiom~(iii), since every element of $\smash{ E_{S/n}^q }$ can be
written as a sum of elements of the form $F_n(b) \otimes \omega$ with
$b \in \mathbb{W}_S(B)$ and $\omega \in \mathbb{W}_{S/n}\Omega_A^q$.

It remains to verify axioms~(iv)--(v) of
Definition~\ref{bigwittcomplex}. To prove axiom~(iv), we must show
that for all $\omega \in E_{S/n}^q$, the equality
$$F_n^Ed^EV_n^E(\omega) + (n-1)d\log\eta_{S/n}^E([-1]_{S/n}) \cdot
\omega = d^E(\omega)$$
holds in $\smash{ E_{S/n}^{q+1} }$. On the right-hand side, $d^E$, by
definition, is the unique graded derivation on $\smash{
  E_{S/n}^{\boldsymbol{\cdot}} }$ that extends the graded derivation
$d$ on $\smash{ \mathbb{W}_{S/n}\Omega_A^{\boldsymbol{\cdot}} }$; and,
on the left-hand side, $D$ also extends $d$. Hence, it will suffice to
show that $D$, too, is a graded derivation. Moreover, since the square
diagram of rings in Theorem~\ref{borgervanderkallentheorem} is cocartesian,
and since $D$ is an additive function, it suffices to show that $D$ is
a graded derivation on elements of the form $\omega =
F_n(b) \otimes \tau$ with $b \in \mathbb{W}_S(B)$ and $\tau \in
\mathbb{W}_{S/n}\Omega_A^{\boldsymbol{\cdot}}$. We claim that
$$D(F_n(b) \otimes \tau) = F_n^E(d^Eb) \cdot n\tau + F_n(b) \otimes
d\tau$$
as elements of $\smash{ E_{S/n}^{\boldsymbol{\cdot}} }$. Granting
this, it follows that axiom~(iv) holds, as the right-hand side clearly
is a graded derivation of $F_n(b) \otimes \tau$. Now,
$$\begin{aligned}
{} & D(F_n(b) \otimes \tau) = F_n^Ed^EV_n^E(F_n(b) \otimes \tau) +
(n-1)d\log\eta_{S/n}^E([-1]_{S/n}) \cdot F_n(b) \otimes \tau \cr
{} & = F_n^Ed^E(b \otimes V_n(\tau)) + F_n(b) \otimes
(n-1)d\log\eta_{S/n}([-1])\tau \cr 
{} & = F_n^E(d^E(b) \cdot V_n(\tau) + b \otimes dV_n(\tau)) + F_n(b) \otimes
(n-1)d\log\eta_{S/n}([-1])\tau \cr
{} & = F_n^E(d^Eb) \cdot n\tau + F_n(b) \otimes (F_ndV_n(\tau) +
(n-1)d\log\eta_{S/n}([-1])\tau) \cr
{} & = F_n^E(d^Eb) \cdot n\tau + F_n(b) \otimes d\tau, \cr
\end{aligned}$$
which proves the claim. Here the first two equalities follow from the
definitions; the third equality holds, since $d^E$ is a derivation;
the fourth equality holds, since $F_n^E$ is a ring homomorphism and
satisfies $F_n^EV_n^E = n \id$; and the last equality holds, since
axiom~(iv) holds in the de~Rham-Witt complex over $A$. 

In order to prove axiom~(v), we consider the following diagram, where
the left-hand horizontal maps are the canonical isomorphisms,
$$\begin{xy}
(-38,7)*+{ \Omega_{\mathbb{W}_S(B)}^1 }="11";
(0,6.5)*+{ \mathbb{W}_S(B) \otimes_{\mathbb{W}_S(A)} \Omega_{\mathbb{W}_S(A)}^1
};
 (0,7)*+{ \phantom{\mathbb{W}_S(B) \otimes_{\mathbb{W}_S(A)}
     \Omega_{\mathbb{W}_S(A)}^1} }="12";
 (38,7)*+{ E_S^1 }="13";
(-38,-7)*+{ \Omega_{\mathbb{W}_{S/n}(B)}^1 }="21";
(0,-7.5)*+{ \mathbb{W}_{S/n}(B) \otimes_{\mathbb{W}_{S/n}(A)}
  \Omega_{\mathbb{W}_{S/n}(A)}^1 };
(0,-7)*+{ \phantom{\mathbb{W}_{S/n}(B) \otimes_{\mathbb{W}_{S/n}(A)}
  \Omega_{\mathbb{W}_{S/n}(A)}^1} }="22";
 (38,-7)*+{ E_{S/n}^1. }="23";
{ \ar "11";"12";};
{ \ar^-{\id \otimes \eta_S} "13";"12";};
{ \ar^-{F_n} "21";"11";};
{ \ar^-{F_n \otimes F_n} "22";"12";};
{ \ar^-{F_n^E} "23";"13";};
{ \ar "21";"22";};
{ \ar^-{\id \otimes \eta_{S/n}} "23";"22";};
\end{xy}$$
Here, the left-hand square commutes, since $\smash{ F_n \colon
\Omega_{\mathbb{W}(B)}^1 \to \Omega_{\mathbb{W}(B)}^1 }$ is
$F_n$-linear and a natural transformation; and right-hand square
commutes by Proposition~\ref{frobeniuscompatible}. Hence, also the
outer square commutes and this immediately implies axiom~(v); compare
Remark~\ref{dlog[-1]}~(d).

We have proved that the domains of the canonical map $\alpha$ at the
beginning of the proof form a Witt complex over $B$. Therefore, there
exists a unique map
$$\beta \colon \mathbb{W}_S\Omega_B^q \to \mathbb{W}_S(B) \otimes_{\mathbb{W}_S(A)} \mathbb{W}_S\Omega_A^q$$
of Witt complexes over $B$. The composition $\alpha \circ \beta$
is a selfmap of the initial object $\mathbb{W}_S\Omega_B^q$, and
therefore, is the identity map. The composition $\beta \circ \alpha$
is a map of Witt complexes over $B$. In particular, it is a map of
$\mathbb{W}_S(B)$-modules, and therefore, is determined by the
composition with the map of Witt complexes
$$\iota \colon \mathbb{W}_S\Omega_A^q \to
\mathbb{W}_S(B) \otimes_{\mathbb{W}_S(A)} \mathbb{W}_S\Omega_A^q$$
that takes $x$ to $[1]_S \otimes x$. But $\iota$ and $\beta \circ
\alpha \circ \iota$ both are maps of Witt complexes over $A$ with
domain the initial Witt complex over $A$. Therefore, the two maps
are equal, and hence, also $\beta \circ \alpha$ is the identity
map. This completes the proof.
\end{proof}

\section{The big de~Rham-Witt complex of the ring of integers}\label{drwofZsection}

We finally evaluate the absolute de~Rham-Witt complex of the ring of
integers. If $m$ and $n$ are positive integers, we write $(m,n)$ and
$[m,n]$ for the greatest common divisor and least common multiple of
$m$ and $n$, respectively. We define $(m,n]$ to be the unique
integer modulo $[m,n]$ such that $(m,n] \equiv 0$ modulo $m$ and
$(m,n] \equiv (m,n)$ modulo $n$, and define $\{m,n\}$ to the 
unique integer modulo $2$ that is non-zero if and only if both $m$ and
$n$ are even. We note that $(m,n] + (n,m] \equiv (m,n)$ modulo
$[m,n]$. We also remark that, by Lemma~\ref{wittcomplexrelations} and
by $d$ being a derivation, in any Witt complex, the element
$dV_n\eta_{S/n}([1]_{S/n})$ is annihilated by $n$. 

\begin{theorem}\label{bigdrwofz}The big de~Rham-Witt complex of
$\Z$ is given as follows:
$$\begin{aligned}
\mathbb{W}_S\Omega_{\Z}^0 & = \prod_{n \in S} \Z \cdot
V_n\eta_{S/n}([1]_{S/n}) \cr 
\mathbb{W}_S\Omega_{\Z}^1 & = \prod_{n \in S} \Z/n\Z \cdot
dV_n\eta_{S/n}([1]_{S/n}) \cr
\end{aligned}$$
and the groups in degrees $q \geqslant 2$ are zero. The multiplication is
given by
$$\begin{aligned}
V_m\eta_{S/m}([1]_{S/m}) \cdot V_n\eta_{S/n}([1]_{S/n}) & = (m,n) \cdot V_{[m,n]}\eta_{S/[m,n]}([1]_{S/[m,n]}) \cr
V_m\eta_{S/m}([1]_{S/m}) \cdot dV_n\eta_{S/n}([1]_{S/n}) & = (m,n] \cdot dV_{[m,n]}\eta_{S/[m,n]}([1]_{S/[m,n]})  \cr
{} & \hskip-10mm + \{m,n\} \sum_{r \geqslant 1} 2^{r-1}[m,n] \cdot
dV_{2^r[m,n]}\eta([1]_{S/2^r[m,n]}), \cr
\end{aligned}$$
and the $m$th Frobenius and Verschiebung maps are given by
$$\begin{aligned}
F_mV_n\eta_{S/n}([1]_{S/n}) & = (m,n) \cdot V_{[m,n]/m}\eta_{S/[m,n]}([1]_{S/[m,n]}) \cr
F_mdV_n\eta_{S/n}([1]_{S/n}) & =  (m,n]/m \cdot dV_{[m,n]/m}\eta_{S/[m,n]}([1]_{S/[m,n]})  \cr
{} & \hskip-15mm + \{m,n\} \sum_{r \geqslant 1} (2^{r-1}[m,n]/m) \cdot
dV_{2^r[m,n]/m}\eta_{S/2^r[m,n]}([1]_{S/2^r[m,n]}) \cr
V_m(V_n\eta_{S/mn}([1]_{S/mn})) & = V_{mn}\eta_{S/mn}([1]_{S/mn}) \cr
V_m(dV_n\eta_{S/mn}([1]_{S/mn})) & = m \cdot dV_{mn}\eta_{S/mn}([1]_{S/mn}). \cr
\end{aligned}$$
\end{theorem}

\begin{proof}We claim that there is a Witt complex
$E_S^{\boldsymbol{\cdot}}$ over $\Z$ with
$$\begin{aligned}
E_S^0 & = \prod_{n \in S} \Z \cdot
V_n\eta_{S/n}([1]_{S/n}), \cr 
E_S^1 & = \prod_{n \in S} \Z/n\Z \cdot
dV_n\eta_{S/n}([1]_{S/n}), \cr
\end{aligned}$$
with $E_S^q = 0$ for $q \geqslant 2$, and with the Witt complex
structure maps defined to be the unique additive maps satisfying the
formulas listed in the statement. For instance, the map
$\eta_S \colon \mathbb{W}_S(\Z) \to E_S^0$ is defined to be
the unique additive map that to $V_n([1]_{S/n})$ assigns
$V_n\eta_{S/n}([1]_{S/n})$; it is a ring isomorphism by
Proposition~\ref{W_S(Z)}. Granting this claim, the map $\eta_S$
extends uniquely to a map
$$\xymatrix{
{ \mathbb{W}_S\Omega_{\Z}^{\boldsymbol{\cdot}} } \ar[r]^-{\eta_S} &
{ E_S^{\boldsymbol{\cdot}} }
}$$
of Witt complex over $\Z$. It is an isomorphism in degree $q = 0$, as
noted above, and it is also an isomorphism in degree $q = 1$. For it
is clearly surjective in degree $q = 1$, and since
Lemma~\ref{wittcomplexrelations} shows that, in every Witt complex
over $\Z$,
$$ndV_n\eta_{S/n}([1]_{S/n}) = V_nd\eta_{S/n}([1]_{S/n}) = 0,$$
it is also injective in degree $q = 1$. Finally, to prove that
$\eta_S$ is an isomorphism in degrees $q \geqslant 2$, we must show
that $\mathbb{W}_S\Omega_{\Z}^2$ is zero, and to this end, it suffices
to show that, for every finite truncation set $S$ and every $n \in S$, the
element $ddV_n\eta_{S/n}([1]_{S/n})$ of $\mathbb{W}_S\Omega_{\Z}^2$
vanishes. Now, using Lemma~\ref{wittcomplexrelations} and the
projection formula, we find
$$\begin{aligned}
{} & ddV_n\eta_{S/n}([1]_{S/n}) = d\log\eta_S([-1]_S) \cdot
dV_n\eta_{S/n}([1]_{S/n}) \cr
{} & = d(d\log\eta_S([-1]_S) \cdot V_n\eta_{S/n}([1]_{S/n})) 
= dV_n(d\log\eta_{S/n}([-1]_{S/n})) \cr
{} & = \smash{ \sum_{r \geqslant 1} 2^{r-1} dV_ndV_{2^r}\eta_{S/2^rn}([1]_{S/2^rn}) }
= n ddV_{2n}\eta_{S/2n}([1]_{S/2n}), \cr
\end{aligned}$$
and since $S$ was assumed to be finite, this furnishes an induction
argument showing that $ddV_n\eta([1]_{S/n})$ is zero.

It remains to prove the claim. For notational convenience, we will
suppress the subscript $S$. We first show that the product on
$E_S^{\boldsymbol{\cdot}}$ is associative. Since $[-1]$ is a square
root of one in $\mathbb{W}_S(\Z)$ which, by Addendum~\ref{teichmuller},
is equal to $-[1] + V_2([1])$, the formula defining the
product in $E_S^{\boldsymbol{\cdot}}$ shows that, as elements of $E_S^1$,
\begin{equation}\label{equation1}
\begin{aligned}
{} & d\log\eta([-1]) = (-\eta([1]) + V_2\eta([1]))d(-\eta([1]) + V_2\eta([1])) \cr
{} & = (-\eta([1]) + V_2\eta([1]))dV_2\eta([1]) 
 = \sum_{r \geqslant 1} 2^{r-1}dV_{2^r}\eta([1]). \cr
\end{aligned}
\end{equation}
Using this formula, we find
$$\begin{aligned}
{} & V_e\eta([1]) \cdot d\log\eta([-1]) = \sum_{r \geqslant 1}
2^{r-1}V_e\eta([1])dV_{2^r}\eta([1]) \cr
{} & = \sum_{r \geqslant 1} 2^{r-1}(e,2^r] dV_{[e,2^r]}\eta([1]) +
\{e,2\}\sum_{s \geqslant 1} 2^{s-1}[e,2]dV_{2^s[e,2]}\eta([1]). \cr
\end{aligned}$$
Moreover, $2^{r-1}(m,2^r]$ is congruent to $2^{r-1}m$ modulo
$[m,2^r]$. For $2^{r-1}m$ is congruent to $0$ modulo $m$ and
to $2^{r-1}(m,2^r)$ modulo $2^r$. Hence, if $e$ is
odd, the lower left-hand summand is equal to  
$\sum_{r \geqslant 1}2^{r-1}edV_{2^re}\eta([1])$ and the lower
right-hand summand is zero, and if $e$ is even, the lower left-hand
summand is zero and the lower right-hand summand is equal to
$\sum_{s \geqslant 1}2^{s-1}edV_{2^se}\eta([1])$. So for any
positive integer $e$, we have
\begin{equation}\label{equation2}
V_e\eta([1]) \cdot d\log\eta([-1]) = \sum_{r \geqslant 1}
2^{r-1}edV_{2^re}\eta([1]).
\end{equation}
We conclude that the product in $E_S^{\boldsymbol{\cdot}}$ satisfies
\begin{equation}\label{equation3}
\begin{aligned}
V_m\eta([1]) \cdot dV_n\eta([1])
{} & = (m,n] \cdot dV_{[m,n]}\eta([1]) \cr
{} & \hskip-4mm + \{m,n\} \cdot V_{[m,n]}\eta([1]) \cdot d\log\eta([-1]), \cr
\end{aligned}
\end{equation}
where we use~(\ref{equation2}) to identify the second term on the
right-hand side. A similar calculation shows that for all positive
integers $a$ and $b$,
$$V_a\eta([1]) \cdot (V_b\eta([1]) \cdot d\log\eta([-1])) 
= (V_a\eta([1]) \cdot V_b\eta([1])) \cdot d\log\eta([-1]).$$
Using this identity, we find, on the one hand, that
$$\begin{aligned}
{} & V_l\eta([1]) \cdot (V_m([1]) \cdot dV_n\eta([1]))
= (l,[m,n]] (m,n] \cdot dV_{[l,[m,n]]}\eta([1]) \cr
{} & \hskip4mm + (\{l,[m,n]\} (m,n] + (l,[m,n]) \{m,n\})
\cdot V_{[l,[m,n]]}\eta([1]) d\log\eta([-1]), \cr
\end{aligned}$$
and, on the other hand, that
$$\begin{aligned}
{} & (V_l\eta([1]) \cdot V_m([1])) \cdot dV_n\eta([1])
= (l,m) ([l,m],n] \cdot dV_{[[l,m],n]}\eta([1]) \cr
{} & \hskip4mm + (l,m)\{[l,m],n\} \cdot V_{[[l,m],n]}\eta([1]) d\log\eta([1]). \cr
\end{aligned}$$
Here $[l,[m,n]] = [[l,m],n]$ and to prove that $(l,[m,n]] 
(m,n]$ and $(l,m)  ([l,m],n]$ are congruent modulo $[l,[m,n]]$,
we use that $[l,[m,n]]\Z$ is the kernel of the map
$$\Z \to \Z/l\Z \times \Z/m\Z \times \Z/n\Z$$
that takes $a$ to $(a + l\Z, a + m\Z, a + n\Z)$. So it will suffice to
verify that the desired congruence holds modulo $l$, $m$, and $n$,
respectively. By definition, both numbers are zero modulo $l$ and $m$,
and the congruence modulo $n$ follows from the identity
$$(l,[m,n]) \cdot (m,n) = (l,m) \cdot ([l,m],n)$$
which is readily verified by multiplying by $[l,[m,n]] = [[l,m],n]$ on
both sides. We also note that
$\{l,[m,n]\}(m,n]+(l,[m,n])\{m,n\}$ and $(l,m)\{[l,m],n\}$ are
well-defined integers modulo $2$ which are non-zero if and
only if $n$ and exactly one of $l$ and $m$ are even. This shows that the
product in $E_S^{\boldsymbol{\cdot}}$ is associative.

We proceed to verify the axioms~(i)--(v) of
Definition~\ref{bigwittcomplex}. First, we note that since the sum
$(m,n] + (n,m]$ is congruent to $(m,n)$ modulo $[m,n] = [n,m]$, we
have
$$\begin{aligned}
{} & dV_m\eta([1]) \cdot V_n\eta([1]) + V_m\eta([1]) \cdot
dV_n\eta([1]) \cr
{} & = (n,m] dV_{[n,m]}\eta([1]) +
\{n,m\}V_{[n,m]}\eta([1])d\log\eta([-1]) \cr
{} & \hskip4mm + (m,n]dV_{[m,n]}\eta([1]) +
\{m,n\}V_{[m,n]}\eta([1])d\log\eta([-1]) \cr
{} & = (m,n)dV_{[m,n]}\eta([1]) = d(V_m\eta([1]) \cdot V_n\eta([1])) \cr
\end{aligned}$$
which verifies axiom~(i). 

To verify axiom~(iv), we first show that for all positive integers $m$,
\begin{equation}\label{equation4}
F_m(d\log\eta([-1])) = d\log([-1]).
\end{equation}
It follows from formula~(\ref{equation1}) that
$$\begin{aligned}
{} & F_m(d\log\eta([-1])) = \sum_{r \geqslant 1}
2^{r-1}F_mdV_{2^r}\eta([1]) \cr
{} & = \sum_{r \geqslant 1} 2^{r-1}(m,2^r]/m \cdot
dV_{[m,2^r]/m}\eta([1]) 
+ \{m,2\} \sum_{s \geqslant 1} 2^{s-1}[m,2]/m \cdot
dV_{2^s[m,2]/m}\eta([1]) \cr
{} & = \sum_{r \geqslant 1} 2^{r-1} \cdot
dV_{[m,2^r]/m}\eta([1]) 
+ \{m,2\} \sum_{s \geqslant 1} 2^{s-1}[m,2]/m \cdot
dV_{2^s[m,2]/m}\eta([1]) \cr
\end{aligned}$$
where the second equality follows from the definition of $F_m$ and 
the last equality uses that $2^{r-1}(m,2^r]$ is congruent to
$2^{r-1}m$ modulo $[m,2^r]$. Now, if the integer $m$ is even, then
the lower the left-hand term is zero, since $[m,2^r]/m = 2^t$ with $t <
r$, and the lower right-hand term is equal to $\sum_{s \geqslant 1}
2^{s-1}dV_{2^s}\eta([1])$; and if $m$ is odd, then the lower left-hand
term is equal to $\sum_{r \geqslant 1}2^{r-1}dV_{2^r}\eta([1])$ and
the lower right-hand term is zero. Hence, using~(\ref{equation1})
again, we conclude that~(\ref{equation4}) holds. By using this
equality, we may restate the definition of the Frobenius on $E_S^1$ in
the form
\begin{equation}\label{equation5}
\begin{aligned}
F_mdV_n\eta([1])
{} & = (m,n]/m \cdot dV_{[m,n]/m}\eta([1]) \cr
{} & \hskip-4mm + \{m,n\} \cdot V_{[m,n]/m}\eta([1]) \cdot
d\log\eta([-1]), \cr
\end{aligned}
\end{equation}
and taking $n = mk$ and $y = V_k\eta([1])$, this verifies
axiom~(iv). 

We next consider axiom~(ii) which is easily verified on $E^0$. We
first show that the identity $F_lF_m = F_{lm}$ holds on
$E^1$. Using~(\ref{equation5}), we have, on the one hand, that
$$\begin{aligned}
{} & F_l(F_mdV_n\eta([1])) 
= (l,[m,n]/m]/l \cdot (m,n]/m \cdot dV_{[l,[m,n]/m]/l}\eta([1]) \cr
{} & + (\{l,[m,n]/m\} (m,n]/m + (l,[m,n]/m) \{m,n\}) \cdot
V_{[l,[m,n]/m]/l}\eta([1]) \cdot d\log\eta([-1]), \cr
\end{aligned}$$
and, on the other hand, that
$$\begin{aligned}
{} & F_{lm}dV_n\eta([1]) = (lm,n]/lm \cdot dV_{[lm,n]/lm}\eta([1]) \cr
{} & \hskip8mm + \{lm,n\} \cdot V_{[lm,n]/lm}\eta([1]) \cdot
d\log\eta([-1]). \cr 
\end{aligned}$$
Here, we have $[l,[m,n]/m]/l = [lm,n]/lm$, since both are equal to
$n/(lm,n)$, and moreover, $(lm,n]$ and $(l,[m,n]/m](m,n]$ are congruent 
modulo $[lm,n]$, since both are congruent to $0$ modulo $lm$ and
congruent to $(lm,n) = (l,[m,n]/m)(m,n)$ modulo $n$. Finally, the two
factors $\{lm,n\}$ and $\{l,[m,n]/m\}(m,n]/m + (l,[m,n]/m)\{m,n\}$ are
well-defined integers modulo $2$ which are non-zero if and only if
$lm$ and $n$ are even. Indeed, for the first factor, this is the
definition, and for the second factor, it is seen as follows. If $n$
is odd, then both summands in this factor are zero, so suppose that
$n$ is even. If $lm$ is odd, then again both summand are zero; if $l$
is odd and $m$ is even, then the first summand is zero and the second
summand is non-zero; if $l$ is even and $m$ is odd, then the first
summand is non-zero and the second summand is zero; if $l$ and $m$ are
both even and if the $2$-adic valuation of $m$ is strictly less than
that of $n$, then $[m,n]/m$ is even and $(m,n]/m$ is not divisible by $2$
modulo $[m,n]/m$, so the first summand is non-zero and the second summand is 
zero; and, finally, if $l$ and $m$ are both even and the $2$-adic
valuation of $m$ is greater than or equal to that of $n$, then
$[m,n]/m$ is odd, so the first summand is zero and the second summand
is non-zero. This completes the proof that $F_lF_m = F_{lm}$. 
The formulas $V_lV_m = V_{lm}$ and $F_mV_m = m \cdot \id$ are
readily verified, so we next show that $F_lV_m = V_mF_l$ if $l$ and
$m$ are relatively prime. To this end, we first note that by~(\ref{equation2}) and by the definition of $V_m$ on $E_S^1$, we
have
\begin{equation}\label{equation6}
V_m(V_e\eta([1]) \cdot d\log\eta([-1])) = V_{me}\eta([1]) \cdot d\log\eta([-1]),
\end{equation}
for all positive integers $m$ and $e$. Using this
identity,~(\ref{equation4}), and ~(\ref{equation5}), we find that
$$\begin{aligned}
F_lV_mdV_n\eta([1])
{} & = m(l,mn]/l \cdot dV_{[l,mn]/l}\eta([1]) \cr
{} & \hskip4mm + m\{l,mn\} \cdot V_{[l,mn]/l}\eta([1]) \cdot d\log\eta([-1]) \cr
V_mF_ldV_n\eta([1])
{} & = m(l,n]/l \cdot dV_{m[l,n]/l}\eta([1]) \cr
{} & \hskip4mm + \{l,n\} \cdot 
V_{m[l,n]/l}\eta([1]) \cdot d\log\eta([-1]). \cr 
\end{aligned}$$
But if $l$ and $m$ are relatively prime, then $[l,mn]$ and $m[l,n]$ are
equal; $m(l,mn]$ and $m(l,n]$ are congruent modulo $[l,mn] = m[l,n]$,
as both are congruent to $0$ modulo $lm$ and to $m(l,mn) = m(l,n)$
modulo $mn$; and $m\{l,mn\} = \{l,n\}$, as is easily checked. This
shows that $F_lV_m = V_mF_l$, concluding the proof of axiom~(ii).

To verify axiom~(iii), we first note that for all positive integers
$l$, $m$, and $n$,
$$V_l\eta([1]) \cdot F_mdV_n\eta([1]) = V_l(\eta([1]) \cdot
F_{lm}dV_n\eta([1])).$$
Indeed, by~(\ref{equation3}) and~(\ref{equation5}), the identity
becomes
$$\begin{aligned}
{} & (l,[m,n]/m] (m,n]/m \cdot dV_{[l,[m,n]/m]}\eta([1]) \cr
{} & + ( \{l,[m,n]/m\} (m,n]/m + (l,[m,n]/m) \{m,n\}) \cdot
V_{[l,[m,n]/m]}\eta([1]) \cdot d\log\eta([-1]) \cr
{} & = (lm,n]/lm \cdot dV_{[lm,n]/m}\eta([1]) + \{lm,n\} \cdot
V_{[lm,n]/m}\eta([1]) \cdot d\log\eta([-1]), \cr
\end{aligned}$$
and hence, the proof of the identity $F_lF_m = F_{lm}$ above shows
that the two sides are equal. Now, from this equation and from the
definition of $V_m$ on $E_S^{\boldsymbol{\cdot}}$, we find that
$$V_l(V_m\eta([1])) \cdot dV_n\eta([1])
 = V_l(V_m\eta([1]) \cdot F_ldV_n\eta([1])).$$
It remains to prove that also
$$V_l(dV_m\eta([1])) \cdot V_n\eta([1])
= V_l(dV_m\eta([1]) \cdot F_lV_n\eta([1])).$$
Using~(\ref{equation3}), the left-hand side becomes
$$l(n,lm] \cdot dV_{[n,lm]}\eta([1]) + l\{n,lm\} \cdot
V_{[n,lm]}\eta([1]) \cdot d\log\eta([-1])$$
and the right-hand sides becomes
$$\begin{aligned}
{} & l(l,n)([l,n]/l,m] \cdot dV_{l[[l,n]/l,m]}\eta([1]) \cr
{} & + (l,n)\{[l,n]/l,m\} \cdot V_{l[[l,n]/l,m]}\eta([1]) \cdot
d\log\eta([-1]). \cr
\end{aligned}$$
We have seen above that $[n,lm] = l[[l,n]/l,m]$, and $l(n,lm]$ and
$l(l,n)([l,n]/l,m]$ are congruent modulo $[n,lm]$, since both are
congruent to $0$ modulo $n$ and congruent to $l(n,lm)$ modulo
$lm$. Here we use that $([l,n]/l,m]$ is congruent to $([l,n]/l,m)$
modulo $m$ and that $(l,n)([l,n]/l,m) = (n,lm)$. Moreover, $l\{n,lm\}$
and $(l,n)\{[l,n]/l,m\}$ are well-defined integers modulo $2$ which
are non-zero if and only if $l$ is odd and $m$ and $n$ are both
even. This completes the proof that axiom~(iii) holds. Indeed, it is
clear that axiom~(iii) holds on $E^0$.

Finally, to verify axiom~(v), it suffices to consider the case $S = \N$. 
Since the formula for $[a]$ in Addendum~\ref{teichmuller} is quite
complicated, it would be a rather onerous task to verify this axiom
directly. By axiom~(i), the map $d\eta \colon \mathbb{W}(A) \to
E_{\N}^1$ is a derivation, provided that we view $E_{\N}^1$ as a
$\mathbb{W}(A)$-module via $\eta^0 = \eta \colon \mathbb{W}(A) \to
E_{\N}^0$, and hence, there is a unique $\mathbb{W}(A)$-linear map 
$\eta^1 \colon \Omega_{\mathbb{W}(A)}^1 \to E_{\N}^1$ such that
$d\eta^0 = \eta^1 d$. We will show that, for every positive integer
$m$, the diagram
$$\xymatrix{
{ \Omega_{\mathbb{W}(A)}^1 } \ar[r]^{\eta^1} \ar[d]^{F_m} &
{ E_{\N}^1 } \ar[d]^{F_m} \cr
{ \Omega_{\mathbb{W}(A)}^1 } \ar[r]^{\eta^1} &
{ E_{\N}^1 } \cr
}$$
commutes. Granting this, we find that
$$\begin{aligned}
{} & F_md\eta([a]) = F_md\eta^0([a]) 
= F_m\eta^1 d([a]) = \eta^1 F_m d([a]) 
= \eta^1([a]^{m-1}d[a]) \cr
{} & = \eta^0([a]^{m-1})\eta^1 d([a])
= \eta^0([a])^{m-1} d\eta^0([a]) 
= \eta([a])^{m-1} d\eta([a]), \cr
\end{aligned}$$
which verifies axiom~(v). Here, the first and last equalities are
identities; the third equality holds by the commutativity of the
diagram above; the fourth equality holds by
Theorem~\ref{dividedfrobenius}; and the remaining equalities hold by
the properties of the maps $\eta^0$ and $\eta^1$. It remains to prove
that the diagram above commutes, and by Theorem~\ref{dividedfrobenius}
and axiom~(i), we may assume that $m = p$ is a prime number. It
further suffices to show that for every positive integer $n$, the
image of the element $dV_n([1])$ by the two composites in the diagram
are equal. We consider three cases separately. First, if $p$ is odd
and $n = ps$ is divisible by $p$, then
$$F_p\eta^1 dV_n([1]) = F_pdV_n\eta^0([1]) = dV_s\eta^0([1]),$$
while
$$\begin{aligned}
\eta^1 F_pdV_n([1])
{} & = \eta^1( V_n([1])^{p-1}dV_n([1]) + d( \smash{ \frac{F_pV_n([1])
    - V_n([1])^p}{p} } )) \cr
{} & = \eta^1( n^{p-2}(V_n([1])dV_n([1]) + dV_s([1]) -
n^{p-2}sdV_n([1]))) \cr
{} & = \eta^1dV_s([1]) = dV_s\eta^0([1]) \cr
\end{aligned}$$
as desired. Here we used that $ndV_n([1]) = 0$. Second, if $p = 2$ and
$n = 2s$ is even, then 
$$F_2\eta^1 dV_n([1]) = F_2dV_n\eta^0([1]) = dV_s\eta^0([1]) + \sum_{r
  \geqslant 1} 2^{r-1}sdV_{2^rs}\eta^0([1])$$
while
$$\begin{aligned}
\eta^1 F_2dV_n([1])
{} & = \eta^1(dV_s([1]) + V_n([1])dV_n([1]) - sdV_n([1])) \cr
{} & = dV_s\eta^0([1]) + V_n\eta^0([1])dV_n\eta^0([1]) - sdV_n\eta^0([1]) \cr
{} & = dV_s\eta^0([1]) + \sum_{t \geqslant
  1}2^{t-1}ndV_{2^tn}\eta^0([1]) - sdV_n\eta^0([1]), 
\end{aligned}$$
so the desired equality holds in this case, too, since
$2sdV_n\eta^0([1]) = 0$. Third, if $n$ is not divisible by $p$, then
$(p,n]$ is congruent to $1 - n^{p-1}$ modulo $[p,n] = pn$, since both
are congruent to $0$ modulo $p$ and to $1$ modulo $n$, and
$\{p,n\}$ is zero. Hence, 
$$F_p\eta^1 dV_n([1]) = F_pdV_n\eta^0([1]) = \frac{1 - n^{p-1}}{p}
dV_n\eta^0([1]).$$
We wish to prove that this is equal to
$$\begin{aligned} \eta^1 F_pdV_n([1])
{} & = \eta^1(V_n([1])^{p-1}dV_n([1]) + d(\frac{F_pV_n([1])
    - V_n([1])^p}{p})) \cr
{} & = \eta^1(n^{p-2}V_n([1])dV_n([1]) + \frac{1 -
  n^{p-1}}{p} dV_n([1])) \cr
{} & = n^{p-2}V_n\eta^0([1])dV_n\eta^0([1]) + \frac{1 -
  n^{p-1}}{p} dV_n\eta^0([1]), \cr
\end{aligned}$$
or equivalently, that $n^{p-2}V_n\eta^0([1])dV_n\eta^0([1])$ is
zero. If $p$ is odd, then this holds, since $ndV_n\eta^0([1])$ is
zero; and if $p = 2$, then it holds, since
$V_n\eta^0([1])dV_n\eta^0([1])$ is zero, for $n$ odd. This completes
the proof of the claim made at the beginning of the proof, and hence,
of the theorem.
\end{proof}

\begin{addendum}\label{dgideal}Let $S$ be a finite truncation set. The
kernel of the canonical map
$$\eta_S \colon \hat{\Omega}_{\mathbb{W}_S(Z)}^{\boldsymbol{\cdot}}
\to \mathbb{W}_S\Omega_Z^{\boldsymbol{\cdot}}$$
is equal to the graded ideal generated by the following
elements~{\rm(i)--(ii)} together with their images by the derivation
$d$.  
\begin{enumerate}
\item[{\rm(i)}]For all $m,n \in S$, the element
$$\begin{aligned}
 & V_m([1]_{S/m})dV_n([1]_{S/n}) - (m,n]dV_{[m,n]}([1]_{S/[m,n]})
\cr
{} & \hskip4mm 
- \{m,n\}\sum_{r \geqslant 1}2^{r-1}[m,n]dV_{2^r[m,n]}([1]_{S/2^r[m,n]})
\cr
\end{aligned}$$
\item[{\rm(ii)}]For all $n \in S$, the element $ndV_n([1]_{S/n})$.
\end{enumerate}
\end{addendum}

\begin{proof}This follows from the proof of Theorem~\ref{bigdrwofz}. 
\end{proof}

We remark that in Addendum~\ref{dgideal}, the graded ring
$\smash{ \hat{\Omega}_{\mathbb{W}_S(\Z)}^{\boldsymbol{\cdot}} }$ may
be replaced by the graded ring
$\Omega_{\mathbb{W}_S(\Z)}^{\boldsymbol{\cdot}}$. 

\begin{acknowledgements}It is a pleasure to gratefully acknowledge the
hospitality and financial support of the following institutions that I
have visited while working on this paper: University of Rennes,
Australian National University, the Lorentz Institute at Leiden, and
the Mathematical Sciences Research Institute at Berkeley. I would also
like to acknowledge the mathematical debt of the work presented here
to my collaboration with Ib Madsen and to conversations with Jim
Borger. Finally, I would like to thank an anonymous referee for many
valuable suggestions on improving the exposition.
\end{acknowledgements}

\providecommand{\bysame}{\leavevmode\hbox to3em{\hrulefill}\thinspace}
\providecommand{\MR}{\relax\ifhmode\unskip\space\fi MR }
\providecommand{\MRhref}[2]{%
  \href{http://www.ams.org/mathscinet-getitem?mr=#1}{#2}
}
\providecommand{\href}[2]{#2}

\end{document}